\documentclass[openany]{report}
\usepackage{amssymb,amsmath,amsthm,bbm,enumerate,url,hyperref,mathabx,changebar,cancel,mathrsfs,chngcntr,relsize,mathrsfs}

\makeatletter
\newcommand\footnoteref[1]{\protected@xdef\@thefnmark{\ref{#1}}\@footnotemark}
\makeatother

\makeatletter

\newcommand\frontmatter{%
    \cleardoublepage
  \pagenumbering{roman}}

\newcommand\mainmatter{%
    \cleardoublepage
  \pagenumbering{arabic}}

\newcommand\backmatter{%
  \if@openright
    \cleardoublepage
  \else
    \clearpage
  \fi
   }

\def\FF{\mathcal{F}}
\def\AA{\mathcal{A}}
\def\BB{\mathcal{B}}
\def\CC{\mathcal{C}}
\def\GG{\mathcal{G}}
\def\HH{\mathcal{H}}
\def\EE{\mathcal{E}}
\def\II{\mathcal{I}}
\def\PP{\mathsf{P}}
\def\PPP{\mathcal{P}}
\def\QQ{\mathsf{Q}}
\def\cov{\mathrm{cov}}
\def\var{\mathrm{var}}
\def\Exp{\mathrm{Exp}}
\def\pr{\mathrm{pr}}

\def\geom{\mathrm{geom}}
\def\supp{\mathrm{supp}}

\def\DD{\mathcal{D}}
\def\LL{\mathcal{L}}
\def\UU{\mathcal{U}}
\def\VV{\mathcal{V}}
\def\ZZ{\mathbb{Z}}
\def\sgn{\mathrm{sgn}}
\def\cc{\mathsf{c}}
\def\id{\mathrm{id}}
\def\dd{\mathrm{d}}
\def\leb{\mathscr{L}}
\def\deff{$\overset{\mathrm{def}}{\Leftrightarrow}$ }
\def\ee{\mathsf{e}}
\def\ii{\mathsf{i}}
\def\ff{\mathsf{f}}
\def\MM{\mathsf{M}}
\def\LLL{\mathrm{L}}

\theoremstyle{definition}
\newtheorem{definition}{Definition}
\numberwithin{definition}{chapter}
\theoremstyle{theorem}
\newtheorem{proposition}[definition]{Proposition}

\newtheorem{theorem}[definition]{Theorem}
\newtheorem{corollary}[definition]{Corollary}

\numberwithin{equation}{chapter}
\theoremstyle{remark}
\newtheorem*{remark}{Remark}
\newtheorem*{remarks}{Remarks}

\newtheorem{example}[definition]{Example}

\counterwithout*{footnote}{chapter}

\begin{document}
\frontmatter
\begingroup
\renewcommand{\thefootnote}{$\star$} 
\title{Probability with Measure\footnote{Lecture notes. University of Ljubljana.}}
\author{\href{mailto:matija.vidmar@fmf.uni-lj.si}{Matija Vidmar}}
\maketitle
\endgroup

\begingroup
\let\clearpage\relax
\begin{quote}
\begin{center}\textbf{Vademecum} 
\end{center}

This text is meant to give the ``bare minimum'' required to perform ``basic bodily functions'' when dealing with measure-theoretic probability (from the point of view of the author).  It focuses on giving the precise ``no-hand-waving'' results with proofs, the accompanying pictures and comments that build intuition are generally omitted (because the author was too lazy to include/typeset them). Also, some of the examples are really more like exercises. More advanced/further material can be found in the literature to follow. The symbol $\subset$ is used for (non-strict) inclusion; $2^\Omega$ denotes the power set of a set $\Omega$; $\uparrow$ (resp. $\downarrow$, $\uparrow\uparrow$, $\downarrow\downarrow$) means nondecreasing (resp. nonincreasing, strictly increasing, strictly decreasing); for $n\in \mathbb{N}_0$, $[n]:=\{1,\ldots,n\}$ ($=\emptyset$ when $n=0$). 
\end{quote}

\vspace{-0.3cm}

\begin{quote}
\begin{center}
\textbf{Literature}
\end{center}
In no particular order. ($\bullet$) Pollard: A User's Guide to Measure-Theoretic Probability. ($\bullet$) Klenke: Probability Theory.  ($\bullet$)  \c{C}inlar: Probability and Stochastics.  ($\bullet$)  Billingsley: Probability and Measure.  ($\bullet$)  Schilling: Measures, Integrals and Martingales.  ($\bullet$)  Kallenberg: Foundations of Modern Probability.  ($\bullet$)  Athreya \& Lahiri: Measure Theory and Probability Theory. ($\bullet$) Williams: Probability with Martingales. ($\bullet$)  Ash \& Dol\'eans-Dade: Probability and Measure Theory. $(\bullet)$ Dudley: Real Analysis and Probability. $(\bullet)$ Meyer: Probability and Potentials. $(\bullet)$ Chung: A Course in Probability Theory. $(\bullet)$ Bhattacharya \& Waymire: Basic Course in Probability. $(\bullet)$ Varadhan: Probability Theory. ($\bullet$) R\'enyi: Foundations of Probability. $(\bullet)$ Fristedt \& Gray: A Modern Approach to Probability Theory. $(\bullet)$ Stroock: Mathematics of Probability.  $(\bullet)$  Lo\'eve: Probability Theory (vols. I and II). $(\bullet)$ Wise \& Hall: Counterexamples in Probability and Real Analysis. $(\bullet)$ Malliavin: Integration and Probability. $(\bullet$) It\^o: An Introduction to Probability Theory. $(\bullet)$ Gikhman \& Skorokhod: Introduction to the Theory of Random Processes. $(\bullet)$ Fremlin: Measure Theory (vols. 1-4). $(\bullet)$ Bogachev: Measure Theory (vols. I and II). $(\bullet)$ Saks: Theory of the Integral.
\end{quote}
\begin{quote}
\vspace{-0.4cm}
\begin{center}
\textbf{Prerequisites}
\end{center}
 Basic set theory and real analysis; a smidget of general topology does not hurt. Naive probability as a guide to intuition/source of motivation.
\end{quote}
\renewcommand\contentsname{}
\vspace{-0.5cm}
\begin{center}
{\textbf{Contents}}
\end{center}
\vspace{-2.5cm}
\footnotesize
\tableofcontents
\normalsize
\vspace{0.25cm}
\vfill 
A quote ``for the road'':  ``If you would be a real seeker after truth,   it is necessary that at least once in your life you doubt,  as far as possible,   all things.'' (Ren\'e Descartes, Principles of Philosophy.) 

\endgroup
\mainmatter
\part{Measure}

\chapter{Measurability and measures}

\section{Measurable sets}
\emph{$\sigma$-algebras; measurable spaces}

\begin{definition}
Let $\AA\subset 2^\Omega$ (i.e. $\AA\in 2^{2^\Omega}$).

$\AA$ is called closed for: 
\begin{itemize}
	\item $\cc^\Omega$ (i.e. for complements in $\Omega$) \deff for all $A\in \AA$ also $\Omega\backslash A\in \AA$; 
	\item $\cap$ (i.e. for intersections) \deff $A\cap A'\in \AA$ whenever $\{A,A'\}\subset \AA$;
	\item  $\cup$ (i.e. for unions) \deff $A\cup A'\in \AA$ whenever $\{A,A'\}\subset \AA$;
	\item $\backslash$ (i.e. for differences) \deff $A'\backslash A\in \AA$ whenever $\{A,A'\}\subset \AA$;
	\item $\sigma \cap$ (i.e. for denumerable intersections) \deff $\cap_{n\in \mathbb{N}}A_n\in\AA$ whenever $(A_n)_{n\in \mathbb{N}}$ is a sequence in $\AA$;
	\item $\sigma \cup$ (i.e. for denumerable unions) \deff $\cup_{n\in \mathbb{N}}A_n\in\AA$ whenever $(A_n)_{n\in \mathbb{N}}$ is a sequence in $\AA$.
\end{itemize}

$\AA$ is a $\sigma$-algebra (also, $\sigma$-field) on $\Omega$ \deff  $(\Omega,\AA)$ is a measurable space \deff $\emptyset\in\AA$ and $\AA$ is closed for $\cc^\Omega$ and $\sigma\cup$. If $\AA$ is a $\sigma$-field on $\Omega$, then: $A$ is $\AA$-measurable \deff $A\in \AA$; $\BB$ is a sub-$\sigma$-field of $\AA$ \deff $\BB\subset\AA$ and $\BB$ is a $\sigma$-field on $\Omega$.

  $\AA$ is an algebra on $\Omega$ \deff  $\emptyset\in\AA$ and $\AA$ is closed for $\cc^\Omega$ and $\cup$. 
\end{definition}

\begin{remarks}
Closure for $\cup$ and $\cap$ implies automatically, by induction, closure for, respectively, finite unions and finite intersections (in the obvious meaning of these qualifications). Of the closure notions, we have referenced $\Omega$ only in the notation for $\cc^\Omega$ because the other ones clearly do not depend on $\Omega$, only on $\AA$ (formally, if an $\Omega$ is not given a priori, one can take $\Omega=\cup\AA$, and then $\AA\subset 2^\Omega$). If the $\sigma$-field $\AA$ can be gathered from context, then one just says ``measurable'' in lieu of $\AA$-measurable. The reason for one's interest in $\sigma$-fields will shortly become apparent, once we have introduced the notion of a measure on a measurable space. 
\end{remarks}
\begin{example}\label{example:trivial}
	$2^\Omega$ and $\{\emptyset,\Omega\}$ are $\sigma$-fields on $\Omega$. They are called, respectively, the discrete and the trivial $\sigma$-field.
\end{example}
\begin{example}\label{example:generated-by-one}
	For $A\subset \Omega$, $\sigma_\Omega A:=\{\emptyset,A,\Omega\backslash A,\Omega\}$ is a $\sigma$-field on $\Omega$. 
\end{example}
\begin{example}\label{example:cc}
$\sigma^{\mathrm{ccc}}_\Omega:=\{A\in 2^\Omega:A\text{ countable or }\Omega\backslash A\text{ countable}\}$ is a $\sigma$-field on $\Omega$, called the countable--co-countable $\sigma$-field.\footnote{Countable means, here and throughout, finite or countably infinite.} Of course $\sigma^{\mathrm{ccc}}_\Omega=2^\Omega$ unless $\Omega$ is not countable.
\end{example}
\begin{example}\label{example:partition}
	If $\PPP$ is a partition of $\Omega$ (meaning: $\cup \PPP=\Omega$, $\{A,A'\}\subset \PPP$ and $A\ne A'$ implies $A\cap A'=\emptyset$, $\emptyset\notin \PPP$), then $\sigma\PPP:=\{\cup P:P\text{ countable or co-countable subset of }\PPP\}$ is a $\sigma$-field on $\Omega$. If $\emptyset\ne A\subsetneq\Omega$, then $\sigma\{A,\Omega\backslash A\} =\sigma_\Omega(A)$. Besides, $\sigma\{\{\omega\}:\omega\in \Omega\}=\sigma_\Omega^{\mathrm{ccc}}$.
\end{example}
\begin{remark}
In the previous example, mainly the case when $\PPP$ is finite or at most denumerable is useful. Then $\sigma\PPP=\{\cup P:P\in 2^\PPP\}$.
\end{remark}
\begin{proposition}
	Let $\AA\subset 2^\Omega$ be closed for $\cc^\Omega$ and let $\emptyset\in \AA$. Then $\AA$ is a $\sigma$-algebra on $\Omega$ iff $\AA$ is closed for $\sigma\cap$, in which case $\Omega\in \AA$ and $\AA$ is closed for $\cap$, $\cup$ and $\backslash$. 
\end{proposition}
\begin{proof}
It follows from: de Morgan's laws; the relation $\Omega=\Omega\backslash \emptyset$; finally from the facts that $A'\backslash A=A'\cap (\Omega\backslash A)$, and that $(A,A',\Omega,\Omega,\ldots)$ and $(A,A',\emptyset, \emptyset,\ldots)$  are sequences in $\AA$ whenever $\{A,A'\}\subset \AA$.
\end{proof}

\section{Measures}
\emph{measure as a countably additive nonnegative set-function null at $\emptyset$; first properties of measures}

\begin{definition}\label{definition:measures}
	Let $(\Omega,\FF)$ be a measurable space and $\mu:\FF\to [0,\infty]$. 
	
	$\mu$ is a measure on $(\Omega,\FF)$ (also, on $\FF$)\footnote{There will never be any ambiguity as to which of the two we intend, since an ordered pair cannot at the same time be a $\sigma$-algebra (on any set).} \deff $\mu(\emptyset)=0$ and $\mu$ is countably additive: $\mu(\cup_{n\in \mathbb{N}}A_n)=\sum_{n\in \mathbb{N}}\mu(A_n)$ whenever $(A_n)_{n\in \mathbb{N}}$ is a sequence in $\FF$ of pairwise disjoint sets.

	A measure $\mu$ on $(\Omega,\FF)$: 
	\begin{itemize}
	\item 
	is finite \deff $\mu(\Omega)<\infty$;
	\item  	 is a probability measure (also, a probability (law)) \deff $\mu(\Omega)=1$;
	\item is $\sigma$-finite \deff there is a sequence $(A_n)_{n\in \mathbb{N}}$ in $\FF$ with $\Omega= \cup_{m\in \mathbb{N}}A_m$ and  $\mu(A_n)<\infty$ for all $n\in \mathbb{N}$; 
	   \item 
	   has $\mu(\Omega)$ for its (total) mass. 
	   \end{itemize}
		
	$(\Omega,\FF,\mu)$ is a measure space \deff $\mu$ is a measure on $(\Omega,\FF)$.
	
	If $(\Omega,\FF,\mu)$ is a measure space, then for  $A\in \FF$: 
	\begin{itemize}
	\item $A$ is $\mu$-negligible \deff $\mu(A)=0$; 
	\item $A$ is $\mu$-trivial \deff $A$ or $\Omega\backslash A$ is $\mu$-negligible;
	\item  $A$ carries $\mu$ \deff $\Omega\backslash A$ is $\mu$-negligible. 
	\end{itemize}
	If in addition one has a property (predicate) $P(\omega)$ in $\omega\in A$, then: 
	\begin{itemize}
	\item $P(\omega)$ holds $\mu$-almost everywhere (abbreviated to $\mu$-a.e. or a.e.-$\mu$) in $\omega\in A$ \deff $A_{\neg P}:=\{\omega\in A:\neg P(\omega)\}\in \FF$ and $\mu(A_{\neg P})=0$; 
	\item $P(\omega)$ holds $\mu$-almost surely (abbreviated $\mu$-a.s. or a.s.-$\mu$) in $\omega\in A$ \deff the preceding and in addition $\mu$ is a probability measure. 
	\end{itemize}
	When $A=\Omega$, then:  ``$\in A$'' can be omitted in the preceding (i.e. one can just say ``for $\mu$-a.e. $\omega$'' in lieu of ``for $\mu$-a.e. $\omega\in \Omega$'', etc.). 
	
	Given a map $f:U\to W$ and a property (predicate) $P(v)$ in $v\in W$ we often write: 
	\begin{itemize}
	\item $\{P(f)\}:=\{u\in U:P(f(u))\text{ is true}\}$;
	\item ``$P(f)$ holds'' to mean ``$P(f(u))$ holds true for all $u\in U$'', 
	\end{itemize}
	and correspondingly, when $\mu$ is a measure on a $\sigma$-field of $U$, 
	\begin{itemize}
	\item ``$P(f)$ holds true $\mu$-a.e.'' to mean ``$P(f(u))$ holds true for $\mu$-a.e. $u$'' etc. 
	\end{itemize} (for instance, if $W=\mathbb{R}$, $f\geq 2$ a.e.-$\mu$ means $f(u)\geq 2$ for $\mu$-a.e. $u$, while $\{f\geq 2\}=\{u\in U:f(u)\geq 2\}$). Similarly, if, in addition, $f':U\to W'$ and we have a property $P(w,w')$ in $w\in W$ and $w'\in W'$, then we interpret 
	\begin{itemize}
	\item $\{P(f,f')\}$ and ``$P(f,f')$ holds true (a.e.-$\mu$)''  by considering $(f,f'):U\to W\times W'$;
	\end{itemize} likewise for any family $f_\lambda:U\to W_\lambda$, $\lambda\in \Lambda$. 	Besides, if $E(w)$ is some expression (term) involving $w\in W$ and if $f$ is as above, then
	\begin{itemize}
	 \item $E(f)$ means the map $(U\ni u\mapsto E(f(u)))$; 
	 \end{itemize}
	 if $E(w,w')$ is some expression involving $w\in W$ and $w'\in W'$ and $f,f'$ are as above, then considering again $(f,f'):U\to W\times W'$,
	 \begin{itemize}
	  \item $E(f,f')$ means the map $(U\ni u\mapsto E(f(u),f'(u)))$; 
	  \end{itemize}
	  analogously for any family $f_\lambda:U\to W_\lambda$, $\lambda\in \Lambda$.
\end{definition}

\begin{remarks}
The reason for defining measures on $\sigma$-fields, rather than some other types of subsets of $2^\Omega$ is roughly the following: this structure is on the one hand restrictive enough to accommodate naturally the nice properties that one desires (countable additivity; closure for complements, differences etc.; one can ``measure'' the empty set and the whole space), yet rich enough for non-trivial interesting measures to exist.  From the point of view that measures ``measure'' the sizes of sets (lengths, areas, volumes, probabilities etc.) or probabilities of events the additivity is quite natural (at least finite, if not countable), as is the requirement that the measure of $\emptyset$ be zero. The insistance on countable (vis-\`a-vis just finite) additivity is a luxury that one can (usually) afford, the upshot being that measures then enjoy many nice properties which they would otherwise not.\footnote{\underline{\href{https://www.jstor.org/stable/2319450}{Here}} is an interesting example of an argument for the ``naturality'' of countable additivity. See also \underline{\href{http://www.jehps.net/juin2010/Bingham.pdf}{here}} and  \underline{\href{https://onlinelibrary.wiley.com/doi/abs/10.1002/tht3.60}{here}} for a discussion of (finite vs.) countable additivity. (All are in the context of probabilities.)}   If, ceteris paribus, in the definition of measure we take $\mu$ as mapping into $(-\infty,\infty]$ or $[-\infty,\infty)$ (resp. into $\mathbb{C}$) we get a so-called charge/signed measure (resp. complex measure), but we shall not dwell on these at all. There is developed and useful also a theory of measure with values in a Banach space; one then speaks of a vector-valued measure. Yet another note-worthy object is a projection-valued measure (its values are projections on a Hilbert space). However, the base case of Definition~\ref{definition:measures} is fundamental. A notion much akin to $\sigma$-finiteness is that of s-finiteness, a measure being s-finite iff it is a countable sum of finite measures.  
\end{remarks}
\begin{example}
	The zero measure on $\FF$, i.e. the map $(\FF\ni A\mapsto 0)$, is always a measure on any given $\sigma$-field $\FF$. 
\end{example}
\begin{example}
	If we define $c_\Omega:2^\Omega\to [0,\infty]$ by putting $c_\Omega(A):=\vert A\vert$ if $A$ is a finite subset of $\Omega$ and $c_\Omega(A):= \infty$ if it is an infinite subset of $\Omega$, then $c_\Omega$ is the so-called counting measure on $\Omega$. When $\Omega$ is finite and non-empty, then $c_\Omega/\vert \Omega\vert$ is a probability measure (the ``classical'' (uniform) probability on $\Omega$).
\end{example}

\begin{example}
	If we define $\delta_x:2^\Omega\to [0,\infty]$ for a fixed $x\in \Omega$ by putting,   for $A\in 2^\Omega$, $\delta_x(A):=0$ if $x\notin A$ and $\delta_x(A):=1$ if $x \in A$, then $\delta_x$ is the so-called Dirac measure at $x$. Any subset of $\Omega\backslash \{x\}$ is $\delta_x$-negligible. For a map $f$ with domain $\Omega$, we have that $f(y)=f(x)$ for $\delta_x$-a.e. $y\in \Omega$ (since $\{z\in \Omega:f(z)\ne f(x)\}\subset \Omega\backslash \{x\}$), and in fact $\delta_x$-a.s. in $y\in \Omega$ (because $\delta_x$ is a probability measure); more succinctly, $f=f(x)$ a.s.-$\delta_x$.
\end{example}
\begin{proposition}
	Let $\mu$ be a measure on a measurable space $(\Omega,\FF)$. Then:
	\begin{enumerate}[(i)]
		\item\label{measures:i} $\mu$ is additive: $\mu(A\cup B)=\mu(A)+\mu(B)$ whenever $\{A,B\}\subset \FF$ and $A\cap B=\emptyset$.
		\item\label{measures:ii} $\mu$ is monotone: $\mu(A)\leq \mu(B)$ whenever $\{A,B\}\subset \FF$ and $A\subset B$.
		\item\label{measures:iii} $\mu$ is continuous from below: $\mu(\cup_{n\in \mathbb{N}}A_n)=\uparrow\!\!\text{-}\lim_{n\to\infty}\mu(A_n)$ whenever $(A_n)_{n\in \mathbb{N}}$ is a nondecreasing (w.r.t. inclusion) sequence in $\FF$.
		\item\label{measures:iii'} $\mu$ is countably subadditive: $\mu(\cup_{n\in \mathbb{N}}A_n)\leq \sum_{n\in\mathbb{N}}\mu(A_n)$ whenever $(A_n)_{n\in \mathbb{N}}$ is a sequence in $\FF$.
		\item\label{measures:iv} Suppose $\mu$ is finite. $\mu(\Omega\backslash A)=\mu(\Omega)-\mu(A)$ for all $A\in \FF$. Further, $\mu$ is continuous from above: $\mu(\cap_{n\in \mathbb{N}}A_n)=\downarrow\!\!\text{-}\lim_{n\to\infty}\mu(A_n)$ whenever $(A_n)_{n\in \mathbb{N}}$ is a nonincreasing (w.r.t. inclusion) sequence in $\FF$. 
		\item\label{measures:v} For  $A\in \FF$: put $\FF\vert_A:=\{B\cap A:B\in \FF\}$; then $\mu_A:=\mu\vert_{\FF\vert_A}$ is a measure on $\FF\vert_A$.
	\end{enumerate}
\end{proposition}
\begin{remarks}
	Of course \eqref{measures:i} generalizes by induction at once to ``finite additivity''. Besides, because of \eqref{measures:v}, in \eqref{measures:iv} one need not assume that $\mu$ is finite, and still continuity from above prevails as long as $\mu(A_n)<\infty$ for some $n\in \mathbb{N}$, while $\mu(B\backslash A)=\mu(B)-\mu(A)$ whenever $A\subset B$ are both from $\FF$ and $\mu(B)<\infty$ (it is in fact true even if merely $\mu(A)<\infty$). More generally, because of \eqref{measures:v}, if all the sets in question are subsets of a measurable set of finite measure, one can use results for finite measures even if the measure is not finite to begin with, simply by restricting the measure to said set of finite measure.
\end{remarks}
\begin{definition}
$\mu_A:=\mu\vert_{\FF\vert_A}$ from \eqref{measures:v}  is	 called the restriction of $\mu$ to $A$. 
\end{definition}
\begin{proof}
	\eqref{measures:i}. $(A,B,\emptyset,\emptyset,\ldots)$ is a sequence in $\FF$ of pairwise disjoint sets. \eqref{measures:ii}. $B=A\cup (B\backslash A)$; apply \eqref{measures:i}. \eqref{measures:iii}. $(A_1,A_2\backslash A_1,A_3\backslash A_2,\ldots)$ is a sequence in $\FF$ of parwise disjoint sets with union $\cup_{n\in \mathbb{N}}A_n$. Apply countable additivity followed by finite additivity. \eqref{measures:iii'}. $(A_1,A_2\backslash A_1,A_3\backslash (A_1\cup A_2),\ldots)$ is a sequence in $\FF$ of parwise disjoint sets with union $\cup_{n\in \mathbb{N}}A_n$. Apply countable additivity followed by monotonicity.  \eqref{measures:iv}. The first statement follows from \eqref{measures:i} upon taking for $A$, $A$, and for $B$, $\Omega\backslash A$. The second statement then follows from \eqref{measures:iii} by applying it to $(\Omega\backslash A_n)_{n\in \mathbb{N}}$ in lieu of $(A_n)_{n\in \mathbb{N}}$. \eqref{measures:v}. One checks that $\FF\vert_A$ is a $\sigma$-field on $A$ (we will also see this independently in Corollary~\ref{proposition:restrictions-of-fields}) equal to $2^A\cap  \FF$; then the claim is immediate. 
\end{proof} 

\begin{example}[Borel-Cantelli I]\label{Borel-Cantelli-I}
Let $(X,\FF,\mu)$ be a measure space and $(A_n)_{n\in \mathbb{N}}$ a sequence in $\FF$ satisfying $\sum_{n\in \mathbb{N}}\mu(A_n)<\infty$. Then $\mu(\limsup_{n\to\infty}A_n)=0$.
\end{example}

\begin{example}
If $\PP$ is a probability on $(\Omega,\FF)$, then $\PP^{-1}(\{0,1\})$ is a sub-$\sigma$-field of $\FF$, the so-called $\PP$-trivial $\sigma$-field.
\end{example}
\section{Measurable maps and generated $\sigma$-fields}
\emph{generated $\sigma$-fields;  initial and final structures; measurable maps; traces of $\sigma$-fields; compositions and restrictions of measurable maps}

\begin{definition}\label{definition:generated-1}
Let $\AA\subset 2^\Omega$; then 
\begin{equation*}
\sigma_\Omega(\AA):=\cap\{\FF\in 2^{2^\Omega}:\FF\text{ a $\sigma$-field on }\Omega\text{ and }\AA\subset \FF\}
\end{equation*}
is called the $\sigma$-field generated on $\Omega$ by $\AA$ [remark that $2^\Omega$ is certainly a $\sigma$-field on $\Omega$ that contains $\AA$ so the intersection is of a non-empty family].  For two $\sigma$-fields $\BB_1$ and $\BB_2$ on $\Omega$ set $\BB_1\lor \BB_2:=\sigma_\Omega(\BB_1\cup \BB_2)$ and call it the join of $\BB_1$ and $\BB_2$. More generally, for a family $(\BB_\lambda)_{\lambda\in \Lambda}$ of $\sigma$-fields on $\Omega$ we set $\lor_{\lambda\in \Lambda}\BB_\lambda:=\sigma_\Omega(\cup_{\lambda\in\Lambda}\BB_\lambda)$ and call it the join of said family.
\end{definition}

\begin{remarks}
One reason why generated $\sigma$-fields are so important in measure theory is because only seldom can we explicitly list/name/describe all the elements of a $\sigma$-field, which is to our liking, but we often can explicitly provide its generators, viz. the elements of $\AA$ in $\sigma_\Omega(\AA)$.  A closely allied notion is that of the $\sigma$-field generated by a map.
\end{remarks}

\begin{definition}\label{definition:generated-2}
	Let $f:\Omega\to \Omega'$. 
	
	Given a $\sigma$-field $\FF'$ on $\Omega'$ we define 
	\begin{equation*}
		\sigma^{\FF'}(f):=f^{-1}(\FF'):=\{f^{-1}(A'):A'\in \FF'\},
	\end{equation*}
the initial structure for $f$ relative to $\FF'$ (or the $\sigma$-field generated by $f$ relative to $\FF'$, also the pull-back of $\FF'$ by $f$). 

Given a $\sigma$-field $\FF$ on $\Omega$ we define 
\begin{equation*}
	\sigma_{\FF}^{\Omega'}(f):=\{A'\in 2^{\Omega'}:f^{-1}(A')\in \FF\},
\end{equation*}
the final structure on $\Omega'$ for $f$ relative to $\FF$ (or the push-forward to $\Omega'$ of $\FF$ by $f$).

Given a $\sigma$-field $\FF'$ on $\Omega'$  and a $\sigma$-field $\FF$ on $\Omega$, we say that $f$ is $\FF/\FF'$-measurable \deff $f^{-1}(A')\in \FF$ for all $A'\in \FF'$.
	\end{definition}
\begin{remarks}
	In the notations $\sigma_\Omega(\AA)$ and $\sigma^{\FF'}(f)$,	one tends to variously omit $\Omega$ or $\FF'$  if they can be gathered from context, writing simply $\sigma(\AA)$ and $\sigma(f)$. In particular if the range of $f$  is countable and no $\FF'$  (resp. or even no $\Omega'$) is provided, then one takes (resp. $\Omega'= \text{range of }f$ and) $\FF'=2^{\Omega'}$. Of the two objects pertaining to $f$ introduced in Definition~\ref{definition:generated-2}, $\sigma^{\FF'}(f)$  and $\sigma_{\FF}^{\Omega'}(f)$, the first is by far the more important one. (One instance in which one does meet final structures naturally is with quotient maps of equivalence relations.) The notation $f^{-1}(\FF')$ is more suggestive and succinct than $\sigma^{\FF'}(f)$ but can in principle be confused with the preimage of $\FF'$ under $f$, however it seems unlikely such confusion would ever arise in practice. The notion of a measurable map is key to subsequent developments; it is to measure theory what a continuous map is to topology/analysis. If $\FF$ and $\FF'$ can be gathered from context, then one just says ``measurable'' in lieu of $\FF/\FF'$-measurable. In probability theory (resp. temporally indexed families of) sub-$\sigma$-fields are used to model (resp. the flow of) information; the notion of generated $\sigma$-fields is of paramount importance in such contexts. 
\end{remarks}

\begin{definition}
Given $\sigma$-fields $\FF$ on $\Omega$ and $\FF'$ on $\Omega'$, we define $\FF/\FF':=\{g\in {\Omega'}^{\Omega}:g\text{ is }\FF/\FF'\text{-measurable}\}$.
\end{definition}

\begin{example}
A constant function is always measurable, no matter what the $\sigma$-fields. For any $\sigma$-field $\FF$ on $ \Omega$, $\mathrm{ id}_\Omega\in \FF/\FF$.
	\end{example}

\begin{definition}
	For $A\subset \Omega$  define $\mathbbm{1}_{A_\Omega}:\Omega\to \{0,1\}$ by putting $$\mathbbm{1}_{A_\Omega}(x):=\begin{cases} 1,&x\in A\\ 0,& x\in \Omega\backslash A.\end{cases}.$$ This is the indicator function of $A$ with underlying space $\Omega$. In deference to standard practice  we shall usually just write $\mathbbm{1}_A$ in lieu of $\mathbbm{1}_{A_\Omega}$ assuming $\Omega$ can be gathered from context. 
\end{definition}

\begin{example}\label{example:indicator-measurable}
	Let $A\subset \Omega$. Then $\sigma^{2^{\{0,1\}}}(\mathbbm{1}_A)=\sigma_\Omega A$ in the notation of Example~\ref{example:generated-by-one}. If further $\FF$ is a $\sigma$-field on $\Omega$, then $\mathbbm{1}_A\in \FF/2^{\{0,1\}}$ iff $A\in \FF$. 
\end{example}

\begin{proposition}\label{proposition:compositions}
	Let $\FF$, $\GG$, $\HH$ be $\sigma$-fields (each on their own set). Suppose $f\in \FF/\GG$ and $g\in \GG/\HH$. Then $g\circ f\in \FF/\HH$.
\end{proposition}
\begin{remark}
	In words, compositions of measurable maps are measurable. 
\end{remark}
\begin{proof}
	$(g\circ f)^{-1}(H)=f^{-1}(g^{-1}(H))$ for $H\in\HH$. 
\end{proof}

\begin{proposition}\label{proposition:maps}
	Let $f:\Omega\to \Omega'$. 
	\begin{enumerate}[(i)]
		\item\label{maps:i} Given a $\sigma$-field $\FF'$ on $\Omega'$, $\sigma^{\FF'}(f)$ is a $\sigma$-field on $\Omega$; it is the smallest (w.r.t. inclusion) $\sigma$-field $\GG$ on  $\Omega$ such that $f\in \GG/\FF'$. 
		\item\label{maps:ii} Given a $\sigma$-field $\FF$ on $\Omega$, $\sigma_{\FF}^{\Omega'}(f)$ is a $\sigma$-field on $\Omega'$; it is the largest (w.r.t. inclusion) $\sigma$-field $\GG'$ on  $\Omega'$ such that $f\in \FF/\GG'$. 
		\item\label{maps:iii} Given a $\sigma$-field $\FF'$ on $\Omega'$  and a $\sigma$-field $\FF$ on $\Omega$, then 	$f\in \FF/\FF'\Leftrightarrow \sigma^{\FF'}(f)\subset \FF \Leftrightarrow \sigma^{\Omega'}_\FF(f)\supset \FF'$.
				\item\label{maps:iv} Let $\AA'\subset 2^{\Omega'}$. $\sigma_{\Omega'}(\AA')$ is the smallest (w.r.t. inclusion) $\sigma$-field on $\Omega'$ that has $\AA'$ as its subset. Given a $\sigma$-field $\FF$ on $\Omega$, then $f\in \FF/\sigma_{\Omega'}(\AA')$ iff ($f^{-1}(A')\in \FF$ for all $A'\in \AA'$). In particular, $\sigma^{\sigma_{\Omega'}(\AA')}(f)=\sigma_\Omega(\{f^{-1}(A'):A'\in \AA'\})$. 
			\end{enumerate}
	\end{proposition}
\begin{remarks}
In plain tongue \eqref{maps:iv} tells us that it is enough to prove the measurability property on a set of generators. Another way of writing  $\sigma^{\sigma_{\Omega'}(\AA')}(f)=\sigma_\Omega(\{f^{-1}(A'):A'\in \AA'\})$ is as $f^{-1}(\sigma_{\Omega'}(\AA'))=\sigma_\Omega(f^{-1}(\AA'))$, which we may read as the operations of taking pre-images and generated $\sigma$-fields ``commuting''.
%
%
\end{remarks}%
\begin{proof}
\eqref{maps:i} and \eqref{maps:ii} follow from the fact that $f^{-1}$ (the taking of preimages) ``commutes'' with unions and complementation (and from $f^{-1}(\emptyset)=\emptyset$). \eqref{maps:iii} is direct from the definitions. Let us prove  \eqref{maps:iv}. The first statement follows from the fact that the intersection of $\sigma$-fields on $\Omega$ is again a $\sigma$-field on $\Omega$. As for the equivalence, the condition is clearly necessary. It is sufficient because it entails that  $\sigma^{\Omega'}_{\FF}(f)\supset \AA'$, so that $\sigma^{\Omega'}_{\FF}(f)\supset \sigma_{\Omega'}(\AA')$, which in view of \eqref{maps:iii}  renders $f\in \FF/\sigma_{\Omega'}(\AA')$. The final statement of \eqref{maps:iv} follows on taking $\FF=\sigma_\Omega(\{f^{-1}(A'):A'\in \AA'\})$ in the second statement of this same item, which gives $f^{-1}(\sigma_{\Omega'}(\AA'))\subset \sigma_\Omega(f^{-1}(\AA'))$; and from \eqref{maps:i}, which, through the evident inclusion $f^{-1}(\AA')\subset f^{-1}(\sigma_{\Omega'}(\AA'))$ delivers $f^{-1}(\sigma_{\Omega'}(\AA'))\supset \sigma_\Omega(f^{-1}(\AA'))$.
	\end{proof}
%
%
\begin{definition}
	We put
	\begin{equation*}
	\AA\vert_A:=\{A'\cap A:A'\in\AA\}
	\end{equation*}
	for the trace of $\AA$ on $A$.\footnote{Not to be confused with the restriction of a function $\AA$ to a subset $A$ of its domain! Let us agree that if it is given a priori that $\AA$ is a function and $A$ a subset of its domain, then the latter (not the trace) meaning of $\AA\vert_A$ prevails (this is still wrong, but hopefully enough to prevent confusion). Of course this is not the only notational shenanigan that we have hitherto seen, another screaming example being $f^{-1}(x)$ which can refer to the preimage of $x$ under $f$ or the pull-back of $x$ under $f$, yet another $2^\Omega$ (is it the power set of $\Omega$, or $\{0,1\}^\Omega$ -- the maps from $\Omega$ into $2=\{0,1\}$?); somehow we trust that it can be gathered from context which of these is intended.} 
\end{definition}

\begin{example}
	If $\FF$ is closed under $\cap$ (in particular if it is a $\sigma$-algebra) and $A\in \FF$, then plainly $\FF\vert_A=\FF\cap 2^A$. 
\end{example}

\begin{corollary}\label{proposition:restrictions-of-fields}
	Let $\AA\subset 2^\Omega$. If further $A\subset \Omega$, then \begin{equation}\label{equation:restrictions}
	\sigma_\Omega(\AA)\vert_A=\sigma_A(\AA\vert_A);
	\end{equation}
	in particular, if $\AA$ is a $\sigma$-field on $\Omega$, then $\AA\vert_A$ is a $\sigma$-field on $A$.

	Besides, $\sigma_\Omega(\AA)=\cup \{\sigma_\Omega(\BB):\BB\text{ countable}\subset \AA\}$. 
	 
\end{corollary}
\begin{proof}
By Proposition~\ref{proposition:maps}\eqref{maps:i}, $\sigma_\Omega(\AA)\vert_A=\sigma^{\sigma_\Omega(\AA)}(\mathrm{id}_A)$ is a $\sigma$-field on $A$ that contains $\AA\vert_A$; therefore $\sigma_\Omega(\AA)\vert_A\supset \sigma_A(\AA\vert_A)$. 
By Proposition~\ref{proposition:maps}\eqref{maps:ii}, $\CC:=\{C\in 2^\Omega:C\cap A\in \sigma_A(\AA\vert_A)\}=\sigma^\Omega_{\sigma_A(\AA\vert_A)}(\id_A)$ is a $\sigma$-field on $\Omega$ that contains $\AA$; therefore $\sigma_\Omega(\AA)\subset \CC$, so $\sigma_\Omega(\AA)\vert_A\subset \sigma_A(\AA\vert_A)$. 

  The final observation follows from the fact that countable unions of countable sets are countable (coupled with more elementary considerations).  
\end{proof}
\begin{example}
		In the context of Example~\ref{example:generated-by-one}, $\sigma_\Omega A=\sigma_\Omega(\{A\})$.
\end{example}
\begin{example}
	In the context of Example~\ref{example:cc}, $\sigma^{\mathrm{ccc}}_\Omega=\sigma_\Omega(\{\{\omega\}:\omega\in \Omega\})$.
	\end{example}
\begin{example}
	In the context of  Example~\ref{example:partition},  $\sigma\PPP=\sigma_\Omega(\PPP)$. 
\end{example}


\begin{remark}
	In general, how to identify $\sigma_\Omega(\AA)$? Answer: start with $\AA$ and add anything that has to be in $\sigma_\Omega(\AA)$ because of the closure properties of $\sigma$-fields: all the complements, countable unions, $\emptyset$ and $\Omega$, countable unions of those, complements of those etc. etc. Stop when you ``feel'' it is enough. Then (if your feeling was right) prove that what you have is a $\sigma$-algebra. This procedure can be made rigorous, but in general the adding of the complements, and the taking of countable unions has to be effected the first uncountable ordinal-many times, in order to ensure that a $\sigma$-field is produced\footnote{See e.g. \underline{\href{https://math.dartmouth.edu/archive/m103f08/public_html/borel-sets-soln.pdf}{this link}}.}.
\end{remark}
\begin{example}
	Let $\{E,F\}\subset 2^\Omega$. Then  $\sigma_\Omega(\{E,F\})$ must include what we might call ``the induced partition'' $$\mathcal{P}:=\{E\cap F,E\backslash F,F\backslash E,\Omega\backslash (E\cup F)\}\backslash \{\emptyset\}.$$ Therefore (since clearly $\{E,F\}\subset \sigma_\Omega(\mathcal{P})$) it follows that $$\sigma_\Omega(\{E,F\})=\sigma_\Omega(\mathcal{P})=\sigma\mathcal{P},$$ which in turn is identified explicitly, as we have seen (Example~\ref{example:partition}). This can easily be generalized to identify $\sigma_\Omega(\AA)$ whenever $\AA$ is a finite (but not beyond) subset of $2^\Omega$.
	 
	Here is how to do it. For $A\in \AA$ and $i\in \{0,1\}$ let $A^0:=A$ and $A^1:=\Omega\backslash A$. Then (because $\AA$ is finite) $\sigma_\Omega(\AA)$ must include the (finite, which does not really matter)  induced partition $\mathcal{P}:=\{\cap_{A\in \AA}A^{F(A)}:F\in \{0,1\}^\AA\}\backslash \{\emptyset\}$  (we dissect $\Omega$ using the elements of $\AA$ and their complements as finely as possible; the empty set is taken away because [by definition] partitions do not include the empty set). At this point $\AA$ could have been denumerable and all said would be true except that $\mathcal{P}$ would not be finite (not even countable). On the other hand $\sigma_\Omega(\mathcal{P})\supset \AA$ (here the finiteness of $\AA$ is essential). 	Therefore $\sigma_\Omega(\AA)=\sigma_\Omega(\mathcal{P})=\sigma\mathcal{P}$. 
	 
\end{example}
\begin{proposition}\label{proposition:restrictions}
	Let $f:\Omega\to \Omega'$ and let  there be given a $\sigma$-field $\FF$ on $\Omega$ and a $\sigma$-field $\FF'$ on $\Omega'$. 
	\begin{enumerate}[(a)]
		\item\label{restrict:a} If $A'\subset\Omega'$ is such that $f:\Omega\to A'$, then $f\in \FF/\FF'$ iff $f\in \FF/(\FF'\vert_{A'})$. 
		\item\label{restrict:b} For $A\subset \Omega$, if $f\in \FF/\FF'$, then $f\vert_A\in (\FF\vert_A)/\FF'$.
		\item\label{restrict:c} If $(A_i)_{i\in \mathbb{N}}$ is a sequence in $\FF$ with $\Omega=\cup_{i\in \mathbb{N}}A_i$ then: $f\in \FF/\FF'$ $\Leftrightarrow$ ($f\vert_{A_i}\in \FF\vert_{A_i}/\FF'$ for all $i\in \mathbb{N}$). 
	\end{enumerate}
\end{proposition}
\begin{remarks}
In common language \eqref{restrict:a} and \eqref{restrict:b} mean that measurability behaves well under restrictions (the notion of the trace is aligned well with that of the restriction of a map). Item \eqref{restrict:c} tells us that measurability can be checked ``locally''; some of the $A_i$, $i\in \mathbb{N}$, may be empty, so a particular case occurs when we can cover $\Omega$ with finitely many measurable sets and check measurability on those.
 
Since for a $\sigma$-field $\AA$ on a set $A$ and $B\supset A$, $\BB:=\{C\in 2^B:C\cap A\in \AA\}=\sigma_\AA^B(\mathrm{id}_A)$ is a $\sigma$-field on $B$ satisfying $\BB\vert_A=\AA$, a $\GG/\AA$-measurable map $g$ can always be viewed also as a $\GG/\BB$ measurable map. 
 
\end{remarks}
\begin{proof}
	\eqref{restrict:a}.  $f^{-1}(F')=f^{-1}(F'\cap A')$ for $F'\in \FF'$. \eqref{restrict:b}.  $(f\vert_A)^{-1}(F')=A\cap f^{-1}(F')$ for $F'\in \FF'$.  \eqref{restrict:c}. Necessity is by \eqref{restrict:b}. Sufficiency: $f^{-1}(F')=\cup_{i\in \mathbb{N}}\underbrace{(f\vert_{A_i})^{-1}(F')}_{\in \FF\vert_{A_i}\subset\FF}$ for $F'\in \FF'$. 
\end{proof}

\section{Borel sets of the extended real line and Borel measurability of numerical functions}
\emph{the Borel $\sigma$-field on $[-\infty,\infty]$;  $[-\infty,\infty]$-valued Borel measurable maps}

\begin{definition}\label{definition:conventions}
$[-\infty,\infty]:=\mathbb{R}\cup \{-\infty,\infty\}:=$ the extended real line. 

The orderings $\leq$ and $<$ are extended to $[-\infty,\infty]$ from $\mathbb{R}$ in the natural way: $-\infty \leq a\leq \infty$ for all $a\in [-\infty,\infty]$; $-\infty<a$ iff $a\in (-\infty,\infty]=\mathbb{R}\cup \{\infty\}$ and $a<\infty$ iff $a\in [-\infty,\infty)=\{-\infty\}\cup \mathbb{R}$. 

We introduce the intervals $[-\infty,a]:=\{-\infty\}\cup (-\infty,a]$ for $a\in [-\infty,\infty)$ (with $(-\infty,-\infty]:=\emptyset$) etc. in the obvious way.

	We specify $0\cdot (\pm\infty):=0=:(\pm\infty)\cdot 0$ and $\infty+(-\infty):=0=:(-\infty)+\infty$.  The remainder of the ``arithmetic'' in $[-\infty,\infty]$ is (then) defined in the ``natural'' way, e.g. $a\cdot \infty=\sgn(a)\infty$ for $a\in [-\infty,\infty]\backslash\{0\}$, $a+\infty=\infty$ for $a\in (-\infty,\infty]$, $\infty-\infty:=\infty+(-\infty)=0$, $\vert \pm\infty\vert=\infty$ etc. 
	
	We endow $[-\infty,\infty]$ with the standard topology corresponding e.g. to the metric $d$ given by $d(x,y):=\vert \arctan(x)-\arctan(y)\vert$ for $\{x,y\}\subset [-\infty,\infty]$ ($\arctan(\infty):=\pi/2$ and $\arctan(-\infty):=-\pi/2$).
	
	 A map with values in $[-\infty,\infty]$ will be said to be numerical\footnote{For our purposes here. In general one would also consider complex-valued functions to be numerical, but we shall not be concerned with those, except for mentioning them in passing here and there.}.
\end{definition}

\begin{definition}
 $\mathcal{B}_{[-\infty,\infty]}:=\sigma_{[-\infty,\infty]}(\{[-\infty,a]:a\in\mathbb{R}\})$. For $A\subset [-\infty,\infty]$, $\mathcal{B}_A:=\mathcal{B}_{[-\infty,\infty]}\vert_A$  is called the Borel $\sigma$-field on $A$, its elements the Borel sets of $A$. 

\end{definition}
\begin{remarks}
We will tend to only need $\BB_A$ for a handful of $A$: $[-\infty,\infty]$, $\mathbb{R}$, $[0,1]$, $[0,\infty]$, $[0,\infty)$. For those that know topology, $\mathcal{B}_A$ coincides with the smallest $\sigma$-field on $A$ that contains all the open sets of $A$ (in the relative standard topology of $A$), and this is what, in general, the Borel $\sigma$-field on a  topological space is defined to be. When looking at measurable maps, the reason we like $\mathcal{B}_{[-\infty,\infty]}$ on the codomain is, roughly speaking, that it is big enough to ensure that measurable maps are sufficiently well-behaved, yet small enough  as to make (most) interesting maps measurable w.r.t. it. Anticipating what follows, we like it also because on $(\mathbb{R},\mathcal{B}_\mathbb{R})$ we can define a nice --- non-trivial translation-invariant ---, so-called Lebesgue measure. One can also argue that to say that a numerical map $f$ is measurable (is interesting from the point of view of measure), one should be able to at least measure the sets $\{f\leq a\}:=f^{-1}([-\infty,a])$ for each $a\in \mathbb{R}$ (in particular, anticipating a little bit the content of the second part of these notes, for a random variable $X$, one should want to be able to say what the probability of the events $\{X\leq a\}$, $a\in \mathbb{R}$, is).
\end{remarks}

\begin{example}
	All the intervals and countable subsets of $[-\infty,\infty]$ belong to $\mathcal{B}_{[-\infty,\infty]}$. All open and all closed sets of $[-\infty,\infty]$ belong to $\mathcal{B}_{[-\infty,\infty]}$. If $A\subset [-\infty,\infty]$ is countable, then $\BB_A=2^A$.
\end{example}

\begin{definition}
	If a function $f$ maps into $[-\infty,\infty]$ (i.e. is numerical), then: $\sigma(f):=\sigma^{\mathcal{B}_{[-\infty,\infty]}}(f)$; for a $\sigma$-field $\FF$ on the domain of $f$, $f$ is $\FF$-Borel measurable \deff $f\in \FF/\mathcal{B}_{[-\infty,\infty]}$; if $g:D_f\to [-\infty,\infty]$, $g\land f:=\min\{g,f\}$ and $g\lor f:=\max\{g,f\}$; $f^+:=\max\{0,f\}$ and $f^-:=\max\{0,-f\}$. 
\end{definition}
\begin{remarks}
	$\FF$ in ``$\FF$-Borel measurable'' is omitted if it can be gathered from context. In particular when the domain of $f$ is an $A\subset [-\infty,\infty]$ we always take $\FF=\mathcal{B}_A$ unless otherwise indicated. Anticipating again what follows, $[-\infty,\infty]$-valued Borel measurable functions are those for which an extremely well-behaved notion of an integral against a measure can be successfully defined. Plainly, $f=f^+-f^-$, $f^-=(-f)^+$ and $\vert f\vert=f^++f^-$.
\end{remarks}

\begin{example}
	By \eqref{equation:restrictions}, $\mathcal{B}_\mathbb{R}=\sigma_\mathbb{R}(\{(-\infty,a]:a\in \mathbb{R}\})$. Consequently, by Propositions~\ref{proposition:maps}\eqref{maps:iv} and~\ref{proposition:restrictions}\eqref{restrict:a}, for a $\sigma$-field $\FF$ on $\Omega$ and $f:\Omega\to\mathbb{R}$, $f$ is $\FF$ Borel measurable iff $f\in \FF/\mathcal{B}_\mathbb{R}$ iff $\{f\leq a\}:=f^{-1}((-\infty,a])\in \FF$ for all $a\in \mathbb{R}$. 
\end{example}

\begin{proposition} 
	If $A\subset [-\infty,\infty]$ and $f:A\to [-\infty,\infty]$ is continuous, then $f\in \mathcal{B}_A/\mathcal{B}_{[-\infty,\infty]}$. If $\{f,g\}\subset \FF/\mathcal{B}_{[-\infty,\infty]}$ for a $\sigma$-field $\FF$, then $\{f+g,fg\}\subset \FF/\mathcal{B}_{[-\infty,\infty]}$ and $\{\{f=g\},\{f\leq g\},\{f<g\}\}\subset \FF$.  
\end{proposition}
 
\begin{proof}
For the first claim  recall for instance that inverse images of closed sets under continuous maps are closed. Let now  $\{f,g\}\subset \FF/\mathcal{B}_{[-\infty,\infty]}$. To show that $f+g\in \FF/\mathcal{B}_{[-\infty,\infty]}$,  observe that for all $t\in \mathbb{R}$, $\{f+g>t\}=(\{f=\infty\}\cap \{g>-\infty\})\cup (\{g=\infty\}\cap \{f>-\infty\})\cup \{f=-\infty,g=\infty,t< 0\}\cup \{f=\infty,g=-\infty,t< 0\}\cup (\{f\in \mathbb{R},g\in \mathbb{R}\}\cap (\cup_{r\in \mathbb{Q}}\{r<f\}\cap \{g>t-r\}))$. To see that $fg\in \FF/\mathcal{B}_{[-\infty,\infty]}$, assume first $f> 0$ and $g> 0$,  write $fg=\exp(\log (f)+\log (g))$ and check that by continuity $\exp\in  \mathcal{B}_{(-\infty,\infty]}/ \mathcal{B}_{(0,\infty]}$ and $\log\in  \mathcal{B}_{(0,\infty]}/ \mathcal{B}_{(-\infty,\infty]}$ (with $\exp\infty:=\infty$ and $\log \infty:=\infty$); this yields $fg\in \FF/\mathcal{B}_{(0,\infty]}$ in the particular case of strictly positive maps, from which the general case follows easily. Finally, note that $\{f<g\}=\cup_{r\in \mathbb{Q}}(\{f<r\}\cap \{r<g\})\in \FF$, taking complements $\{g\leq f\}\in \FF$, hence also $\{f\leq g\}\in \FF$ and $\{f=g\}=\{f\leq g\}\cap \{g\leq f\}\in \FF$.
\end{proof}

\begin{example}
	The map $f:\mathbb{R}\to \mathbb{R}$  given by $f(x)=\mathbbm{1}_{[0,1]}(x)\sin(x+\mathbbm{1}_{[0,2]}(x))/(1+x^2)$, $x\in \mathbb{R}$, belongs to  $\mathcal{B}_\mathbb{R}/\mathcal{B}_\mathbb{R}$ (in other words, is Borel measurable).
	\end{example}

\begin{proposition}
	Let $\FF$ be a $\sigma$-field and let $(f_n)_{n\in \mathbb{N}}$ be a sequence in $\FF/\mathcal{B}_{[-\infty,\infty]}$. Then $$\left\{\sup_{n\in \mathbb{N}}f_n,\inf_{n\in \mathbb{N}}f_n,\limsup_{n\to\infty}f_n,\liminf_{n\to\infty}f_n\right\}\subset \FF/\mathcal{B}_{[-\infty,\infty]}.$$ 
Besides, $$\left(\sum_{n\in \mathbb{N}}f_n\text{ on }\left\{\sum_{n\in \mathbb{N}}f_n\text{ converges}\right\}\right)\in \FF\vert_{\{\sum_{n\in \mathbb{N}}f_n\text{ converges}\}}/\mathcal{B}_{[-\infty,\infty]};$$ in particular, 	if $f_n\geq 0$ for all $n\in \mathbb{N}$, then $\sum_{n\in \mathbb{N}}f_n\in \FF/\BB_{[0,\infty]}$.
\end{proposition}
\begin{proof}
	For $a\in \mathbb{R}$, $\{\sup_{n\in \mathbb{N}}f_n\leq a\}=\cap_{n\in \mathbb{N}}\{f_n\leq a\}$. Apply Proposition~\ref{proposition:maps}\eqref{maps:iv} to find that $\sup_{n\in \mathbb{N}}f_n\in \FF/\mathcal{B}_{[-\infty,\infty]}$. To see it for the $\inf$, replace the $f$s, by $-f$s. The inferior and superior limits are just combinations of taking suprema and then infima or vice versa. Finally, a series is the limit of its partial sums (when convergent). 
\end{proof}

\begin{corollary}
Let $\FF$ be a $\sigma$-field. Then $\{f\lor g,f\land g,f^+,f^-,\vert f\vert\}\subset \FF/\mathcal{B}_{[-\infty,\infty]}$ whenever $\{f,g\}\subset \FF/\mathcal{B}_{[-\infty,\infty]}$. Also, $\{\{f_n\text{ converges as $n\to\infty$}\},\{f_n\text{ converges to a value in $\mathbb{R}$ as $n\to\infty$}\},\{\lim_{n\to\infty}f_n=f_0\}\}\subset \FF$ whenever $(f_n)_{n\in \mathbb{N}_0}$ is a sequence in $\FF/\mathcal{B}_{[-\infty,\infty]}$.\qed 
		\end{corollary}

\begin{remark}
Roughly speaking, any subset of $[-\infty,\infty]$ and any $[-\infty,\infty]$-valued map one meets in practice is Borel measurable. Still, it is often not trivial to prove that a map or a set is measurable (and even more non-trivial to prove that it is not!). There exist non-Borel measurable subsets of $\mathbb{R}$, i.e. $\mathcal{B}_\mathbb{R}\subsetneq 2^\mathbb{R}$.
\end{remark}

\section{Monotone class arguments}
\emph{(approximation with) simple functions,  monotone class theorem; Doob-Dynkin/factorization lemma; $\pi$-systems, $\lambda$-systems, Dynkin's lemma; measures that agree on a $\sigma$-localizing generating $\pi$-system agree everywhere}\footnote{The content of this section is by no means exhaustive; see \underline{\href{https://www.fmf.uni-lj.si/~vidmarm/Dynkin_and_pi_systems.pdf}{here}} for many more results of similar flavor, especially the so-called functional monotone class.}

\begin{definition}
	Let $\FF$ be a $\sigma$-field on $\Omega$ and $f:\Omega\to [0,\infty)$. $f$ is $\FF$-simple \deff $f\in \FF/\mathcal{B}_{[0,\infty)}$ and the range of $f$ is finite. 
\end{definition}
\begin{remark}
	A caveat: it is not standard, but we insist all simple functions are nonnegative, which will save us from adding the qualifier ``nonnegative'' in what follows.
\end{remark}

\begin{proposition}\label{proposition:approx}
	Let $(\Omega,\FF)$ be a measurable space and $f:\Omega\to[0,\infty]$. Then $f$ is $\FF$-simple iff $f=\sum_{i=1}^nc_i\mathbbm{1}_{A_i}$ for some $(c_i)_{i=1}^n$ from $[0,\infty)$, $(A_i)_{i=1}^n$ from $\FF$ and $n\in \mathbb{N}$. Furthermore, if $f\in \FF/\mathcal{B}_{[0,\infty]}$, then $((2^{-n}\lfloor 2^nf\rfloor )\land n)_{n\in \mathbb{N}}$ is a sequence of $\FF$-simple functions that is $\uparrow f$ (even uniformly on every set on which $f$ is bounded).\footnote{We take $\lfloor\infty\rfloor:=\infty$, of course.}\qed
\end{proposition}
\begin{remark}
There is nothing ``canonical'' about the choice of $(2^{-n}\lfloor 2^nf\rfloor )\land n$ in the preceding; what matters is only the existence of an approximating sequence with the indicated properties.
\end{remark}

\begin{corollary}[Monotone class theorem]\label{corollary:monotone-class}
Let $\FF$ be a $\sigma$-field on $\Omega$ and $\mathcal{M}\subset \FF/\mathcal{B}_{[0,\infty]}$. If $\mathbbm{1}_A\in \mathcal{M}$ for all $A\in \FF$, if $\mathcal{M}$ is a convex cone (i.e. closed under nonnegative linear combinations\footnote{Meaning: $af+bg\in \mathcal{M}$ whenever $\{a,b\}\subset [0,\infty)$ and $\{f,g\}\subset \mathcal{M}$.}), and if $\mathcal{M}$ is closed under nondecreasing limits\footnote{Meaning: $\lim_{n\to\infty}f_n\in \mathcal{M}$ whenever $(f_n)_{n\in \mathbb{N}}$ is a nondecreasing sequence of functions from $\mathcal{M}$.}, then $\mathcal{M}=\FF/\mathcal{B}_{[0,\infty]}$. 
\end{corollary}
\begin{proof}
Because  $\mathcal{M}$ is a convex cone containing the indicators of $\FF$ it contains all $\FF$-simple  functions. Then the claim follows from Proposition~\ref{proposition:approx} and the fact that $\mathcal{M}$ is closed under nondecreasing limits. 
\end{proof}
\begin{remark}
	The importance of this result will only become fully apparent later on. In a nutshell it allows us to ``raise'' a claim from knowing that it holds for indicators of measurable sets, to all nonnegative measurable maps (using arguments of linearity and monotone convergence), and then to all measurable numerical maps (again by an argument of linearity, writing a function as the difference of its positive and negative part). A first application along these lines is given in (the proof of) the proposition to follow. 
\end{remark}
\begin{proposition}[Doob-Dynkin factorization lemma]\label{prop:doob-dynkin}
Let $X:\Omega\to A$, $(A,\AA)$ a measurable space. Then
$$Y\in \sigma^\AA(X)/\mathcal{B}_{[-\infty,\infty]} \text{ iff } \left(\exists h\in \AA/\mathcal{B}_{[-\infty,\infty]}\text{ with }Y=h(X)\right).$$
\end{proposition}
\begin{remarks}
In the preceding $h(X)$ is $h\circ X$. This is an (and will be a standard) abuse of notation. Note that for measurable spaces $(A,\AA)$ and $(C,\CC)$, for $f:A\to B$ and $g:B\to C$ we have that 
$$f\in \AA/\sigma^\CC(g)\text{ iff } g\circ f\in \AA/\CC\text{ iff }g\in \sigma_\AA^B(f)/\CC;$$
understanding measurability in an initial structure on the codomain (and in a final structure on the domain) is easy. The present proposition aims to understand it when it features on the domain, at least if the codomain is the extended real line with its Borel $\sigma$-field (the latter could be relaxed, but we do not do it here).

A trivial consequence of the Doob-Dynkin lemma is that for a measurable space $(\Omega,\FF)$ and ${\Omega'}\subset \Omega$, $$f\in (\FF\vert_{\Omega'})/\BB_{[-\infty,\infty]} \text{ iff }\exists g\in \FF/\BB_{[-\infty,\infty]}\text{ such that }f=g\vert_{\Omega'}$$ (indeed we have only to consider $X=\mathrm{id}_{\Omega'}$, $(A,\AA)=(\Omega,\FF)$ in Proposition~\ref{prop:doob-dynkin}).
%
\end{remarks}
\begin{proof}
Certainly $X\in  \sigma^\AA(X)/\AA$. The condition is then sufficient because compositions of measurable maps are measurable. For the converse, let $\mathcal{M}:=\{Y\in\sigma^\AA(X)/\mathcal{B}_{[0,\infty]}:  \exists h\in \AA/\mathcal{B}_{[0,\infty]}\text{ with }Y=h(X)\}$. Then $\mathcal{M}$ is a convex cone closed under nondecreasing limits and, by the very definition of $\sigma^\AA(X)$, it contains all the indicators of the members of $\sigma^\AA(X)$. By monotone class $\mathcal{M}=\sigma^\AA(X)/\mathcal{B}_{[0,\infty]}$. If finally $Y\in \sigma^\AA(X)/\mathcal{B}_{[-\infty,\infty]}$ then by what we have just proven $Y^+=h_+(X)$ and $Y^-=h_-(X)$ for some $h_+$ and $h_-$ from  $\AA/\mathcal{B}_{[0,\infty]}$. Then $Y=Y^+-Y^-=(h_+-h_-)(X)$ and $h_+-h_-\in \AA/\mathcal{B}_{[-\infty,\infty]}$.
\end{proof}
\begin{definition}
	Let $\DD\subset 2^\Omega$.
	
	 $\DD$ is  a Dynkin system (also, a $\lambda$-system) on $\Omega$ \deff $\Omega\in \DD$ and ($B\backslash A\in \DD$ provided $\DD\ni A\subset B\in \DD$) and (whenever $(A_i)_{i\in \mathbb{N}}$ is a sequence in $\DD$ with $A_i\subset A_{i+1}$ for all $i\in\mathbb{N}$, then $\cup_{i\in \mathbb{N}}A_i\in\DD$). 
	
		 $\DD$ is a $\pi$-system \deff $\DD$ is closed for $\cap$. 
\end{definition}
\begin{remark}
In words, $\DD$ is a Dynkin system on $\Omega$ if it contains $\Omega$, it is closed for differences of elements that are comparable w.r.t. inclusion (``comparable differences'') and also for countable nondecreasing unions. Clearly $\emptyset$ belongs to any Dynkin system.
\end{remark}
\begin{example}
	$\{(-\infty,a]:a\in \mathbb{R}\}$ is a $\pi$-system. 
\end{example}
\begin{proposition}
	Let $\DD\subset 2^\Omega$. 	 $\DD$ is a Dynkin system on $\Omega$ iff
	
	\begin{quote} $\Omega\in \DD$,  $\DD$ is closed for $\cc^\Omega$, and whenever $(A_i)_{i\in\mathbb{N}}$ is a sequence in $\DD$ with $A_i\cap A_j=\emptyset$ for all $i\ne j$ from $\mathbb{N}$, then $\cup_{i\in \mathbb{N}}A_i\in\DD$.\end{quote} 	$\DD$ is a $\sigma$-field on $\Omega$ iff it is a  $\pi$-system and a $\lambda$-system on $\Omega$. 
\end{proposition}
 
\begin{proof}
Consider the first statement.	If $\DD$ is a Dynkin system, then for $\{A,B\}\subset \DD$ with $A\cap B=\emptyset$, $B\subset \Omega\backslash A$ and so $A\cup B=\Omega\backslash ((\Omega\backslash A)\backslash B)\in \DD$. Conversely, if the stated conditions hold,  then for $\{A,B\}\subset \DD$ with $A\subset B$, $A\cap (\Omega\backslash B)=\emptyset$ and so  $B\backslash A=\Omega\backslash (A\cup (\Omega\backslash B))\in \DD$. The remainder of the argument is (even more) trivial.
\end{proof} 
 
\begin{definition}
	Given  $\LL\subset 2^\Omega$ we put $$\lambda_\Omega(\LL):=\cap\{\DD\in 2^{2^\Omega}:\DD\text{ is a Dynkin system on }\Omega\text{ and }\LL\subset \DD\}.$$
\end{definition}

 \begin{remark}
 	Just as for $\sigma$-fields, $\lambda_\Omega(\LL)$ is the smallest (w.r.t inclusion) $\lambda$-system on $\Omega$ that has $\LL$ for a subset.
 \end{remark}

\begin{proposition}
	Let $\LL$ be a $\pi$-system and $\LL\subset 2^\Omega$. Then $\sigma_\Omega(\LL)=\lambda_\Omega(\LL)$. 
\end{proposition}
 
\begin{proof}
	Since every $\sigma$-algebra on $\Omega$ is a $\lambda$-system on $\Omega$, the inclusion $\sigma_\Omega(\LL)\supset \lambda_\Omega(\LL)$ is manifest. For the reverse inclusion, it will be sufficient (and, indeed, necessary) to show that $\lambda_\Omega(\LL)$ is a $\sigma$-algebra on $\Omega$. Then it will be sufficient to check that $\lambda_\Omega(\LL)$ is a $\pi$-system. Define $\UU:=\{A\in \lambda_\Omega(\LL): A\cap B\in \lambda_\Omega(\LL)\text{ for all }B\in \LL\}$: $\UU$ is a $\lambda$-system, containing $\LL$, so $\lambda_\Omega(\LL)\subset \UU$. Now define $\VV:=\{A\in \lambda_\Omega(\LL):A\cap B\in \lambda_\Omega(\LL)\text{ for all }B\in \lambda_\Omega(\LL)\}$. Again $\VV$ is a $\lambda$-system, containing, by what we have just shown, $\LL$. But then $\lambda_\Omega(\LL)\subset \VV$, and we conclude.
\end{proof}
 
\begin{corollary}[$\pi$-$\lambda$--theorem/Dynkin's lemma/Sierpinski class theorem]\label{corollary:dynkin}
	Let $\LL$ be a $\pi$-system, $\DD$ a Dynkin system on $\Omega$, $\LL\subset \DD$. Then $\sigma_\Omega(\LL)\subset \DD$. \qed
\end{corollary}
\begin{remark}
	The significance of the above result stems mainly from the following observation. Let $\AA$ be a $\sigma$-field, and $\LL\subset 2^\Omega$ be a $\pi$-system. Suppose $\AA=\sigma_\Omega(\LL)$. Often in measure theory (and in probability theory, in particular) we wish to show that some property (typically involving  measures), $P(A)$ in $A\in \AA$, holds true of all the  members $A$ of $\AA$, and we wish to do so, having already been given, or having already established beforehand, that $P(A)$ holds true for all $A\in \LL$. Now, it is usually easy (because measures behave nicely under disjoint/nondecreasing countable unions, and because finite\footnote{But not arbitrary measures; for this reason the result is most often applicable in settings involving finite/probability or at best (via some localization) $\sigma$-finite measures.} measures also behave nicely under ``comparable differences''/complements) to check directly that the collection $\DD:=\{A\in \sigma_\Omega(\LL):P(A)\}$ is a $\lambda$-system, but it is usually not so easy  (because measures do not behave so nicely under arbitrary countable unions) to check directly that $\DD$ is a $\sigma$-field. The $\pi$/$\lambda$--lemma establishes an extremely useful shortcut of doing so indirectly. Its applicability is further strengthened by the fact that it is typically not difficult to find a generating $\pi$-system on which the property $P$ ``obviously'' holds true. The next result is got along the lines indicated above and we shall find repeated use of it in the remainder of this text.
\end{remark}

\begin{proposition}[Measures agreeing on a $\sigma$-localizing generating $\pi$-system]\label{proposition:equality of measures}
Let $\mu$ and $\nu$ be two measures on a measurable space $(E,\Sigma)$, $\LL\subset \Sigma$ a $\pi$-system with $\sigma_E(\LL)=\Sigma$. Suppose $\mu\vert_{\LL}=\nu\vert_{\LL}$ and that there exists a sequence $(L_n)_{n\in \mathbb{N}}$ in $\LL$ that consists of   $\uparrow$ or pairwise disjoint sets and is such that $\mu(L_n)=\nu(L_n)<\infty$ for each $n\in \mathbb{N}$, while $\cup_{n\in \mathbb{N}}L_n=E$  [a ``localizing'' sequence]. Then $\mu=\nu$. 
\end{proposition}
\begin{remark}
In words: two measures that agree on a ``$\sigma$-localizing'' generating $\pi$-system agree everywhere. If $\mu$ and $\nu$ are finite measures with the same total mass (in particular, if $\mu$ and $\nu$ are probabilities), then the existence of the sequence $(L_n)_{n\in \mathbb{N}}$ becomes superfluous (for one can replace $\LL$ with $\LL\cup \{E\}$ and take $L_n=E$ for each $n\in\mathbb{N}$).
\end{remark}
\begin{proof}
According as the sequence $(L_n)_{n\in \mathbb{N}}$ consists of   $\uparrow$ or pairwise disjoint sets it is enough by continuity from below or by countable additivity of $\mu$ and $\nu$ to argue that $\mu_{L_n}=\nu_{L_n}$ for each $n\in \mathbb{N}$: indeed, if we have argued so, then we see that, for each $A\in \Sigma$,  $\mu(A)=\lim_{n\to\infty}\mu(A\cap L_n)=\lim_{n\to\infty}\nu(A\cap L_n)=\nu(A)$ or $\mu(A)=\sum_{n\in \mathbb{N}}\mu(A\cap L_n)=\sum_{n\in \mathbb{N}}\nu(A\cap L_n)=\nu(A)$.\footnote{One calls the preceding a ``localization argument''/``reduction by localization'' (the localization being to finite measure sets).} Therefore, because $\mu_{L_n}(L_n)=\mu(L_n)=\nu(L_n)=\nu_{L_n}(L_n)<\infty$ and because $\LL\vert_{L_n}$ is  a $\pi$-system with $\sigma_{L_n}(\LL\vert_{L_n})=\sigma_E(\LL)\vert_{L_n}=\Sigma\vert_{L_n}$ and with $\mu_{L_n}\vert_{\LL\vert_{L_n}}=\nu_{L_n}\vert_{\LL\vert_{L_n}}$ ($\because$ $\LL\vert_{L_n}\subset \LL$) for each $n\in \mathbb{N}$, we may just as well, and do assume that $\mu$ and $\nu$ are finite measures of the same mass to begin with.
 Then from the hypotheses and from the basic properties of measures, the collection of sets $\mathcal{D}:=\{A\in \Sigma:\mu(A)=\nu(A)\}$ is a $\lambda$-system on $E$ containing $\LL$: $\Omega\in \mathcal{D}$ because $\mu$ and $\nu$ have the same mass; if $\{A,B\}\subset \DD$ with $A\subset B$, then $\mu(B\backslash A)=\mu(B)-\mu(A)=\nu(B)-\nu(A)=\nu(B\backslash A)$, therefore $B\backslash A\in \mathcal{D}$; finally, if $(A_n)_{n\in\mathbb{N}}$ is a $\uparrow$ sequence in $\mathcal{D}$, then $\mu(\cup_{n\in \mathbb{N}}A_n)=\lim_{n\to\infty}\mu(A_n)=\lim_{n\to\infty}\nu(A_n)=\nu(\cup_{n\in \mathbb{N}}A_n)$, hence $\cup_{n\in \mathbb{N}}A_n\in\DD$; the fact that $\LL\subset\DD$ is just the assumption that $\mu\vert_{\LL}=\nu\vert_{\LL}$. It remains to apply Dynkin's lemma, which yields $\mathcal{D}\supset \sigma_E(\LL)=\Sigma$, that is to say $\mu=\nu$.
\end{proof}

\section{Lebesgue-Stieltjes measures}
\emph{measure attached to a nondecreasing right-continuous function; Lebesgue measure}
\begin{theorem}[Lebesgue-Stieltjes measures]\label{theorem:lebesgue-stieltjes}
Let $F:\mathbb{R}\to \mathbb{R}$ be right-continuous and nondecreasing. Then there exists a unique measure $\mu$ on $\mathcal{B}_\mathbb{R}$ such that $\mu((a,b])=F(b)-F(a)$ for all real $a\leq b$. $\blacktriangle$
\end{theorem}
\begin{proof}[Proof of uniqueness]  Apply Proposition~\ref{proposition:equality of measures} with the $\pi$-system $\{(a,b]:a\leq b\text{ real}\}$ that generates $\mathcal{B}_\mathbb{R}$ (a localizing sequence is got for instance from $(n,n+1]$, $n\in \mathbb{Z}$).
\end{proof}
\begin{remarks}
One typically proves existence via the extension theorem of Carath\'eodory. 
Clearly, ceteris paribus, for $\mu$ with the prescribed property to exist, it is necessary that $F$ is non-decreasing and right-continuous; thus the result is, in a sense, ``best possible''. It is in fact clear that any measure $\mu$ on $(\mathbb{R},\BB_\mathbb{R})$ that is finite on bounded sets (``locally finite'') is equal to to $\dd F$ for a suitable (unique up to an additive constant) $F$, namely $F(x)=\mu((0,x])$ for $x\in \mathbb{R}$ (here we interpret $\mu((0,x]):=-\mu((x,0])$ for $x\in (-\infty,0)$).
\end{remarks}
 
\begin{proof}[Proof of existence] 
By restricting to each $(n,n+1]$, $n\in \ZZ$, and then patching back together we may restrict everything to the interval $(0,1]$. Consider the algebra $\AA$ on $[0,1]$ consisting of all the finite disjoint unions of $\{0\}$ and of intervals of the form $(a,b]$, where $0\leq a\leq b\leq 1$ are real numbers. We define a set-function $\nu$ on $\AA$ by insisting that $\nu(\{0\})=0$, that  $\nu((a,b])=F(b)-F(a)$ for real $0\leq a\leq b\leq 1$, and then extending to $\AA$ by additivity. By Theorem~\ref{theorem:caratheodory} to follow (which will be proved independently of this argument, of course) it now suffices to establish that $\nu$ is countably additive on $\AA$. But by right-continuity of $F$ this clearly falls out of Proposition~\ref{prop:tsirelson} (which again will be proved independently of this section). 
\end{proof}

\begin{definition}
	We denote the $\mu$ of Theorem~\ref{theorem:lebesgue-stieltjes} by $\dd F$ and call it the measure associated to $F$ in the Lebesgue-Stieltjes sense. In the special case $F=\id_\mathbb{R}$, it is called the Lebesgue measure and we will denote it with the symbol $\leb:=\dd(\id_\mathbb{R})$.
\end{definition}
\begin{remark}
	Coming back to the need for $\sigma$-fields: there is no extension of $\leb$ to a measure on $2^\mathbb{R}$ (though, there is a finitely additive translation-invariant extension of $\leb$ to $2^\mathbb{R}$, but such extension is not unique).
\end{remark}
\begin{proposition}\label{proposition:lebesgue-stieltjes}
Let $F:\mathbb{R}\to \mathbb{R}$ be right-continuous and nondecreasing. Then $\dd F$ is: $\sigma$-finite; finite iff $F$ is bounded; a probability iff  $\lim_\infty F-\lim_{-\infty}F=1$. For $x\in \mathbb{R}$, $\dd F(\{x\})=F(x)-F(x-)$. \qed
\end{proposition}

\begin{example}
$\leb$ is $\sigma$-finite and $\leb(A)=0$ for any countable $A\subset \mathbb{R}$.
\end{example}

\begin{example}[The Cantor set]
Consider the following map $V$ defined on the finite disjoint unions of closed subintervals of $[0,1]$: $V(A)$ removes from every connected component  of $A$ (from each closed interval that makes up $A$) the open middle third. Thus $V([0,1])=[0,1/3]\cup[2/3,1]$, $V^2([0,1])=[0,1/9]\cup [2/9,1/3]\cup [2/3,7/9]\cup [8/9,1]$ etc. Then $C:=\cap_{n\in \mathbb{N}}V^n([0,1])$ is an uncountable (in fact of cardinality continuum) compact ($\therefore$ Borel) set with $\leb(C)=0$.
\end{example}
 
\begin{example}
Let $r:\mathbb{N}\to\mathbb{Q}$ be surjective. For $\epsilon\in (0,\infty)$ the set $\cup_{n\in \mathbb{N}}(r_n-\frac{\epsilon}{2^{n+1}},r_n+\frac{\epsilon}{2^{n+1}})$ has Lebesgue measure $\leq \epsilon$ and yet its complement is nowhere dense in the real line.
\end{example}
\section{Completeness of measures}
\emph{complete measures and completions; measurability in the completion}


\begin{definition}
Let $\mu$ be a measure on $\FF$, where $\FF$ is  a $\sigma$-field on $X$. $N$ is $\mu$-null \deff   $N\subset M$ for some $M\in \FF$ with $\mu(M)=0$. We put $\mathcal{N}_\mu:=\{N\in 2^X:N\text{ is $\mu$-null}\}$ and $\overline{\FF}^\mu:=\sigma_X(\FF\cup \mathcal{N}_\mu)$.  $\mu$ is complete \deff $\mathcal{N}_\mu\subset \FF$. 
\end{definition}
\begin{remark}
Referencing $\FF$ in $\overline{\FF}^\mu$ is actually superfluous.
\end{remark}
\begin{proposition}
Let $\mu$ be a measure on $(X,\FF)$ and $f:X\to [-\infty,\infty]$. Then $f\in \overline{\FF}^\mu/\mathcal{B}_{[-\infty,\infty]}$ iff $g\leq f\leq h$ for some $\{g,h\}\subset \FF/\mathcal{B}_{[-\infty,\infty]}$ satisfying $\mu(h>g)=0$. Furthermore, one may insist in the condition that $g$, $h$ are nonnegative (resp. indicators), when $f$ is nonnegative (resp. an indicator).
\end{proposition}
\begin{proof}
Only the necessity of the condition  (involving $g$ and $h$) is not immediately clear. Trivially the condition is true for $f\in \overline{\FF}^\mu/\mathcal{B}_{[-\infty,\infty]}$ which are  indicators of sets from the $\pi$-system $\FF\cup \mathcal{N}_\mu$ (that generates $\overline{\FF}^\mu$), and one may even insist that in this case $g$ and $h$ are themselves both indicators of sets (from $\FF$). Note also that if $(f_n)_{n\in \mathbb{N}}$ is a $\uparrow$ sequence in $\overline{\FF}^\mu/\mathcal{B}_{[0,\infty]}$  and if $g_n\leq f_n\leq h_n$, $\{g_n,h_n\}\subset \FF/\mathcal{B}_{[0,\infty]}$, $\mu(h_n>g_n)=0$ for all $n\in \mathbb{N}$, then $\limsup_{n\to\infty} g_n\leq \lim_{n\to\infty}f_n\leq \liminf_{n\to\infty}h_n$, $\{\limsup_{n\to\infty} g_n,\liminf_{n\to\infty} h_n\}\subset \FF/\mathcal{B}_{[0,\infty]}$, $\mu(\liminf_{n\to\infty} h_n>\limsup_{n\to\infty} g_n)\leq \mu(\limsup_{n\to\infty}\{h_n>g_n\})=0$. This, and a couple of easier considerations, imply easily by Dynkin's lemma and monotone class that the condition obtains for all $f\in  \overline{\FF}^\mu/\mathcal{B}_{[0,\infty]}$, and one may even insist that the $g$ and $h$ are nonnegative (resp. indicators) when $f$ is nonnegative (resp. an indicator). Finally let us take a general $f\in \overline{\FF}^\mu/\mathcal{B}_{[-\infty,\infty]}$. Then we find  $\{g_+,h_+,g_-,h_-\}\subset \FF/\mathcal{B}_{[0,\infty]}$ satisfying $\mu(h_+>g_+)=0=\mu(h_->g_-)$ and $g_+\leq f^+\leq h_+$, $g_-\leq f^-\leq h_-$. We see that $g_+-h_-\leq f\leq h_+-g_-$ (it works with the convention $\infty-\infty=0$), $\{g_+-h_-,h_+-g_-\}\subset  \overline{\FF}^\mu/\mathcal{B}_{[-\infty,\infty]}$ and $\mu(g_+-h_-<h_+-g_-)\leq \mu(\{h_+>g_+\}\cup \{h_->g_-\})=0$ (because there is even inclusion of sets; it still works with the convention $\infty-\infty=0$(!)). Now the proof is complete.
\end{proof}

\begin{definition}
Let $\mu$ be a measure on the $\sigma$-field $\FF$. We define $\overline{\mu}:\overline{\FF}^\mu\to [0,\infty]$ by putting $$\overline{\mu}(F):=\mu(B),\quad F\in \overline{\FF}^\mu,$$  where $B\in \FF$ is such that there is $A\in \FF$ with $A\subset F\subset B$ and $\mu(B\backslash A)=0$. We call $\overline{\mu}$ the completion of $\mu$.
\end{definition}
\begin{remark}
It is easy to check that the preceding definition is without ambiguity, i.e. if also  $B'\in \FF$ is such that there is $A'\in \FF$ with $A'\subset F\subset B'$ and $\mu(B'\backslash A')=0$, then $\mu(B')=\mu(B)$. The existence of $A$ and $B$ is guaranteed by the preceding proposition.
\end{remark}
\begin{proposition}\label{proposition:capture}
If $\mu$ is a measure on $\FF$, where $\FF$ is a $\sigma$-field on $X$, then $\overline{\mu}$ is  the unique extension of $\mu$ to a measure on $\overline{\FF}^\mu$ and is the minimal extension of $\mu$ to a complete measure on a $\sigma$-field of $X$.\qed
\end{proposition}

\begin{corollary}
Let $\mu$ be a measure on $(X,\FF)$ and $f:X\to [-\infty,\infty]$. Then $f\in \overline{\FF}^\mu/\mathcal{B}_{[-\infty,\infty]}$ iff $f=j$ a.e.-$\overline{\mu}$ for some $j\in  \FF/\mathcal{B}_{[-\infty,\infty]}$. \qed
\end{corollary}

\begin{remarks}
In some sense $\overline{\mu}$ is a most natural extension of $\mu$: being trapped in-between two measurable sets, whose difference is of zero measure, ``morally'' implies measurability, even if the latter is not there to begin with. Often completions also ``save the day'' in terms of measurability of maps.  They  are not almighty, though. For instance, it is still the case that the domain  $\mathfrak{L}$ of the completed Lebesgue measure $\overline{\leb}$ (its  members are called Lebesgue-measurable sets)  is not the whole of $2^\mathbb{R}$ (the existence of non-Lebesgue-measurable sets actually hinges on the axiom of choice, cf. the Vitali sets). On the other hand, it is also the case that $\mathcal{B}_\mathbb{R}\subsetneq\mathfrak{L}$, in fact the two families differ in cardinality: $\vert \mathcal{B}_\mathbb{R}\vert=\vert\mathbb{R}\vert$, while $\vert \mathfrak{L}\vert=2^{\vert \mathbb{R}\vert}$. 
\end{remarks}

\section{Outer measures and Carath\'eodory's extension theorem}
\emph{outer measures; extending a countably additive set-function on an algebra to a measure}

\begin{definition}
 $\mu$ is an outer measure on $X$ \deff $\mu:2^X\to [0,\infty]$ and $\mu(\emptyset)=0$ and [monotonicity] $\mu(A)\leq\mu(B)$ whenever $A\subset B\subset X$ and [countable subadditivity] $\mu(\cup_{n\in \mathbb{N}}A_n)\leq \sum_{n\in \mathbb{N}}\mu(A_n)$ for any sequence $(A_n)_{n\in \mathbb{N}}$ in $2^X$.
\end{definition}
\begin{remark}
Because an outer measure does not charge the empty set, countable subaddivity implies finite subadditivity in the obvious meaning of these qualifications.
\end{remark}

\begin{definition}
Let $\mu$ be an outer measure on $X$ and $E\subset X$. $E$ is $\mu$-Carath\'eodory--measurable \deff $\mu(A)=\mu(A\cap E)+\mu(A\backslash E)$ for all $A\subset X$. We put $\MM(\mu):=\{E\in 2^X:E\text{ is $\mu$-Carath\'eodory--measurable}\}$.
\end{definition}
\begin{remark}
In the definition of $\mu$-Carath\'eodory--measurability the ``$\leq$-inequality'' is automatic by subadditivity.
\end{remark}

\begin{proposition}
Let $\mu$ be an outer measure on $X$. Then $\MM(\mu)$ is a $\sigma$-field on $X$ and $\mu\vert_{\MM(\mu)}$ is a complete measure thereon.
\end{proposition}
\begin{proof}
It is clear that $X\in \MM(\mu)$ and that $\MM(\mu)$ is closed for $\cc^X$. Let $\{E,F\}\subset \MM(\mu)$. Then for any $A\subset X$, from the subadditivity of $\mu$, from $F\in \MM(X)$ and from $E\in \MM(X)$ (in this order), we obtain $\mu(A\cap (E\cup F))+\mu(A\backslash (E\cup F))\leq \mu((A\cap E)\cap F)+\mu((A\cap E)\backslash  F)+\mu((A\backslash E)\cap  F)+\mu((A\backslash E)\backslash F)= \mu(A\cap E)+\mu(A\backslash E)=\mu(A)$. Therefore $E\cup F\in \MM(X)$. Now suppose $(E_n)_{n\in \mathbb{N}}$ is a sequence of pairwise disjoint sets in $\MM(\mu)$. Set $F_n:=\cup_{i=1}^nE_i$ for $n\in \mathbb{N}$ and $F:=\cup_{i\in \mathbb{N}}E_i$. Then for $A\subset X$ and $n\in \mathbb{N}$: $\mu(A\cap F_n)=\mu(A\cap F_n\cap E_n)+\mu((A\cap F_n)\backslash E_n)=\mu(A\cap E_n)+\mu(A\cap F_{n-1})$; inductively $\mu(A\cap F_n)=\sum_{j=1}^n\mu(A\cap E_j)$; besides, $F_n\in \MM(\mu)$, hence $\mu(A)=\mu(A\cap F_n)+\mu(A\backslash F_n)\geq \sum_{j=1}^n\mu(A\cap E_j)+\mu(A\backslash F)$ by monotonicity of $\mu$. Letting $n\to\infty$ in the preceding, by subaddivity of $\mu$, $\mu(A)\geq \sum_{j=1}^\infty\mu(A\cap E_j)+\mu(A\backslash F)\geq  \mu(A\cap F)+\mu(A\backslash F)\geq \mu(A)$. Therefore these inequalities are in fact equalities, $F\in \MM(\mu)$, $\MM(\mu)$ is a $\sigma$-field on $X$, and taking $A=F$ in the preceding yields at once also the countable additivity of $\mu$ on $\MM(\mu)$. Finally, the completeness assertion follows from the observation that $\mu^{-1}(\{0\})\subset \MM(\mu)$ and the monotonicity of $\mu$.
\end{proof}

\begin{definition}
Let $\AA\subset 2^X$ and $\mu:\AA\to [0,\infty]$. Define $\mu^*_X:2^X\to [0,\infty]$  by $$\mu^*_X(A):=\inf \left\{\sum_{j\in \mathbb{N}}\mu(A_j):(A_j)_{j\in \mathbb{N}}\text{ a sequence in $\AA$ with }\cup_{j\in \mathbb{N}}A_j\supset A\right\},\quad A\in 2^X.$$
\end{definition}

\begin{proposition}
Let $\AA\subset 2^X$, $\mu:\AA\to [0,\infty]$, $\mu(\emptyset)=0$. Then $\mu^*_X$ is an outer measure on $X$. \qed
\end{proposition}

\begin{theorem}[Carath\'eodory]\label{theorem:caratheodory} 
Let $\AA$ be an algebra on $X$ and $\mu:\AA\to [0,\infty]$ satisfy $\mu(\emptyset)=0$ and [countable additivity] $\mu(\cup_{n\in \mathbb{N}}A_n)=\sum_{n\in \mathbb{N}}\mu(A_n)$ for every sequence $(A_n)_{n\in\mathbb{N}}$ in $\AA$ consisting of pairwise disjoint sets for which $\cup_{n\in \mathbb{N}}A_n\in \AA$. Then $\sigma_X(\AA)\subset \MM(\mu^*_X)$ and $\mu^*_X\vert_{\sigma_X(\AA)}$ is an extension of $\mu$ to a measure on $\sigma_X(\AA)$. If furthermore [$\sigma$-finiteness] there exists a sequence $(A_n)_{n\in \mathbb{N}}$ in $\AA$ with $\cup_{n\in \mathbb{N}}A_n=X$ and with  $\mu(A_n)<\infty$ for all $n\in \mathbb{N}$, then the extension of $\mu$ to a measure on $\sigma_X(\AA)$ is unique and $\mu^*_X\vert_{\MM(\mu^*_X)}=\overline{\mu^*_X\vert_{\sigma_X(\AA)}}$.
\end{theorem}
\begin{proof}
Write $\mu^*:=\mu^*_X$ and $\MM:=\MM(\mu^*)$ for short.

It is easy to see that $\mu$ is countably subadditive and monotone, from which it follows  that $\mu^*$ extends $\mu$. For the first claim it therefore remains to check that $\sigma_X(\AA)\subset \MM$ and this will follow from $\AA\subset \MM$. Let then $E\in \AA$, $A\subset X$ and $(A_n)_{n\in \mathbb{N}}$ be a cover from $\AA$ of $A$. Then  $(A_n\cap E)_{n\in \mathbb{N}}$ is a cover from $\AA$ of $A\cap E$ and likewise  $(A_n\backslash  E)_{n\in \mathbb{N}}$ is a cover from $\AA$ of $A\backslash E$. Therefore 
$\mu^*(A\cap E)+\mu^*(A\backslash E)\leq \sum_{j\in \mathbb{N}}\mu(A_j\cap E)+\sum_{j\in \mathbb{N}}\mu(A_j\backslash E)=\sum_{j\in \mathbb{N}}\mu(A_j)$ by additivity of $\mu$, which entails that $\mu^*(A\cap E)+\mu^*(A\backslash E)\leq \mu^*(A)$, and hence that indeed $E\in \MM$.

Suppose now the $\sigma$-finitness condition holds true. Uniqueness of the extension follows  from Proposition~\ref{proposition:equality of measures} ($\sigma$-finiteness gives a localizing sequence). Next, the inclusion (set-theoretic notation) $\mu^*\vert_{\MM}\supset \overline{\mu^*\vert_{\sigma_X(\AA)}}$ follows from the two facts: (i) $\overline{\mu^*\vert_{\sigma_X(\AA)}}$  is the minimal extension to a complete measure of the measure  $\mu^*\vert_{\sigma_X(\AA)}$; and (ii) $\mu^*\vert_{\MM}$ is a complete measure extending the measure $\mu^*\vert_{\sigma_X(\AA)}$. For the reverse inclusion, let $M\in \MM$. It will now be sufficient to establish that $M$ belongs to the completed $\sigma$-field $\overline{\sigma_X(\AA)}^{\mu^*\vert _{\sigma_X(\AA)}}$ (since $\overline{\mu^*\vert_{\sigma_X(\AA)}}$ extends $\mu^*\vert_{\sigma_X(\AA)}$ uniquely to $\overline{\sigma_X(\AA)}^{\mu^*\vert _{\sigma_X(\AA)}}$). By the $\sigma$-finiteness of $\mu$ (and the definition of $\mu^*$) we may assume that $\mu^*(M)<\infty$. From the very definition of $\mu^*$, for each $n\in \mathbb{N}$, there is then a sequence $(A^n_i)_{i\in \mathbb{N}}$ from $\AA$, with $M\subset A^n:=\cup_{i\in \mathbb{N}}A^n_i$ and $\mu^*(M)>\sum_{i=1}^\infty\mu(A^n_i)-\frac{1}{n}\geq \mu^*(A^n)-\frac{1}{n}\geq \mu^*(A)-\frac{1}{n}$ where $A:=\cap_{n\in \mathbb{N}} A^n\in \sigma_X(\AA)$. Passing to the limit $n\to\infty$, we find that $\mu^*(A\backslash M
)=0$. Hence it will be sufficient to establish that every $\mu^*$-negligible set is contained in a $\mu^*\vert_{\sigma_X(\AA)}$-negligible set. But this also follows at once by what we have just proven.
\end{proof}
\begin{remark}
The preceeding proof of the existence of the extension goes through with ``only'' countable subadditivity, additivity and monotonicity of $\mu$ in lieu of countable additivity. Nevertheless, for the conclusion (existence of extension) to prevail, it is clearly necessary that $\mu$ be countably additive. In fact,  countable subaddivity, additivity and monotonicity of $\mu$ imply its countable additivity as is easy to check, so there is nothing strange going on here.
\end{remark}

\begin{proposition}\label{prop:tsirelson}
Let $X$ be a compact topological space, $\AA$ an algebra on $X$ and let $\mu:\AA\to [0,\infty]$ be additive (meaning: $\mu(A\cup B)=\mu(A)+\mu(B)$ for all disjoint $A$ and $B$ from $\AA$). Suppose the following condition holds true:  for every $A\in \AA$ and for every $\epsilon\in (0,\infty)$ there is a $B\in \AA$ such that $\overline{B}\subset A$ and $ \mu(A\backslash B)\leq\epsilon$. Then $\mu$ is countably additive and $\mu(\emptyset)=0$. 
\end{proposition}
\begin{proof}
 The conclusion $\mu(\emptyset)=0$ is straightforward (apply the condition with $A=\emptyset$). 
  By additivity of $\mu$, in order to establish countable additivity, it will be enough to show that 
for any sequence $(A_n)_{n\in \mathbb{N}}$ in $\AA$ that is $\downarrow \emptyset$, one has $\mu(A_n)\downarrow 0$ as $n\to\infty$. Let $\epsilon\in (0,\infty)$. For each $k\in \mathbb{N}$ there is $B_k\in \AA$ such that $\overline{B_k}\subset A_k$ and $\mu(A_k\backslash B_k)\leq 2^{-k}\epsilon$. Now $\cap_{k\in \mathbb{N}}\overline{B_k}\subset \cap_{k\in \mathbb{N}}A_k=\emptyset$, hence by compactness there is $n\in \mathbb{N}$ such that $\cap_{k=1}^n\overline{B_k}=\emptyset$ and a fortiori $\cap_{k=1}^nB_k=\emptyset$. Therefore by subadditivity of $\mu$, we obtain $\mu(A_n)\leq \mu(A_n\backslash B_n)+\cdots +\mu(A_1\backslash B_1)<\epsilon$, and the proof is complete.
\end{proof}

\chapter{Integration on measurable spaces}

\begin{remarks}
	Recall the ``arithmetic'' of $[-\infty,\infty]$ form Definition~\ref{definition:conventions}. The convention  $\infty\cdot 0=0=0\cdot \infty$  is ``the only natural one'' in the measure-theoretic context. Among other desiderata, this is because it ensures that $a_n\cdot b \uparrow a\cdot b$  whenever $0\leq a_n\uparrow a$ and $b\in [0,\infty]$, even if $a=\infty$ and $b=0$, which is one of the reasons why  monotone convergence (see below) prevails in the generality that it does. Another aspect of it is that it is quite natural to have: (i) the integral of the zero function over any set be zero, even if the set in question has infinite measure; (ii) the integral of any function over a set of zero mesure be zero, even if the function in question is infinite-valued (again see below).  The stipulation that $\infty+(-\infty)=0=(-\infty)+\infty$, on the other hand, is (quite a bit more) arbitrary. But, 
	the expressions $\infty+(-\infty)$ and $(-\infty)+\infty$ will appear in what follows only on sets of measure zero, which are ``not seen'' by the integral (as we will see), or else in cases that would otherwise have been a priori excluded (had we left  $\infty+(-\infty)$ and $(-\infty)+\infty$ undefined). 
	 Hence, to an overriding extent, 	the mentioned arbitrariness is inconsequential. 
\end{remarks}

\section{The Lebesgue integral}
\emph{definition of Lebesgue integration and first properties; examples of integrals; Riemann-Darboux vs. Lebesgue}

\begin{definition}
	Let $(\Omega,\FF,\mu)$ be a measure space and $f\in \FF/\mathcal{B}_{[-\infty,\infty]}$. 
	\begin{enumerate}[(a)]
		\item\label{integral:a} If $f$ is $\FF$-simple, we put $\int f \dd\mu:=\sum_{a\in \mathrm{range}(f)}a\mu(f^{-1}(\{a\}))$ \big($=\sum_{a\in \mathrm{range}(f)}a\mu(f=a)$\big).\footnote{It may happen that $\mathrm{range}(f)=\emptyset$; of course this occurs iff $\Omega=\emptyset$. We always put $\sum_\emptyset:=0$. (And $\prod_\emptyset:=1$.)} 
		\item\label{integral:b} If $f\geq 0$, but $f$ is not $\FF$-simple, then we put 
		\begin{equation*}
		\int f \dd \mu:=\sup\left\{\int g\dd\mu: g\text{ is $\FF$-simple and }g\leq f\right\}.
		\end{equation*}
		\item\label{integral:c} If $f$ is not nonnegative, then we put $\int f\dd \mu:=\int f^+\dd\mu-\int f^-\dd \mu$.
	\end{enumerate}

$\int f\dd \mu$ is called the (Lebesgue\footnote{We shall tend to omit this qualification.}) integral (also, expectation if $\mu$ is a probability measure) of $f$ under $\mu$. We write variously $\mu[f]:=\mu^x[f(x)]:=\int f(x)\mu(\dd x):=\int f\dd\mu$, whichever is the more convenient. If further $A\in \FF$, we put $\mu[f;A]:=\mu^x[f(x);A]:=\mu^x[f(x);x\in A]:=\int_Af(x)\mu(\dd x):=\int_{x\in A}f(x)\mu(\dd x):=\int_A f\dd \mu:=\int f\mathbbm{1}_A\dd \mu$.

The integral of $f$ against $\mu$ is well-defined \deff  $\int f^+\dd \mu\land \int f^-\dd \mu<\infty$.

$f$ is $\mu$-integrable \deff $\int f^+ \dd\mu\lor \int f^-\dd\mu<\infty$. 
\end{definition}

\begin{remarks}
Given that \eqref{integral:b} and~\eqref{integral:c} (but not \eqref{integral:a}) can handle a non-measurable  $f$  equally as they can handle a measurable one, why consider integrals of only measurable maps (at least in the first instance)? The reason lies in the fact that measurable maps are precisely those that can be approximated by (pointwise monotone) limits of simple functions. As a consequence the integral is most well-behaved, and naturally defined by the preceding definition precisely on this class. Note also that, by definition, $\int f\dd \mu=\infty-\infty=0$ when $\int f\dd\mu$ is not well-defined; however, this case is included for definiteness/convenience only, and no non-trivial consequences should, or will be attached to this convention.

The notation $\mu[f]$ for the integral, apart from being very concise, conveys also the idea that the integral be viewed as an extension of the map $\mu$ from $\FF$ --- identified with $\{\mathbbm{1}_A:A\in \FF\}$ --- to a map defined on the whole of $\FF/\mathcal{B}_{[-\infty,\infty]}$ (or some ``rich'' subset thereof); this is the perspective of the (functional-analytic) Daniell integration. 

It is elementary to note that, for $A\in \FF$, $\mu[f;A]=\mu_A[f\vert_A]$, the left-hand side being well-defined iff the right-hand side is so. It is less obvious, but nevertheless true (it will follow from monotone convergence [presented below in Theorem~\ref{theorem:convergence}\eqref{theorem:convergence:Levi}], and from the monotone class theorem [Corollary~\ref{corollary:monotone-class}]) that for a sub-$\sigma$-field $\GG$ of $\FF$ and an $f\in \GG/\mathcal{B}_{[-\infty,\infty]}$, $\mu[f]=(\mu\vert_\GG)[f]$, the left-hand side being well-defined iff the right-hand side is so. Thus, the integral behaves ``naturally'' under restrictions.
 
	\end{remarks}

\begin{definition}
	Let $(\Omega,\FF,\mu)$ be a measure space. $\LL^1(\mu):=\{f\in \FF/\mathcal{B}_{\mathbb{R}}:f\text{ is }\mu\text{-integrable}\}$. For $g:\Omega\to \mathbb{C}$ with $\{\Re g,\Im g\}\subset \LL^1(\mu)$ one defines $\int g\dd \mu:=\int \Re g\dd\mu+\mathrm{i}\int \Im g\dd\mu$ along with the usual variations in notation ($\mu^x[g(x)]$ etc.).
\end{definition}

\begin{remark}
We will keep to integrating $[-\infty,\infty]$-valued maps, leaving aside the complex case to which one can usually extend the properties of the Lebesgue integral in a  straightforward fashion, by linearity and/or some other trick (e.g. writing $\vert x\vert=\sup_{\theta\in \mathbb{R}}\Re (e^{\ii\theta}x)$ for $x\in \mathbb{C}$). 
\end{remark}

\begin{theorem}\label{theorem:int:basics}
		Let $(\Omega,\FF,\mu)$ be a measure space. The integral enjoys the following properties.
		\begin{enumerate}[(i)]
			\item\label{lebesgue:additivity} Additivity: $\mu[ f+g ]=\mu[f]+\mu[g]$ whenever $\{f,g\}\subset \FF/\mathcal{B}_{[-\infty,\infty]}$ and $\mu[f^-]\lor \mu[g^-]<\infty$.
			\item\label{properties:indicator} Integral of indicator: $\int \mathbbm{1}_A\dd \mu=\mu(A)$ for all $A\in \FF$ (in particular $\int 0\dd \mu=0$, hence $\int f\dd\mu=\int f^+\dd\mu-\int f^-\dd\mu$ for all $f\in \FF/\mathcal{B}_{[-\infty,\infty]}$). 
			\item\label{properties:v} Integrals vanishing, integrals finite: For $f\in \FF/\mathcal{B}_{[0,\infty]}$, $\int f\dd \mu=0$ iff $f=0$ a.e.-$\mu$, while $\int f\dd \mu<\infty$ implies that $f<\infty$ a.e.-$\mu$.  
			\item\label{properties:iv} Triangle inequality: $\vert \int f\dd\mu\vert\leq \int \vert f\vert \dd\mu$ for all $f\in \FF/\mathcal{B}_{[-\infty,\infty]}$.
			\item\label{properties:vi} Measure zero sets are ``invisible'': $\int f\dd\mu=\int g\dd \mu$ whenever $\{f,g\}\subset \FF/\mathcal{B}_{[-\infty,\infty]}$ and $f=g$ a.e.-$\mu$. 
					\item\label{properties:iii}  Monotonicity: $\int g\dd\mu\leq\int f\dd \mu$ whenever  $\{g,f\}\subset \FF/\mathcal{B}_{[-\infty,\infty]}$, $g\leq f$ and $\int g^-\dd\mu <\infty$.
			\item\label{properties:ii} Homogeneity: $\int cf\dd\mu=c\int f\dd\mu$ for all $f\in \FF/\mathcal{B}_{[-\infty,\infty]}$ for which $\int (cf)^-\dd \mu\land \int (cf)^+\dd\mu<\infty$,  for all $c\in [-\infty,\infty]$.	
		\end{enumerate}
		Furthermore, the integrals appearing in  \eqref{lebesgue:additivity}-\eqref{properties:indicator}-\eqref{properties:v} \& \eqref{properties:iii} are all well-defined. The same is true for \eqref{properties:ii} except  when $c=0$  and $\mu[f^+]= \mu[f^-]=\infty$. Finally, in \eqref{properties:vi}, $\int f\dd\mu$ is well-defined iff $\int g\dd\mu$ is well-defined. 
\end{theorem}
\begin{remarks}
	The proof of the additivity of the integral is not so easy; one shows it usually first for nonnegative functions, the above extension is then ``immediate''. If one uses both additivity and homogeneity of the integral in an argument, then one refers to it simply as linearity (of the integral). The condition ``$\int (cf)^-\dd \mu\land \int (cf)^+\dd\mu<\infty$'' of \eqref{properties:ii} is certainly met when $c\in \mathbb{R}$ and $\int f\dd\mu$ is well-defined. Because of the homogeneity property (resp. because an integral does not see sets of measure zero), many statements concerning integrals get complemented if one changes any integrand(s) to its (their) negative (resp. can  be reinforced by weakening the hypotheses with an ``a.e.'' qualifier). One should keep this in mind; we will not stress it further.
\end{remarks}
 
\begin{proof}[Proof of \eqref{lebesgue:additivity} for simple functions]
The key is to prove the following particular case first: if $f$ is $\FF$-simple and if $f=\sum_{i=1}^na_i\mathbbm{1}_{A_i}$ for some sequence $(a_i)_{i\in [n]}$ in $[0,\infty)$ and some sequence $(A_i)_{i\in [n]}$ in $\FF$, some $n\in \mathbb{N}$, then $\int f\dd \mu=\sum_{i=1}^na_i\mu(A_i)$. To see this, let 
$\mathcal{Q}$ be the partition generated on $\Omega$ by $(A_i)_{i\in [n]}$ (i.e. $\mathcal{Q}$ is the coarsest partition of $\Omega$ such that each $A_i$, $i\in [n]$, is the union of some (perhaps none) of the members of $\mathcal{Q}$). It is clear that $\mathcal{Q}\subset \mathcal{F}$, that for each $a\in \text{range}(f)$ the set $\{f=a\}$ is the union of some (perhaps none) of the $Q\in \mathcal{Q}$, and that, for each $Q\in \mathcal{Q}$ with $Q\subset \{f=a\}$, $a=\sum_{i\in [n],A_i\cap Q\ne \emptyset}a_i$. By additivity of $\mu$ and because $\mu(\emptyset)=0$ we therefore see that $\sum_{i=1}^na_i\mu(A_i)=\sum_{i=1}^na_i\sum_{Q\in \mathcal{Q}}\mu(A_i\cap Q)=\sum_{Q\in \mathcal{Q}}\sum_{i=1}^na_i\mu(A_i\cap Q)=\sum_{a\in \text{range}(f)}\sum_{Q\in \mathcal{Q},Q\subset \{a=f\}}\sum_{i=1}^na_i\mu(A_i\cap Q)=\sum_{a\in \text{range}(f)}\sum_{Q\in \mathcal{Q},Q\subset \{a=f\}}\sum_{i=1,A_i\cap Q\ne \emptyset}^na_i\mu(Q)=\sum_{a\in \text{range}(f)}a\sum_{Q\in \mathcal{Q},Q\subset \{a=f\}}\mu(Q)=\sum_{a\in \text{range}(f)}a\mu(f=a)=\int f\dd\mu$. Now that this has been established, additivity for simple functions follows at once. To extend to nonnegative functions  we will use monotone convergence for nonnegative functions (Theorem~\ref{theorem:convergence}\eqref{theorem:convergence:Levi} with $g=0$ to follow). To safeguard against  a circulus vitiosus we present its proof already here.
\end{proof}
\begin{proof}[Proof of monotone convergence for nonnegative functions]\label{Levi-nonnegative}
We prove Theorem~\ref{theorem:convergence}\eqref{theorem:convergence:Levi} in case $g=0$. The fact that the sequence $(\int f_n\dd\mu)_{n\in \mathbb{N}}$ is $\uparrow$ as well as the inequality $\lim_{n\to\infty}\int f_n\dd \mu\leq \int \lim_{n\to\infty}f_n\dd \mu$ are trivial (for instance by additivity and hence monotonicity of the integral over simple functions). To prove the reverse inequality let $s$ be an $\FF$-simple function minorizing $\lim_{n\to\infty}f_n$. Let $\alpha\in [0,1)$ be arbitrary and define $E^\alpha_n:=\{f_n\geq \alpha s\}$ for $n\in \mathbb{N}$. Because $s\leq \lim_{n\to\infty}f_n$ we see that $\cup_{n\in \mathbb{N}}E^\alpha_n=\Omega$. Besides,  for each $n\in \mathbb{N}$, one has $\int f_n\dd\mu\geq \int_{E_n^\alpha} f_n\dd\mu\geq \int_{E_n^\alpha}\alpha s \dd\mu=\alpha\int_{E_n^\alpha} s\dd\mu$ , which is $\uparrow \alpha\int s\dd\mu$ as $n\to\infty$ (because $\mu$ is continuous from below; recall $s$ is simple); so $\lim_{n\to\infty}\int f_n\dd\mu\geq \alpha\int s\dd\mu$, which is finally $\uparrow \int s\dd\mu$ as $\alpha\uparrow 1$. Now the sought-for reverse inequality follows from the very definition of the Lebesgue integral.
\end{proof}

\begin{proof}[Proof of additivity in the general case]
Additivity for nonnegative functions follows by approximation (Proposition~\ref{proposition:approx}) and  by monotone convergence for nonnegative functions (Theorem~\ref{theorem:convergence}\eqref{theorem:convergence:Levi} with $g=0$). For the general  case note that $(f+g)^++f^-+g^-=(f+g)^-+f^++g^+$. Apply to the latter additivity of the integral for nonnegative functions, and subtract finite quantities, noting that always $\int h\dd\mu=\int h^+\dd\mu-\int h^-\dd\mu$ for $h\in \FF/\mathcal{B}_{[-\infty,\infty]}$ (because $\int 0\dd\mu=0$). 
\end{proof}

\begin{proof}[Proof of \eqref{properties:indicator} through \eqref{properties:ii} (assuming \eqref{lebesgue:additivity})] 
	\eqref{properties:indicator}. If $\Omega=\emptyset$ then $A=\emptyset$ and by the definition of the integral of an $\FF$-simple function $\int  \mathbbm{1}_A\dd \mu=\sum_\emptyset=0=\mu(A)$. Assume $\Omega\ne\emptyset$. Again by definition of the integral of an $\FF$-simple function: $\int \mathbbm{1}_A\dd \mu=0\cdot\mu(\Omega\backslash A)+1\cdot \mu(A)=\mu(A)$, unless $A=\emptyset$ or $A=\Omega$; but $\int\mathbbm{1}_\emptyset \dd\mu=\int 0\dd\mu=0\cdot \mu(\Omega)=0=\mu(\emptyset)$, while $\int\mathbbm{1}_\Omega \dd\mu=\int 1\dd\mu=1\cdot \mu(\Omega)=\mu(\Omega)$.  
	
	\eqref{properties:v}. Suppose first $f=0$ a.e.-$\mu$. Then if $g$ is $\FF$-simple and $g\leq f$, for all $a\in \mathrm{range}(g)$, $a>0$ implies $\mu(g=a)=0$, whence $\int g\dd\mu=\sum_{a\in\mathrm{range}(g)}a\mu(g=a)=0$. So, by definition of the integral $\int f\dd\mu=0$. Conversely, suppose $\int f\dd \mu=0$. If $f$ is $\FF$-simple we obtain at once from the definition of the integral that $f=0$ a.e.-$\mu$. Otherwise suppose per absurdum that $\mu(f>0)>0$. Then by continuity from below $\mu(f\geq 1/n)>0$ for some $n\in \mathbb{N}$, and  by definition of the integral  $\int f\dd \mu\geq \int \frac{1}{n}\mathbbm{1}_{\{f\geq 1/n\}}\dd \mu=\frac{1}{n}\mu(f\geq 1/n)>0$, a contradiction. Finally, assume $\int f\dd \mu<\infty$ and, per absurdum, that $\mu(f=\infty)>0$. Then, for $n\in \mathbb{N}$, again by definition of the integral, $\int f\dd\mu\geq \int n\mathbbm{1}_{\{f=\infty\}}  \dd \mu=n \mu(f=\infty)\uparrow \infty$ as $n\to\infty$, a contradiction.
	
	\eqref{properties:iv}. If $f\geq 0$ there is nothing to prove, because the integral of a nonnegative function is anyway nonnegative. Otherwise, by definition of the integral $\int f\dd \mu=\int f^+\dd \mu-\int f^-\dd \mu$, so $\vert \int f\dd \mu\vert\leq \int f^+\dd \mu+\int f^-\dd\mu=\mu[f^++f^-]=\int \vert f\vert \dd \mu$, where we have used  \eqref{lebesgue:additivity}.
	
	\eqref{properties:vi}. We may assume $f$ and $g$ are both $\geq 0$. Indeed the assumption $\mu(f\ne g)=0$ entails that a fortiori $\mu(f^+\ne g^+)=0$, $\mu(f^-\ne g^-)=0$, while by \eqref{properties:indicator}, $\mu[f]=\mu[f^+]-\mu[f^-]$ and $\mu[g]=\mu[g^+]-\mu[g^-]$. 
	Therefore, assuming now that $f$ and $g$ are both $\geq 0$, by \eqref{lebesgue:additivity}, $\mu[f]=\mu[f\mathbbm{1}_{\{f<\infty\}}+f\mathbbm{1}_{\{f=\infty\}}]=\mu[f\mathbbm{1}_{\{f<\infty\}}]+\mu[f\mathbbm{1}_{\{f=\infty\}}]=\mu[f\mathbbm{1}_{\{f<\infty\}}]+\mu[\infty\mathbbm{1}_{\{f=\infty\}}]=\mu[f\mathbbm{1}_{\{f<\infty\}}]+\infty\mu(f=\infty)=\mu[f\mathbbm{1}_{\{f<\infty\}}]+\infty\mu(f=\infty,f=g)+\infty\mu(f=\infty,f\ne g)=\mu[f\mathbbm{1}_{\{f<\infty\}}]+\infty\mu(f=\infty,f=g)+\infty\cdot 0=\mu[f\mathbbm{1}_{\{f<\infty\}}]+\infty\mu(f=\infty,f=g)$ (we have used the fact that $\mu[\infty;A]=\infty\mu(A)$ for  $A\in \FF$, but this is plain from \eqref{properties:v}); similarly for $g$. Hence, since  $\mu(f\ne g)=0$  implies that a fortiori $\mu(f\mathbbm{1}_{\{f<\infty\}}\ne g\mathbbm{1}_{\{g<\infty\}})=0$,  we may also assume that $f$ and $g$ are both finite. In such case,  by \eqref{properties:v} we see that $\int \vert g-f\vert\dd \mu=\int (g-f)^+\dd \mu=\int (g-f)^-\dd \mu=\mu[ g-f]=0$ and thus,  by  \eqref{lebesgue:additivity}, $\mu[f]=\mu[f]+\mu[g-f]=\mu[f+(g-f)]=\mu[g]$. 
	
	\eqref{properties:iii}. Consider first the  case when $\{g=-\infty\}=\emptyset$. By  \eqref{lebesgue:additivity} we then see that $\int f\dd \mu=\mu[g+(f-g)]=\mu[g]+\mu[f-g]\geq \int g\dd\mu$, because the integral of a nonnegative function is always nonnegative by definition.  For the general case, by \eqref{properties:v} $g>-\infty$ a.e.-$\mu$; so $g=g\mathbbm{1}_{\{g>-\infty\}}=:g'$ a.e.-$\mu$ and  $f=f\mathbbm{1}_{\{g>-\infty\}}=:f'$ a.e.-$\mu$. Now $-\infty<g'\leq f'$; apply the case we have  proved already and  \eqref{properties:vi}.
	
	\eqref{properties:ii}. We use the previous items without special mention. For $c\in (0,\infty)$ the relation is immediate from the definitions. Also for $c=0$. For $c=\infty$, the assumption $\int (cf)^-\dd \mu\land \int (cf)^+\dd\mu<\infty$ entails that $f\geq 0$ a.e.-$\mu$ or $f\leq 0$ a.e.-$\mu$. If the former occurs, $f=f^+$ a.e.-$\mu$ and: if $\int f^+\dd \mu>0$, then $\mu(\infty f^+=\infty)=\mu(f^+>0)>0$, hence $\int \infty f\dd \mu=\int \infty f^+\dd \mu=\infty=\infty \int f^+\dd \mu=\infty \int f\dd \mu$; if $\int f^+\dd \mu=0$, then $f^+=0$ a.e.-$\mu$, hence also $\infty f^+=0$ a.e.-$\mu$, and so $\int \infty f\dd \mu=\int \infty f^+\dd \mu=0=\infty\int f^+\dd \mu=\infty\int f\dd \mu$. Otherwise, if the latter prevails,  by what we have just shown $\int \infty f\dd \mu=-\int  (\infty f)^-\dd\mu =-\int \infty f^-\dd\mu=-\infty \int  f^-\dd\mu=\infty (-\int f^-\dd \mu)=\infty \int f\dd \mu$. Besides, for $c=-1$: 
	 $\int (-f)\dd\mu=\int (-f)^+\dd\mu-\int (-f)^-\dd\mu=\int f^-\dd\mu-\int f^+\dd\mu=-(\int f^+\dd\mu-\int f^-\dd\mu)=-\int f\dd\mu$. The claim follows.
	
	The last string of assertions of the theorem is immediate.
\end{proof}

%
%

\begin{example}
Let $x\in \Omega$ and $f:\Omega\to [-\infty,\infty]$. Then $\int f\dd \delta_x$ is well-defined and $\int f\dd\delta_x=f(x)$. 
\end{example}

\begin{example}
	For $f:\Omega\to [0,\infty]$: put $\sum f:=\sup\{\sum_{x\in F}f(x):F\text{ finite}\subset \Omega\}$; note, if $\sum f<\infty$, then $\{f\geq \epsilon\}$ is finite for all $\epsilon>0$ and so $\{f>0\}=\cup_{n\in \mathbb{N}}\{f\geq \frac{1}{n}\}$ is countable. Then,  for $f:\Omega\to [-\infty,\infty]$,  $\int f\dd c_\Omega$ is well-defined iff $\sum f^+\land \sum f^-<\infty$, in which case $\int f\dd c_\Omega=\sum f^+-\sum f^-$ (in particular if $\Omega=\mathbb{N}$, then $\int f\dd c_\Omega=\sum_{n\in \mathbb{N}}f(n)$ /in the sense of the convergence, in $[-\infty,\infty]$, of the partial sums, as always/).
\end{example}

\begin{proposition}\label{proposition:riemann-darboux}
	Suppose $a\leq b$ are real numbers and $f:[a,b]\to \mathbb{R}$. If $f$ is continuous, then $f\mathbbm{1}_{[a,b]}$ is $\leb$-integrable and \begin{equation}
	\int_{[a,b]}f\dd \leb=\int_a^bf(x)\dd x,
	\end{equation}
	where the integral on the right-hand side is in the Riemann-Darboux sense.

	Moreover, $f$ is Riemann-Darboux integrable iff it is continuous $\overline{\leb}$-a.e. and bounded, in which case $f$ is $\overline{\leb}$-integrable and  $
	\int_{[a,b]}f\dd\overline{ \leb}=\int_a^bf(x)\dd x$. 
\end{proposition}
\begin{proof}
We write $\int f$ for the Riemann-Darboux integral. If $f$ is Riemann-Darboux integrable, then it is bounded, so we may and do assume $f$ is bounded. For each $n\in \mathbb{N}$ define $I_k^n:=(k2^{-n},(k+1)2^{-n}]\cap [a,b]$ for $k\in \ZZ$  and the functions, mapping $[a,b]\to\mathbb{R}$: $$g_n:=\sum_{k\in \ZZ}\left(\inf_{I_k^n} f\right)\mathbbm{1}_{I_k^n}\text{ and }h_n:= \sum_{k\in \ZZ}\left(\sup_{I_k^n} f\right)\mathbbm{1}_{I_k^n};$$ then $g_n\leq f\leq h_n$, and $(g_n)_{n\in \mathbb{N}}$ is $\uparrow$ to some $g$ and $(h_n)_{n\in \mathbb{N}}$ is $\downarrow$ to some $h$, in particular $g\leq f\leq h$.

Suppose first that $f$ is Riemann-Darboux integrable. Then  $\int g_n\dd\leb=\int g_n\to\int f$ and $\int  h_n\dd \leb=\int h_n\to \int f$ as $n\to\infty$. 
By monotone (or dominated) convergence (see Theorem~\ref{theorem:convergence} to follow, whose proof will of course not depend on this argument) $\int g\dd \leb=\int f=\int h\dd \leb$. This forces $g=h$ a.e.-$\leb$, and hence $f$ is $\overline{\leb}$-integrable with $\int f\dd\overline{\leb}=\int f$. Also, for $\overline{\leb}$-a.e. $x\in [a,b]$ we have that: $g_n(x)\uparrow f(x)$ and $h_n(x)\downarrow f(x)$ as $n\to\infty$ with $g_n$ and $h_n$ continuous at $x$ for every $n\in \mathbb{N}$; this altogether entails continuity of $f$ at $x$.

For the converse, assume $f$ is continuous a.e.-$\overline{\leb}$. Because $g_n\uparrow f$ and $h_n\downarrow f$ as $n\to\infty$ at every continuity point of $f$, by monotone (or dominated) convergence $\int h_n-\int g_n=\int h_n\dd\leb-\int g_n\dd\leb\to \int f\dd \leb-\int f\dd \leb=0$ as $n\to\infty$. By Archimedes-Riemann it follows that $f$ is Riemann-Darboux integrable.
\end{proof}
 
\begin{remarks}
For a map $f$ defined at least on $A$ we understand $f\mathbbm{1}_A$ as $=0$ outside of $A$ even if $f$ is not defined on (all of) the complement of $A$ (``outside of $A$'' means ``on $\Omega\backslash A$'', where $\Omega$ is the ambient set (the domain) of $\mathbbm{1}_A$ that must be gathered from context). Similarly an expression of the form $R(x)\mathbbm{1}_A(x)$ in $x\in \Omega$ is taken as $=0$ for $x\in \Omega\backslash A$ even when $R(x)$ is not defined  for (all) $x\in \Omega\backslash A$.
	
	The Lebesgue integral is seen to subsume the proper Riemann-Darboux integral (of functions on compact intervals) assuming one takes for the measure the completion of the Lebesgue measure. We prefer the former to the latter because it is vastly more general, and because of the extremely nice properties that it enjoys (we have seen some already, and will see many more in the next sections). A caveat though: $\frac{\sin(x)}{x}$ is Riemann integrable on $x\in (0,\infty)$ in the improper Riemann-Darboux sense, but its Lebesgue integral is not well-defined, because both the integral of its positive as well as the integral of its negative part diverge. 
\end{remarks}

\begin{example}
$\mathbbm{1}_{[0,1]\cap \mathbb{Q}}$ is not Riemann-Darboux integrable on $[0,1]$, while $\int \mathbbm{1}_{[0,1]\cap \mathbb{Q}}\dd \leb$ is well-defined and $=0$. 
\end{example}
 
\begin{remark}
The different types of integrals are a bit of a zoo. We have mentioned thus far those of Riemann-Darboux,  of Lebesgue and Daniell. Of the more important ones, we may add to such a list integration of forms on (smooth) manifolds, spectral integration (against projection-valued measures), the Bochner integral (of Banach space-valued maps)  and stochastic integrals (It\^o, Stratonovich, Paley-Wiener).
\end{remark}

\section{Convergence theorems with some immediate corollaries}
\emph{monotone (Levi) and dominated (Lebesgue) convergence; semicontinuity (Fatou's lemma); integral of a series of functions and  against a series of measures; push-forwards and the image measure theorem; differentiation under the integral sign}

\begin{theorem}\label{theorem:convergence}
Let  $(\Omega,\FF,\mu)$ be a measure space and $(f_n)_{n\in \mathbb{N}}$ be a sequence in $\FF/\mathcal{B}_{[-\infty,\infty]}$. 
\begin{enumerate}[(i)]
	\item Suppose for some $g\in \FF/\mathcal{B}_{[0,\infty]}$ with $\mu[g]<\infty$, $f_n^-\leq g$ for all $n\in \mathbb{N}$. Then we have as follows.
	\begin{enumerate}[(a)]
		\item Lower semi-continuity (Fatou[{}'s lemma]): $\mu[(\liminf_{n\to\infty}f_n)^-]<\infty$ and 
\begin{equation}		
		\int \liminf_{n\to\infty}f_n\dd \mu\leq \liminf_{n\to\infty}\int f_n\dd\mu.
		\end{equation}
		\item\label{theorem:convergence:Levi} Monotone convergence (Levi): If  $f_n\leq f_{n+1}$ for all $n\in \mathbb{N}$, then $\mu[(\lim_{n\to\infty}f_n)^-]<\infty$ and 
		\begin{equation}
		\int \lim_{n\to\infty}f_n\dd \mu= \uparrow\!\!\text{-}\!\!\lim_{n\to\infty}\int f_n\dd\mu.
		\end{equation}
	\end{enumerate}
\item Dominated convergence (Lebesgue): If there is a $\mu$-integrable $g\in \FF/\mathcal{B}_{[0,\infty]}$ with $\vert f_n\vert\leq g$ for all $n\in \mathbb{N}$ and if $\lim_{n\to\infty}f_n$ exists (everywhere), then one has  $\mu[\vert \lim_{m\to\infty}f_m\vert]<\infty$ and $\lim_{n\to\infty}\int \vert f_n-\lim_{m\to\infty}f_m\vert \dd\mu=0$ and (hence)
\begin{equation}
\int \lim_{n\to\infty}f_n\dd \mu= \lim_{n\to\infty}\int f_n\dd\mu.
\end{equation}
	\end{enumerate}
	All the integrals appearing above are well-defined. $\blacktriangle$
\end{theorem}
\begin{example}
	$\int_{[0,\infty)}e^{-x}\leb(\dd x)=\lim_{n\to\infty}\int_{[0,n]}e^{-x}\leb(dx)=\lim_{n\to\infty}\int_0^ne^{-x}\dd x=\int_0^\infty e^{-x}\dd x=1$. 
\end{example}
\begin{remarks}
A special case of dominated convergence occurs when $\mu$ is finite and $g$ bounded; it is then called bounded convergence. Usually one uses monotone convergence and Fatou's lemma for nonnegative functions (in which case one can take $g=0$). Monotone convergence is perhaps the single most important property of Lebesgue integration. The way to remember which way the inequality appears in Fatou's lemma is to consider e.g. the sequence $(\mathbbm{1}_{(n,n+1]})_{n\in \mathbb{N}}$ and the Lebesgue measure $\leb$.   A ``grown-up'' version of (extension to) dominated convergence is due to Vitali. It is worth noting what happens to Fatou if one applies it to the negatives of the functions under consideration (under the relevant integrability assumptions, of course): the inferior limits change to the superior limits and the inequality is reversed (``reverse Fatou''). As a final point, ``embedded'' in the convergence theorems are sufficient conditions for the interchange of limiting operations on double real sequences: if $(a_{mn})_{(m,n)\in \mathbb{N}^2}$ is a double sequence in $\mathbb{R}$ such that $\lim_{m\to\infty}a_{mn}$ exists in $\mathbb{R}$ for each $n\in \mathbb{N}$, then, setting $a_{m0}:=0$ for $m\in \mathbb{N}$, we have the string of equalities $\lim_{m\to\infty}\lim_{n\to\infty}a_{mn}=\lim_{m\to\infty}\sum_{n\in \mathbb{N}}a_{mn}-a_{m(n-1)}=\lim_{m\to\infty}\int a_{mn}-a_{m(n-1)}c_\mathbb{N}(\dd n)=\int \lim_{m\to\infty } a_{mn}-a_{m(n-1)} c_\mathbb{N}(\dd n)=\int \lim_{m\to\infty } a_{mn}-\lim_{m\to\infty}a_{m(n-1)} c_\mathbb{N}(\dd n)=\lim_{n\to\infty}\lim_{m\to\infty}a_{mn}$ by monotone (resp. by dominated) convergence \emph{provided} $a_{mn}$ is $\uparrow $ in $n\in\mathbb{N}$ at fixed $m\in\mathbb{N}$ and $a_{mn}-a_{m(n-1)}$ is $\uparrow $ in $m\in\mathbb{N}$  at fixed $n\in\mathbb{N}$ (resp. \emph{provided} $\sum_{n\in \mathbb{N}}\sup_{m\in \mathbb{N}}\vert a_{mn}-a_{m(n-1)}\vert<\infty$).  
\end{remarks}
\begin{proof}[Proof that Levi implies Fatou and Lebesgue] To prove Fatou's lemma from Levi's monotone convergence result note that $(\inf_{m\in \mathbb{N}_{\geq n}}f_m)_{n\in \mathbb{N}}$ is a nondecreasing sequence in $\FF/\mathcal{B}_{[-\infty,\infty]}$. 
Lebesgue's conclusion then follows by applying Fatou to $(-\vert f_n-\lim_{m\to\infty}f_m\vert )_{n\in \mathbb{N}}$ (the fact that $\mu[\vert \lim_{m\to\infty}f_m\vert]<\infty$ is just by monotonicity of the integral, since $\vert \lim_{m\to\infty}f_m\vert\leq g$).
\end{proof}
\begin{remark}
	The proof of monotone convergence is not completely trivial. One shows it first  for nonnegative functions  and this requires a certain amount of  ``muscle power'' (the above slight extension is then trivial). 
\end{remark}
 
\begin{proof}[Proof of Levi]
Considering the sequence $(f_n+g)_{n\in\mathbb{N}}$, by additivity of the integral we reduce at once to the case when $g=0$, and we have already seen the proof of that on p.~\pageref{Levi-nonnegative}. 
\end{proof}
 
\begin{corollary}
Let  $(\Omega,\FF,\mu)$ be a measure space and $(f_n)_{n\in \mathbb{N}}$ be a sequence in $\FF/\mathcal{B}_{[0,\infty]}$. Then 
\begin{equation}
\int \sum_{n\in \mathbb{N}}f_n\dd\mu=\sum_{n\in \mathbb{N}}\int f_n\dd \mu,
\end{equation}
the integrals being well-defined.\qed
\end{corollary}

\begin{corollary}\label{corollary:integral-against-sum}
	Let $(\mu_n)_{n\in \mathbb{N}}$ be a sequence of measures on a measurable space $(\Omega,\FF)$. Then $\sum_{n\in \mathbb{N}}\mu_n$ is a measure on $(\Omega,\FF)$. Furthermore, for all $f\in \FF/\mathcal{B}_{[-\infty,\infty]}$, $\int f\dd(\sum_{n\in \mathbb{N}}\mu_n)$ is well-defined iff  $(\sum_{n\in \mathbb{N}} \int f^+\dd\mu_n)\land \left(\sum_{n\in \mathbb{N}} \int f^-\dd\mu_n\right)<\infty$, and then
	\begin{equation}\label{eq:integral-against-sum}
	\left(\sum_{n\in \mathbb{N}}\mu_n\right)[f]=\sum_{n\in \mathbb{N}} \int f\dd\mu_n. 
	\end{equation}
\end{corollary}
\begin{remark}
		Because one always has access to the zero measure, the preceding includes also the case of finite sums of measures, in particular the sum of two measures is a measure. 
		 
		We may then  note the allied fact that for a measure $\mu$ on $(\Omega,\FF)$ and an $\alpha\in [0,\infty]$, $\alpha\mu$ is again a measure on $(\Omega,\FF)$ (which is trivial for $\alpha=0$,  evident for $\alpha\in (0,\infty)$,  and almost immediate for $\alpha=\infty$); when $\alpha\in (0,\infty)$, then further for $f\in \FF/\BB_{[-\infty,\infty]}$, $(\alpha\mu)[f]=\alpha\cdot \mu[f]$, the left-hand side being well-defined iff the right-hand side is so.
		 
\end{remark}
\begin{proof}
We may write $(\sum_{n\in \mathbb{N}}\mu_n)(A)=\int \mu_n(A) c_\mathbb{N}(\dd n)$ for all $A\in \FF$. This gives the countable additivity of $\sum_{n\in \mathbb{N}}\mu_n$ (for instance). Continuing to see sums as integrals against the counting measure, by a monotone class argument (exploiting linearity and monotone convergence of the integral) \eqref{eq:integral-against-sum}, being on true on indicators of sets from $\FF$, prevails for all $f\in \FF/\mathcal{B}_{[0,\infty]}$. It extends to arbitrary $f\in \FF/\mathcal{B}_{[-\infty,\infty]}$ by considering $f^+$ and $f^-$, and taking the difference (provided $(\sum_{n\in \mathbb{N}}\mu_n)[f^+]\land (\sum_{n\in \mathbb{N}}\mu_n)[f^-]<\infty$). 
\end{proof}

\begin{definition}
	Let $(\Omega,\FF,\mu)$ be a measure space, $(\Omega',\FF')$ a measurable space and $f\in \FF/\FF'$. Then we define the push-forward (also, image) of $\mu$ by $f$ on $\FF'$, denoted $f_{\star_{\FF'}}\mu$ (also $\mu{\circ_{\FF'}} f^{-1}$, $\mu_{f_{\FF'}}$), as the map $f_{\star_{\FF'}}\mu:\FF'\to [0,\infty]$ given by $$(f_{\star_{\FF'}}\mu)(A'):=\mu(f^{-1}(A')),\quad A'\in \FF'.$$ If $\mu$ is a probability measure, then  $f_{\star_{\FF'}}\mu$ is also called the law or distribution of $f$ relative to $\FF'$ under $\mu$.
\end{definition}
\begin{remark}
	One tends to, and we shall almost always suppress $\FF'$ in the notation, assuming that it can be gathered from context, writing just $f_{\star}\mu=\mu{\circ} f^{-1}=\mu_{f}$. 
	  It is immediate from the definition that $g_\star(f_\star\mu)=(g\circ f)_\star \mu$ for a measure $\mu$ and suitably measurable maps $g$ and $f$.
	 
\end{remark}

\begin{corollary}[Image measure theorem]\label{proposition:image-measure-theorem}
	Let $(\Omega,\FF,\mu)$ be a measure space, $(\Omega',\FF')$ a measurable space and $f\in \FF/\FF'$. Then $f_\star \mu$ is a measure on $\FF'$, finite or probability accordingly as  $\mu$ is so. If further $g\in \FF'/\mathcal{B}_{[-\infty\infty]}$, then 
	\begin{equation}\label{eq:image}
	\int g\dd(f_\star\mu)=\int g\circ f\dd\mu,
	\end{equation}
	the integral on the left-hand side being well-defined iff 	the integral on the right-hand side is well-defined.
\end{corollary}
\begin{proof}
	The fact that $f_\star\mu$ is  a measure (probability or finite according as to whether $\mu$ is so) follows from the fact that preimages commute with unions (besides the even more trivial observations). Formula~\eqref{eq:image} holds for $g=\mathbbm{1}_{F'}$, $F'\in \FF'$, by the very definition of $\mu\circ f^{-1}$. It extends to $\FF'/\mathcal{B}_{[0,\infty]}$ by a monotone class argument (exploiting linearity and monotone convergence of the integral); finally to all of $\FF'/\mathcal{B}_{[-\infty\infty]}$ by considering $g^+$ and $g^-$ and then taking differences.
\end{proof}
 
\begin{remark}
So measures can always be ``pushed forward'' through  a measurable map. Can they be ``pulled back''? This is a little trickier. Let $f:A\to B$, $\BB$ a $\sigma$-field on $B$ and $\mu$ a measure thereon. Then it easy to see that [there exists a measure $\nu$ on $\sigma^\BB(f)$ s.t. $\mu=f_\star\nu$] iff [$\mu(B_1)=\mu(B_2)$ whenever $f^{-1}(B_1)=f^{-1}(B_2)$ and $\{B_1,B_2\}\subset \BB$] iff [$\mu$ vanishes on $\{B'\in \BB:f^{-1}(B')=\emptyset\}$] iff  [$\mu_\star(B\backslash \text{range}f):=\sup \{\mu(B'):B'\in \BB,B'\cap \text{range}(f)=\emptyset\}=0$, i.e. the $\mu$-inner measure of the complement of the range of $f$, is equal to zero]. When it does, let us denote it by $f^\star \mu$ and call it the pull-back of $\mu$ along $f$, so that $f_\star(f^\star \mu)=\mu$. If $\nu$  is a measure on $\sigma^\BB(f)$, then the pull-back of the push-forward of $\nu$ exists and $f^\star(f_\star\nu)=\nu$. Though quite natural, this notion of pull-back is apparently not particularly useful.
\end{remark}
 
	\begin{corollary}[Differentiation under the integral sign]
Let $(X,\Sigma,\mu)$ be a measure space, $O$ an open subset of the reals and let $F:X\times O\to 
\mathbb{R}$ be such that:
\begin{itemize}
\item for all $t\in O$, $F(\cdot, t)\in \LL^1(\mu)$ and 
\item for all $x\in X$, $F(x,\cdot)$ is differentiable.
\end{itemize}
Assume that there exists a $g\in \Sigma/\mathcal{B}_{[0,\infty]}$ with $\int g\dd\mu<\infty$, such that 
$\vert \frac{\partial F}{\partial t}(x,t)\vert \leq g(x)$ for all $(x,t)\in X\times O$. Then: 
\begin{enumerate}[(a)]
\item for each $t\in O$, $(X\ni x\mapsto \frac{\partial F}{\partial t}(x,t))$ is in $\LL^1(\mu)$,
\item $(O\ni t\mapsto \int F(x,t)\mu(\dd x))$ is differentiable and
\item one has that 
\begin{equation}
\frac{\dd}{\dd t}\int F(x,t)\mu(\dd x)=\int \frac{\partial F}{\partial t}(x,t)\mu(\dd x)
\end{equation}
for all $t\in O$. 
\end{enumerate}
\end{corollary}
\begin{remark}
Informally: one may differentiate under the integral sign, if the resulting derivative is bounded in absolute value, locally uniformly in the differentiating variable, by an integrable function of the integrating variable. 
\end{remark}
\begin{proof}
Fix $t\in O$. Measurability of the partial derivative at $t$ follows from the fact that, for some sequence of small enough $0\ne h_n\to 0$ (small enough, in the sense that for all $n\in \mathbb{N}$, $t+h_n\in O$), and then for all $x\in X$, $\frac{\partial F}{\partial 
t}(x,t)=\lim_{n\to\infty}\frac{F(x,t+h_n)-F(x,t)}{h_n}$. Next, for, now any, 
sequence of small enough $0\ne h_n\to 0$ (small enough, in the sense that for all $n\in \mathbb{N}$, $\text{conv}(\{t,t+h_n\})\subset O$):
$$\frac{1}{h_n}\left(\int F(x,t+h_n)\mu(\dd x)-\int F(x,t)\mu(\dd x)\right)= \int\frac{F(x,t+h_n)-F(x,t)}{h_n}\mu(\dd x)\to \int \frac{\partial F}{\partial t}(x,t)\mu(\dd x)\text{ as }n\to\infty$$
by dominated convergence, since by the mean value theorem, for given $n\in\mathbb{N}$ and $x\in X$,  for some $t_{x,n}$ between $t$ and $t+h_n$, $\frac{F(x,t
+h_n)-F(x,t)}{h_n}=\frac{\partial F}{\partial t}(x,t_{x,n})$ (automatically also the integrability of the partial derivative follows by dominated convergence).
\end{proof}
\begin{remark}
In some cases monotone convergence can be used in lieu of dominated convergence in the preceding proof, and then the integrability conditions can be relaxed.
\end{remark}


\section{Results concerning the interchange of the order of integration}
	\emph{product of measurable spaces and product of measures; Tonelli and Fubini}
	\begin{definition}
	Let $(\Omega,\FF)$ and $(\Omega',\FF')$ be measurable spaces. We set $\FF\otimes \FF':=\sigma_{\Omega\times\Omega'}(\{A\times A':(A,A')\in \FF\times \FF'\})$ for the  product of $\FF$ and $\FF'$. We put $\mathcal{B}_{\mathbb{R}^2}:=\mathcal{B}_\mathbb{R}\otimes \mathcal{B}_\mathbb{R}$ for the Borel $\sigma$-field on $\mathbb{R}^2$ and then $\mathcal{B}_A:=\mathcal{B}_{\mathbb{R}^2}\vert_A$ for $A\subset \mathbb{R}^2$.
	\end{definition}
\begin{proposition}
	If $A\subset \mathbb{R}^2$ and $f:A\to [-\infty,\infty]$ is continuous, then $f\in \mathcal{B}_A/\mathcal{B}_{[-\infty,\infty]}$. 
\end{proposition}
\begin{remarks}
	By induction, or directly (but equivalently) in terms of ``measurable rectangles'', one can define, in the obvious way, the product of  a finite sequence of $\sigma$-fields (and $\mathcal{B}_{\mathbb{R}^n}$ for $n\in \mathbb{N}_{\geq 3}$ in particular, for which there is then an analogue of the previous proposition). We leave this ultimately trivial extension aside. Again for those that know topology: one can show that, for $A\subset \mathbb{R}^n$, $\mathcal{B}_{\mathbb{R}^n}\vert_A$ is also generated by all the open (for the relative standard topology) subsets of $A$. This is then also where the preceding proposition comes from (plainly continuous functions are measurable for the Borel $\sigma$-fields).
\end{remarks}
 
\begin{proof}
One shows that $\mathcal{B}_{\mathbb{R}^2}$ coincides with the Borel $\sigma$-field on $\mathbb{R}^2$, i.e. the $\sigma$-field generated by the standard topology on $\mathbb{R}^2$. It is a consequence of the fact that the latter is the product standard topology of the real line, which is second countable. 
\end{proof}
 
\begin{proposition}
	Let $(\Omega,\FF)$ and $(\Omega',\FF')$ be measurable spaces. Then $\FF\otimes \FF'$ is the smallest $\sigma$-field $\GG$ on $\Omega\times\Omega'$ such that $(\Omega\times\Omega'\ni (\omega,\omega')\mapsto \omega)\in \GG/\FF$ and $(\Omega\times\Omega'\ni (\omega,\omega')\mapsto \omega')\in \GG/\FF'$. Furthermore, if $f\in (\FF\otimes\FF')/\mathcal{B}_{[-\infty,\infty]}$, then $f(x,\cdot)\in \FF'/\mathcal{B}_{[-\infty,\infty]}$ for all $x\in \Omega$ and $f(\cdot,x')\in \FF/\mathcal{B}_{[-\infty,\infty]}$ for all $x'\in \Omega'$. Finally, let $(G,\GG)$ be a measurable space and $f:G\to\Omega$, $f':G\to\Omega'$; then  $(f,f')\in \GG/(\FF\otimes \FF')$ iff $f\in \GG/\FF$ and $f'\in \GG/\FF'$. 
	\end{proposition}
	 
\begin{proof}
The first claim is immediate from the definitions. 

As concerns the second claim, let us prove only the first part; the second then follows ``by analogy''. But it is clear for $f$ of the form $\mathbbm{1}_{A\times A'}$, where $(A,A')\in \FF\times \FF'$. Then the class $\mathcal{Q}:=\{M\in \FF\otimes \FF':\mathbbm{1}_M(x,\cdot)\in \FF'/\mathcal{B}_{[-\infty,\infty]}\}$ is a $\sigma$-algebra containing the system $\{A\times A':(A,A')\in \FF\times \FF'\}$ that generates $\FF\otimes \FF'$. It follows that $\mathcal{Q}=\FF\otimes \FF'$. The extension to $(\FF\otimes \FF')/\mathcal{B}_{[0,\infty]}$ follows by a monotone class argument: the class $\mathcal{M}:=\{f\in \FF\otimes\FF'/\mathcal{B}_{[0,\infty]}:f(x,\cdot)\in  \FF'/\mathcal{B}_{[-\infty,\infty]}\}$ is a convex cone closed under nondecreasing limits (because nonnegative linear combinations and nondecreasing limits of Borel measurable maps are Borel measurable) and containing the indicators of $\FF\otimes \FF'$ (we have just proven it). In full generality it follows because differences of Borel measurable maps are Borel measurable.

For the third claim, in the sufficiency direction, one checks the measurability property on the set of generators defining $\FF\otimes \FF'$. In the necessity direction one uses the measurability of the projection maps and the fact that compositions of measurable maps are measurable.
\end{proof}

\begin{remark}
The preceding proof of the second claim is absolutely typical of measure theory. With some practice, one usually just says that the claim follows by a ``$\pi$-$\lambda$--monotone-class argument'', performing the details in the head. We have seen less involved versions of the argument already in the proofs of Corollaries~\ref{corollary:integral-against-sum} and~\ref{proposition:image-measure-theorem}. 
\end{remark}

\begin{theorem}\label{thm:tonelli-fubini}
		Let $(\Omega,\FF,\mu)$ and $(\Omega',\FF',\mu')$ be $\sigma$-finite measure spaces. Then we have as follows.
		\begin{enumerate}[(i)]
			\item\label{product:i} There exists a unique measure $\nu$ on $\FF\otimes \FF'$, denoted $\mu\times \mu'$,\footnote{The $\times$ in $\mu\times\mu'$ is not the cartesian product; lest we introduce nonstandard notation,  we cannot  avoid denoting different objects with the same symbol (as reprehensible and dangerous as this practice may be).} such that $\nu(A\times A')=\mu(A)\mu'(A')$ for all $(A,A')\in \FF\times \FF'$.
			\item\label{product:ii} Let $f\in (\FF\otimes \FF')/\mathcal{B}_{[-\infty,\infty]}$ and suppose
			
						\begin{enumerate}
				\item\label{product:ii:a} either $f\geq 0$ (Tonelli) or else
				\item\label{product:ii:b} $(\mu\times\mu')[\vert f\vert]<\infty$ (Fubini) or else
				\item\label{product:ii:c}  $\int\int f^-(x,x')\mu(\dd x)\mu'(\dd x')\land \int\int f^-(x,x')\mu'(\dd x')\mu(\dd x)<\infty$ (Tonelli-Fubini).
			\end{enumerate}
			 Then 
			$$\left(\Omega'\ni x'\mapsto \int f(x,x')\mu(\dd x)\right)\in \FF'/\mathcal{B}_{[- \infty,\infty]}\text{ and } \left(\Omega\ni x\mapsto \int f(x,x')\mu'(\dd x')\right)\in \FF/\mathcal{B}_{[-\infty,\infty]},$$
			$$\int f^-(x,x')\mu'(\dd x')<\infty\text{ for $\mu$-a.e. $x\in \Omega$ and }\int f^-(x,x')\mu(\dd x)<\infty\text{ for $\mu'$-a.e. $x'\in \Omega'$},$$ and 
		\begin{equation}\label{interchange}
			\int f\dd(\mu\times \mu')=\int\int f(x,x')\mu(\dd x)\mu'(\dd x')=\int\int 
			f(x,x')\mu'(\dd x')\mu(\dd x).\\
			\end{equation}
			\end{enumerate}
			Furthermore, all the outer integrals appearing in \eqref{interchange} are well-defined. 
\end{theorem}
\begin{definition}\label{definition:product-measure}
	The measure $\mu\times \mu'$ from Theorem~\ref{thm:tonelli-fubini}\eqref{product:i} is called the product of $\mu$ and $\mu'$. $\leb^2:=\leb\times\leb$ is Lebesgue measure on $(\mathbb{R}^2,\mathcal{B}_{\mathbb{R}^2})$. 
\end{definition}
 \begin{remarks}
 	One extends by induction at once to ``finite products'' (but in general not beyond, at least not in any obvious way; though, there is a very natural extension to arbitrary products for probability measures) and given a $\sigma$-finite measure $\mu$ one writes $\mu^n:=\underbrace{\mu\times\cdots\times \mu}_{\text{$n$-times}}$ for the $n$-fold product of $\mu$ with itself, $n\in \mathbb{N}_0$ (for $n=0$ it means $\delta_\emptyset$ on the measurable space $(\{\emptyset\},\{\emptyset,\{\emptyset\}\})$). The $\sigma$-finiteness condition is essential (otherwise there is no uniqueness of a product measure, and the equality of the iterated integrals may fail). Just like $\leb$ in dimension $1$, $\leb^2$ also extends to a finitely additive function on $2^{\mathbb{R}\times \mathbb{R}}$ that is invariant under isometries. The analogous claim is no longer true in dimension $3$ (cf. the Banach-Tarski paradox).

Applying Fubini's result to both $f$ and $-f$ we see that under assumption \eqref{product:ii:b} of Theorem~\ref{thm:tonelli-fubini} one has $f(x,\cdot)\in \LL^1(\mu')$ for $\mu$-a.e. $x$ and $f(\cdot,x')\in \LL^1(\mu)$ for $\mu$-a.e. $x'$.
  
 \end{remarks}
  
\begin{proof}
As concerns uniqueness of $\nu$ in \eqref{product:i} one need only apply Proposition~\ref{proposition:equality of measures} with the $\pi$-system $\Pi:=\{A\times A':(A,A')\in \FF\times \FF'\}$, the existence of a ``localizing'' sequence following from the $\sigma$-finiteness of $\mu$ and $\nu$: there is a sequence $(F_n)_{n\in \mathbb{N}}$ in $\FF$ of pairwise disjoint sets covering $\Omega$ and such that $\mu(F_n)<\infty$ for all $n\in \mathbb{N}$, and there is a sequence $(F_n')_{n\in \mathbb{N}}$ in $\FF'$ of pairwise disjoint sets covering $\Omega'$ and such that $\mu'(F_n')<\infty$ for all $n\in \mathbb{N}$; therefore the countable system $\{F_n\times F_{n'}': (n,n')\in \mathbb{N}^2\}$ does the trick. 

The measurabilities asserted in \eqref{product:ii} follow by a localization--$\pi$-$\lambda$--monotone-class argument. 

Then, for the existence part in \eqref{product:i}, we may simply set $$\nu(M):=\int\int \mathbbm{1}_M(x,x')\mu(\dd x)\mu'(\dd x'),\quad M\in \FF\otimes \FF'.$$ Nothing was to prevent us of course from having taken for $\nu$ instead 
$$\nu'(M):=\int\int \mathbbm{1}_M(x,x')\mu'(\dd x')\mu(\dd x),\quad M\in \FF\otimes \FF';$$ hence by uniqueness of $\nu$ [we get $\nu'=\nu$ and] \eqref{interchange} follows for $f$ which are indicators of sets from $\FF\otimes \FF'$. By a monotone class argument it extends at once to $(\FF\otimes \FF')/\mathcal{B}_{[0,\infty]}$, which is Tonelli's result. The extension to \eqref{product:ii:c}, hence \eqref{product:ii:b},  is by applying \eqref{product:ii:a} to $f^+$ and $f^-$ and then subtracting the results.
\end{proof}
 	\begin{remarks}
(i) 	Even if $\mu$ and $\mu'$ are complete, $\mu\times \mu'$ need not be; this happens for instance with $\mu=\mu'=\overline{\leb}$. The completion of $\mu\times \mu'$ is given by Carath\'eodory's theorem when applied to $\mu\times\mu'$. There are versions of Tonelli-Fubini for $\overline{\mu\times \mu'}$. (ii) A generalization of product measures deals with kernels; such an approach has also a kind-of converse in the so-called disintegration of measures, but these are beyond the intended scope of these notes (though, it is a very relevant topic with implications for conditioning in probability; we touch breifly on it in Proposition~\ref{proposition:regular}). (iii) 
We have seen in the preceding the construction of the product of two (and, by induction, any finite number of) $\sigma$-finite measure spaces and we have already remarked that for probability spaces a notion of product extends to arbitrary families thereof. Is there a corresponding operation of ``sum'' of measure spaces? Yes, however it is much more straightforward. Let indeed $(X_\alpha,\Sigma_\alpha,\mu_\alpha)$, $\alpha\in \mathcal{A}$, be a family of measure spaces. Put $X:=\sum_{\alpha\in \mathcal{A}}X_\alpha:=\cup_{\alpha\in \mathcal{A}}X_\alpha\times \{\alpha\}$ for the disjoint union. For each $\alpha\in \mathcal{A}$ let $i_\alpha:X_\alpha\to \sum_{\alpha\in \mathcal{A}}X_\alpha$ be the canonical inclusion map (it sends $x_\alpha\in X_\alpha$ to $(x_\alpha,\alpha)$). We define $\Sigma:=\oplus_{\alpha\in \mathcal{A}}\Sigma_\alpha$ as the largest $\sigma$-field on $X$ w.r.t. which each $i_\alpha$ is $\Sigma_\alpha$-measurable, $\alpha\in \mathcal{A}$ (think about how you can describe this $\sigma$-field; you will then be able to note that for a measurable space $(Y,\mathcal{Y})$ and a map $f:\sum_{\alpha\in \mathcal{A}}X_\alpha\to Y$, one has $f\in \oplus_{\alpha\in \mathcal{A}}\Sigma_\alpha/\mathcal{Y}$ iff $f\circ \iota_\alpha\in \mathcal{X}_\alpha/\mathcal{Y}$ for all $\alpha\in\mathcal{A}$, which you might later compare with Proposition~\ref{proposition:arbitrary-products}\eqref{product--arbit:ii}). Then define $\mu=:+_{\alpha\in \mathcal{A}}\mu_\alpha:\Sigma\to [0,\infty]$ by putting $\mu(A):=\sum_{\alpha\in \mathcal{A}}\mu_\alpha(i_\alpha^{-1}(A))=\int \mu_\alpha(i_\alpha^{-1}(A))c_\mathcal{A}(\dd\alpha)$ for $A\in \Sigma$; $(X,\Sigma,\mu)$ is a measure space, the direct sum of the measure spaces $(X_\alpha,\Sigma_\alpha,\mu_\alpha)$, $\alpha\in \mathcal{A}$. Basically, on each ``plot'' $X_\alpha\times \{\alpha\}$ ``lives'' $\mu_\alpha$. If $(\Omega,\FF,\gamma)$ is a given measure space and $\Omega$ is the disjoint  union of  sets $\Omega_\alpha\in \FF$ as  $\alpha$ runs over some countable index set $\mathcal{A}$, then the direct sum of $(\Omega_\alpha,\FF\vert_{\Omega_\alpha},\gamma_{\Omega_\alpha})$ recovers  $(\Omega,\FF,\gamma)$ up to the ``identification'' map that sends each $\omega\in\Omega_\alpha$ to $(\omega,\alpha)$ for $\alpha\in \mathcal{A}$, and which ``preserves everything in sight''.
 	 \end{remarks}

\begin{example}\label{example:mean-to-lebesgue}
Let $f:[0,\infty)\to[0,\infty)$ be continuously differentiable with nonnegative derivative on $(0,\infty)$, also let it be continuous on $[0,\infty)$ with $f(0)=0$ (e.g. $([0,\infty)\ni x\mapsto (1-e^{-\lambda x}))$ for a given $\lambda\in [0,\infty)$, the $p$-th power for a $p\in(0,\infty)$, are all possible choice), and let $\nu$ be a $ \sigma$-finite measure on $\mathcal{B}_{[0,\infty)}$. Then we may compute $\nu[f]=\int \int_0^xf'(y)\dd y\nu(\dd x)=\int \int \mathbbm{1}_{(0,x)}(y)f'(y)\leb(\dd y)\nu(\dd x)=\int \int \mathbbm{1}_{(0,x)}(y)f'(y)\nu(\dd x)\leb(\dd y)=\int \mathbbm{1}_{(0,\infty)}(y)f'(y)\int \mathbbm{1}_{(y,\infty)}(x)\nu(\dd x)\leb(\dd y)=\int_{(0,\infty)}f'(y)\nu((y,\infty))\leb(\dd y)$. In particular, for a probability $\PP$ on a $\sigma$-field $\FF$ and an $X\in \FF/\BB_{[0,\infty)}$ (a nonnegative random variable in the terminology to follow) $\PP[f(X)]=(X_\star \PP)[f]=\int_{(0,\infty)}f'(y)(X_\star\PP)((y,\infty))\leb(\dd y)=\int_{(0,\infty)}f'(y)\PP(X>y)\leb(\dd y)$; especially $\PP[X]=\PP[\mathrm{id}_{[0,\infty)}(X)]=\int_{(0,\infty)}\PP(X>y)\leb(\dd y)$.

In lieu of the conditions on $f$ stated above it suffices that $f:[0,\infty)\to [0,\infty]$ is given as the indefinite integral $f= \int_{(0,\cdot)} g\dd\leb$ of some $g\in \mathcal{B}_{(0,\infty)}/\mathcal{B}_{[0,\infty]}$ (with the same argumentation, $g$ replacing the role of $f$). The $\sigma$-finiteness condition on $\nu$ can be dispensed with (as a little more thought reveals). Besides, the improper Riemann-Darboux integral may replace integration against $\leb$ in the preceding when the integrands are finite, provided such $g$ is finite and continuous a.e. w.r.t. $\overline{\mathscr{L}_{(0,\infty)}}$ and provided $\nu((y,\infty))$ is finite for all $y\in (0,\infty)$. 
 
\end{example}

\begin{example}
We have that $\pr_i\in\mathcal{B}_{\mathbb{R}^2}/\mathcal{B}_\mathbb{R}$, $i\in \{1,2\}$; therefore  $\pr_1-\pr_2\in \mathcal{B}_{\mathbb{R}^2}/\mathcal{B}_\mathbb{R}$. Hence $D_\mathbb{R}:=\{(x,x):x\in \mathbb{R}\}=(\pr_1-\pr_2)^{-1}(\{0\})\in \mathcal{B}_{\mathbb{R}^2}\subset  \mathcal{B}_{[-\infty,\infty]}\otimes \mathcal{B}_{[-\infty,\infty]}$. Consequently also $D_{[-\infty,\infty]}:=\{(x,x):x\in [-\infty,\infty]\}=D_\mathbb{R}\cup \{(-\infty,-\infty)\}\cup\{(\infty,\infty)\}=D_\mathbb{R}\cup (\{-\infty\}\times\{-\infty\})\cup(\{\infty\}\times\{\infty\})\in \mathcal{B}_{[-\infty,\infty]}\otimes \mathcal{B}_{[-\infty,\infty]}$.
\end{example}
\begin{proposition}\label{proposition:uniquness-image-meausure}
	Let $(\Omega,\FF,\mu)$ be a measure space, $(\Omega',\FF')$ a measurable space, $X\in \FF/\FF'$, $(A,\AA)$ another measurable space with $D_A:=\{(x,x):x\in A\}\in \AA\otimes \AA$  and $\{f,g\}\subset \FF'/\AA$. Then $f(X)=g(X)$ a.e.-$\mu$ iff $f=g$ a.e.-$X_\star\mu$.
\end{proposition}
 
\begin{remark}
For a $\sigma$-field  $\FF$ on $\Omega$, we have as follows. On the one hand, if there is a countable $\AA\subset \FF$, which separates the points of $\Omega$ (we say that $\FF$ is countably separated), then (the diagonal of $\Omega)=D_\Omega=\Omega\backslash \cup_{A\in \AA}[(A\times (\Omega\backslash A))\cup (\Omega\backslash A)\times A]\in \FF\otimes \FF$. Conversely, suppose $D_\Omega\in \FF\otimes \FF$. Then (think, why) $D_\Omega\in \FF^0\otimes \FF^0$ for a sub-$\sigma$-field $\FF^0$ of $\FF$ that is countably generated (it means existence of a countable $\CC\subset \FF$ for which $\sigma_\Omega(\CC)=\FF^0$) and $\FF^0$, therefore $\CC$, must separate the points of $\Omega$. For if this fails, then there are $x\ne y$ from $\Omega$ such that $\mathbbm{1}_F(x)=\mathbbm{1}_F(y)$ for all $F\in \FF^0$. Consequently, $\FF^0\otimes \FF^0\subset \{R\in 2^{X\times X}:R\text{ contains  all or none of $(x,x)$, $(y,y)$, $(x,y)$, $(y,x)$}\}\notni D_\Omega$, which is a contradiction. In short, $D_ \Omega\in \FF\otimes \FF$ if and only if $\FF$ is countably separated.

We may also mention that when $X$ has cardinality greater than that of the continuum, then $2^X\otimes 2^X$ does not contain the diagonal $D_X$ and in particular $2^X\otimes 2^X\subsetneq 2^{X\times X}$ (\underline{\href{https://www.drmaciver.com/2006/04/journal-of-obscure-results-1-nedomas-pathology/}{Nedoma's pathology}}).
\end{remark}
 
\begin{proof}
$\mu(f(X)\ne g(X))=\mu(X\in \{f\ne g\})=X_\star \mu(f\ne g)$, because $(f,g)\in \FF'/(\AA\otimes \AA)$ and hence $\{f\ne g\}=(f,g)^{-1}((A\times A)\backslash D_A)\in \FF'$.
\end{proof}
	\section{Indefinite integration and absolute continuity}
	\emph{indefinite integration and absolute continuity; integrable measurable functions are determined a.e. by their indefinite integrals; the Radon-Nikodym theorem, derivative of one measure w.r.t. another}
	
%

\begin{definition}\label{definition:indefinite-integral}
Let $(X,\AA,\mu)$ be a measure space, $f\in \AA/\mathcal{B}_{[-\infty,\infty]}$ with the integral of $f$ against $\mu$ well-defined. The map $f\cdot \mu:=(\AA\ni A\mapsto \mu[f;A])$ is called the $\mu$-indefinite integral of $f$ or also the indefinite integral of $f$ against $\mu$. 
\end{definition}

	\begin{definition}
		Given two measures $\mu$ and $\nu$ on a $\sigma$-field $\FF$ we say that: $\mu$ is absolutely continuous w.r.t. $\nu$ (and write $\mu\ll \nu$) \deff for all $A\in \FF$, $\mu(A)=0$ whenever $\nu(A)=0$; $\mu$ and $\nu$ are equivalent (and write $\mu\sim\nu$) \deff $\mu\ll\nu$ and $\nu\ll\mu$. 
	\end{definition}
 
\begin{remark}
It may be interesting to note in passing that absolute continuity ($\therefore$ equivalence) is preserved under push-forwards, i.e. $\mu\ll\nu$ implies $f_\star\mu\ll f_\star\nu$, which is trivial, and also under products, i.e. $\mu_1\ll\mu_2$ and $\nu_1\ll\nu_2$, all $\sigma$-finite, implies $\mu_1\times \nu_1\ll \mu_2\times \nu_2$, which follows via Tonelli. 
\end{remark}
 
\begin{proposition}\label{proposition:density}
		Let $(\Omega,\FF,\mu)$ be a measure space and $f\in \FF/\mathcal{B}_{[0,\infty]}$. Then $f\cdot \mu$ is a measure on $\FF$ that is absolutely continuous w.r.t. $\mu$; furthermore, 
		\begin{equation}\label{eq:density}
		\int g\dd (f\cdot \mu)=\int gf\dd \mu
		\end{equation} 
	for all $g\in \FF/\mathcal{B}_{[-\infty,\infty]}$, the integral on the left-hand side of \eqref{eq:density} being well-defined iff the integral on the right-hand side of \eqref{eq:density} is well-defined, in which case (associativity of indefinite integration) $g\cdot (f\cdot \mu)=(gf)\cdot \mu$. If $f>0$ a.e.-$\mu$, then $f\cdot \mu\sim \mu$.
\end{proposition}
 
\begin{remark}
Note also the formula $f\cdot (h_\star \mu)=h_\star ((f\circ h)\cdot \mu)$ under the ``obvious'' conditions on the maps $f,h$ and measure $\mu$. 
\end{remark}
 
\begin{proof}
The observation that $f\cdot \mu$ is a measure on $\FF$ follows from the fact that integrals and nonnegative series can be interchanged (among more trivial considerations); absolute continuity is a consequence of the fact that integrals do not see sets of measure zero (again among more trivial considerations). Formula~\eqref{eq:density} is true for $g=\mathbbm{1}_A$ with $A\in \FF$ by the very definition of $f\cdot \mu$. It extends to $\FF/\mathcal{B}_{[0,\infty]}$ by a monotone class argument. It follows in full generality by applying it to $g^+$ and $g^-$; then taking the difference. To see $ g\cdot (f\cdot \mu)=(gf)\cdot \mu$ just apply \eqref{eq:density} to $g\mathbbm{1}_A$ in lieu of $g$ as $A$ runs over $\FF$.  The final claim is immediate.
\end{proof}
\begin{example}\label{lemma:cts/diff/density}
Suppose $G:\mathbb{R}\to \mathbb{R}$ is $\uparrow$ and right-continuous. If $G$ is continuously differentiable, then by the fundamental theorem of calculus $(G'\cdot \leb)((a,b])=\leb[G';(a,b]]=G(b)-G(a)$ for all real $a\leq b$; it follows that $G'\cdot \leb=\dd G$, therefore $\dd G\ll \leb$ and $\dd G[g]=\leb[gG']$ for all $g\in \mathcal{B}_{\mathbb{R}}/\mathcal{B}_{[-\infty,\infty]}$, the integral on the l.h.s. being well-defined iff the one on the r.h.s.  is.
\end{example}
	\begin{remark}
	There is a converse to Proposition~\ref{proposition:density}, whose proof is non-trivial and will be given below as Theorem~\ref{thm:RN}. In preparation thereof, Corollary~\ref{corollary} to follow is  a kind of reverse to Theorem~\ref{theorem:int:basics}\eqref{properties:vi}. We give a  more general statement beforehand.
\end{remark}

\begin{proposition}\label{lemma}
Let $(X,\AA,\mu)$ be a measure space, $\{f,g\}\subset \AA/\mathcal{B}_{[-\infty,\infty]}$. 
\begin{enumerate}[(a)]
\item\label{lemma:a} Suppose $\int_Af^+\dd\mu\lor \int_A g^-\dd\mu<\infty$ and  $\int_{A} f\dd\mu\leq \int_{A}g\dd\mu$ for all $A\in \AA$ (or just for $A=\{f>g\}$). Then $f\leq g$ a.e.-$\mu$. 
\item \label{lemma:b} Suppose $\mu$ is $\sigma$-finite, $(\int f^-\dd\mu\land \int f^+\dd\mu)\lor (\int g^-\dd\mu\land \int g^+\dd\mu)<\infty$, and $\int_A f\dd\mu\leq \int_A g\dd\mu$ for all $A\in \AA$. Then again $f\leq g$ a.e.-$\mu$. 
\end{enumerate}
\end{proposition}
\begin{example}
The finiteness condition in \eqref{lemma:a} of Proposition~\ref{lemma} cannot entirely be dispensed with (in favor of just demanding that the integrals be well-defined). Take for instance $X=\{0\}$, $\mu(\{0\})=\infty$. Then $\int_X c d\mu=(\mathbbm{1}_{(0,\infty]}(c)-\mathbbm{1}_{[-\infty,0)}(c))\infty$ and (of course) $\int_\emptyset cd\mu=0$ for all $c\in [-\infty,\infty]$. (In view of  Proposition~\ref{lemma}\eqref{lemma:b} this means that $\mu$ cannot be $\sigma$-finite (as it is not).)
\end{example}
 
 \begin{proof}
\eqref{lemma:a}. One has that $\int_{\{f>g\}} (f-g)\dd\mu$ is $\geq 0$ and $\leq 0$ hence $=0$ (by the assumed finiteness condition, $\int_{\{f>g\}} (f-g)\dd\mu=\int_{\{f>g\}} f\dd\mu+\int_{\{f>g\}} (-g)\dd\mu=\int_{\{f>g\}} f\dd\mu-\int_{\{f>g\}} g\dd\mu$). Then $(f-g)\mathbbm{1}_{\{f>g\}}=0$ a.e.-$\mu$ which forces $f\leq g$ a.e.-$\mu$. \eqref{lemma:b}. By the $\sigma$-finiteness assumption, there is a nondecreasing sequence $(A_n)_{n\in \mathbb{N}}$ in $\AA$ with $\mu(A_n)<\infty$ for each $n\in \mathbb{N}$ whose union covers $X$. Then by \eqref{lemma:a} we find that $(f\land n)\mathbbm{1}_{A_n}\leq (g\lor (-n))\mathbbm{1}_{A_n}$ a.e.-$\mu$ for each $n\in \mathbb{N}$. Letting $n\uparrow \infty$ yields the desired conclusion.
\end{proof}
 
\begin{corollary}\label{corollary}
Let $(X,\AA,\mu)$ be a measure space, $\{f,g\}\subset \AA/\mathcal{B}_{[-\infty,\infty]}$. Suppose $\int_A f\dd\mu= \int_Ag\dd\mu$ for every $A\in \AA$, all the integrals being well-defined. If  $f$ or $g$ is $\mu$-integrable (in which case both are) or if $\mu$ is $\sigma$-finite, then $f=g$ a.e.-$\mu$. Furthermore, in case $f$ and $g$ are $\mu$-integrable, then actually, ceteris paribus, it is enough to ask for  $\int_A f\dd\mu= \int_Ag\dd\mu$ to hold true merely for every  $A\in \BB\cup \{X\}$, where $\BB\subset \AA$ is a $\pi$-system that generates $\AA$ on $X$, and still  the same conclusion prevails.
\end{corollary}
\begin{proof}
If $f$ and $g$ are $\mu$-integrable, then by Dynkin's lemma the equality  $\int_A f\dd\mu= \int_Ag\dd\mu$ extends from $A\in \BB\cup \{X\}$ at once to all $A\in \AA$. 
\end{proof}

\begin{remarks}
So, a measurable numerical integrable function, or a measurable numerical function whose integral against a $\sigma$-finite measure is well-defined, is determined a.e. by its indefinite integral. Of course in Corollary~\ref{corollary}, $\BB\cup \{X\}$ is itself a $\pi$-system, so it would have been equally well to ask for the property to hold on a $\pi$-system $\BB\in 2^\AA$ that has $X$ for an element. 
\end{remarks}
	\begin{theorem}[Radon-Nikodym]\label{thm:RN}
		Suppose $\mu\ll \nu$ are $\sigma$-finite measures defined on a measurable space $(X,\FF)$. Then there exists a $\nu$-a.e. unique $f\in  \FF/\mathcal{B}_{[0,\infty)}$  such that $\mu=f\cdot \nu$; furthermore, $f>0$ a.e.-$\mu$. $\blacktriangle$
	\end{theorem}
	\begin{proof}[Proof of uniqueness and last part] Apply Corollary~\ref{corollary}. Also, $\mu(f=0)=\nu[f;f=0]=0$.
	\end{proof}
	 
		\begin{proof}[Proof of existence] 
By $\sigma$-finiteness and ``patching together'' one reduces at once to the case when $\mu$ and $\nu$ are finite, which we henceforth assume is the case; also, it will be enough to find an $f\in  \FF/\mathcal{B}_{[0,\infty]}$  with $\mu=f\cdot \nu$.	
	
Let $F:=\{f\in \FF/\mathcal{B}_{[0,\infty]}:f\cdot\nu\leq \mu\}$. We note that $0\in F$, so $F$ is non-empty; also, if $\{f_1,f_2\}\subset F$, then $f_1\lor f_2\in F$: if $A\in \FF$, then $((f_1\lor f_2)\cdot \nu)(A)=((f_1\lor f_2)\cdot \nu)(A\cap \{f_1\leq f_2\})+((f_1\lor f_2)\cdot \nu)(A\cap \{f_1> f_2\})\leq (f_2\cdot \nu)(A\cap \{f_1\leq f_2\})+(f_1\cdot \nu)(A\cap \{f_1> f_2\})\leq \mu(A\cap \{f_1\leq f_2\})+\mu(A\cap \{f_1> f_2\})=\mu(A)$. Now let $(f_n)_{n\in\mathbb{N}}$ be a sequence in $F$ such that $\lim_{n\to\infty}\nu[f_n]=\sup_{f\in F}\nu[f]$, which we remark is $<\infty$, indeed it is $\leq \mu(X)$. We may assume the sequence $(f_n)_{n\in\mathbb{N}}$ is $\uparrow$; set $f:=\lim_{n\to\infty}f_n$. By monotone convergence $f\cdot \nu\leq \mu$ and $\nu[f]=\sup_{g\in F}\nu[g]$. Assume per abusurdum that $f\cdot \nu\ne \mu$, so that there is $A\in \FF$ with $(f\cdot \nu)(A)<\mu(A)$, hence also $(f\cdot \nu)(X)<\mu(X)$. By absolute continuity necessarily $\nu(X)>0$; let $\epsilon:=\frac{\mu(X)-(f\cdot\nu)(X)}{2\nu(X)}$. Consider the map $\gamma:=\mu-(f+\epsilon)\cdot \nu:\FF\to \mathbb{R}$. We see that $\gamma(X)=\epsilon>0$; let $\alpha:=\sup\{\gamma(A):A\in \FF\}\in (0,\infty)$. For each $n\in \mathbb{N}$, let $A_n\in \FF$ be such that $\gamma(A_n)\geq \alpha-2^{-n}$ and put $A:=\limsup_{n\to\infty}A_n$. We see that, for all $N\in \mathbb{N}$ and then all $n\in \mathbb{N}_{\geq N}$, one has $\gamma(A_n\backslash \cup_{m\in \mathbb{N}_{\geq N,<n}}A_m)\geq -2^{-n}$, since $\alpha-2^{-n}\leq \gamma(A_n)=\gamma((A_n\backslash \cup_{m\in \mathbb{N}_{\geq N,<n}}A_m)\cup \cup_{m\in \mathbb{N}_{\geq N,<n}}A_m)=\gamma(\cup_{m\in \mathbb{N}_{\geq N,<n}}A_m) +\gamma(A_n\backslash \cup_{m\in \mathbb{N}_{\geq N,<n}}A_m) \leq \alpha+\gamma(A_n\backslash \cup_{m\in \mathbb{N}_{\geq N,<n}}A_m)$. In consequence, by continuity from below and countable additivity, $\gamma(A)=\lim_{N\to\infty}\gamma(\cup_{n\in \mathbb{N}_{\geq N}}A_n)\geq \limsup_{N\to\infty} \alpha-2^{-N}-2^{-N}=\alpha$. This means that $\gamma(A)=\alpha>0$ and so: $\gamma(B)\geq 0$ for all $B\in \FF\cap 2^A$ (otherwise $\gamma(A\backslash B)=\gamma(A)-\gamma(B)>\alpha$); $\mu(A)>0$, hence by absolute continuity $\nu(A)>0$. Therefore $f+\epsilon\mathbbm{1}_A\in F$ and $\nu[f+\epsilon\mathbbm{1}_A]>\nu[f]$, which is a contradiction.
	\end{proof}
	 
	\begin{definition}
		The $f$ from the preceding theorem is denoted $\frac{\dd \mu}{\dd\nu}$ and called the Radon-Nikodym derivative (also, density) of $\mu$ w.r.t. $\nu$. Statements involving it are asserted for any of its versions (as a result they tend to attract the a.e. qualifier).\footnote{Alternatively, to circumvent the issue of the derivatives being defined only a.e. precisely, one could pass (passes) to their equivalence classes w.r.t. a.e. equality, but it seems an even bigger nuissance than having to live with a.e. qualifiers.}
		  Sometimes we write $\frac{\mu(\dd x)}{\nu(\dd x)}$ when we wish to provide an expression for $\frac{\dd\mu}{\dd \nu}$ at the point $x$. 
		 
	\end{definition}
\begin{example}\label{lemma:cts/diff/density-1}
In the context of Example~\ref{lemma:cts/diff/density} we have $\frac{\dd (\dd G)}{\dd \leb}=G'$ a.e.-$\leb$, provided $G$ is continuously differentiable.
\end{example}
\begin{example}
The $\sigma$-finiteness condition in Theorem~\ref{thm:RN} cannot be entirely dropped. For instance $\leb\ll c_\mathbb{R}\vert_{\mathcal{B}_\mathbb{R}}$, however there is no $f\in \FF/\mathcal{B}_{[0,\infty]}$  for which  $\leb=f\cdot c_\mathbb{R}\vert_{\mathcal{B}_\mathbb{R}}$ because the latter entails that $0=\leb(\{x\})=(f\cdot c_\mathbb{R}\vert_{\mathcal{B}_\mathbb{R}})(\{x\})=f(x)$ for all $x\in \mathbb{R}$. In particular, since $\leb$ is, so $c_\mathbb{R}\vert_{\mathcal{B}_\mathbb{R}}$ cannot be $\sigma$-finite (which is anyway obvious by other means, even $c_\mathbb{R}$ is not $\sigma$-finite).
\end{example}

\begin{corollary}
Suppose $\mu\ll \nu\ll\lambda$ are $\sigma$-finite measures on a $\sigma$-field $\FF$. Then $\frac{\dd \mu}{\dd \lambda}=\frac{\dd \mu}{\dd \nu}\frac{\dd \nu}{\dd \lambda}$ a.e.-$\lambda$. In particular, if $\mu\sim\nu$ then $1=\frac{\dd \mu}{\dd \nu}\frac{\dd \nu}{\dd \mu}$ a.e.-$\nu$ and a.e.-$\mu$.

  If further $\mu$ is finite, then we have also the following assertion: for every $\epsilon\in (0,\infty)$ there exists a $\delta\in (0,\infty)$ such that $\mu(A)\leq \epsilon$ whenever $\nu(A)\leq \delta$, $A\in \FF$. 
\end{corollary}
\begin{remarks}
Clearly if the preceding epsilon-delta property holds of two measures $\mu$ and $\nu$ defined on a $\sigma$-field $\FF$, then $\mu\ll\nu$, so one has actually a characterization of absolute continuity (under the assumption that $\mu$ is finite). Also, the finiteness assumption cannot be altogether dispensed with: $\leb_{(1,\infty)}\ll ((1,\infty)\ni x\mapsto x^{-2})\cdot \leb\vert_{(1,\infty)}$ (there is even equivalence), however the epsilon-delta property fails for this pair of $\sigma$-finite measures (of which the second is even finite).

The remark immediately following Proposition~\ref{proposition:density} shows that for $\sigma$-finite measures $\mu\ll\nu$ defined on a $\sigma$-field $\AA$ and for an $\AA$/$\BB$-measurable injective $g$ such that $g^{-1}\in \BB/\AA$ (for a ``bimeasurable bijection'' $g$) the equality 
$$g_\star \left[\left(\frac{\dd g_\star\mu}{\dd g_\star \nu}\circ g\right)\cdot\nu\right]=\frac{\dd g_\star\mu}{ \dd g_\star\nu}\cdot g_\star\nu=g_\star\mu$$ holds true; pushing forward along $g^{-1}$ we get the (note-worthy) formula $$\frac{\dd g_\star\mu}{\dd g_\star \nu}\circ g=\frac{\dd \mu}{\dd\nu}\text{ a.e.-$\nu$}.$$
\end{remarks} 
\begin{proof}
Clearly $\mu\ll \lambda$. For the first part just note that by Proposition~\ref{proposition:density} $(\frac{\dd \mu}{\dd \nu}\frac{\dd \nu}{\dd \lambda})\cdot \lambda=\frac{\dd \mu}{\dd \nu}\cdot (\frac{\dd \nu}{\dd \lambda}\cdot \lambda)=\frac{\dd \mu}{\dd \nu}\cdot \nu=\mu$.  
By definition it follows that  $\frac{\dd \mu}{\dd \lambda}=\frac{\dd \mu}{\dd \nu}\frac{\dd \nu}{\dd \lambda}$. 

The second claim follows on specializing the first assertion to $\lambda=\mu$, noting that trivially  $\frac{\dd \mu}{\dd\mu}=1$ a.e.-$\mu$. 

  For the last part, there is an $N\in \mathbb{N}$ such that $\nu(\frac{\dd\mu}{\dd\nu}> N)<\epsilon/2$ (by dominated convergence in $\nu(\frac{\dd\mu}{\dd\nu}> N)=\int_{\{\frac{\dd\mu}{\dd\nu}> N\}}\frac{\dd\mu}{\dd\nu}\dd\nu$, because the $\nu$-integral of $\frac{\dd\mu}{\dd\nu}$ is finite). Take $\delta=\epsilon/(2N)$. If $A\in \FF$ and $\nu(A)\leq \delta$ it follows that $\mu(A)=\int_A\frac{\dd\mu}{\dd\nu}d\nu=\int_{A\cap \{\frac{\dd\mu}{\dd\nu}> N\}}\frac{\dd\mu}{\dd\nu}d\nu+\int_{A\cap \{\frac{\dd\mu}{\dd\nu}\leq N\}}\frac{\dd\mu}{\dd\nu}d\nu\leq \int_{\{\frac{\dd\mu}{\dd\nu}> N\}}\frac{\dd\mu}{\dd\nu}d\nu+\int_{A\cap \{\frac{\dd\mu}{\dd\nu}\leq N\}}Nd\nu= \nu(\frac{\dd\mu}{\dd\nu}> N)+N\nu(A\cap \{\frac{\dd\mu}{\dd\nu}\leq N\})\leq \frac{\epsilon}{2}+N\nu(A)\leq \epsilon$.  
\end{proof}

\begin{example}[Change of variables for integrals against Lebesgue measure]
	In the context of Examples~\ref{lemma:cts/diff/density}  and~\ref{lemma:cts/diff/density-1} assume $\dd G \ll     \leb$ and denote $G':=\frac{\dd (\dd G)}{\dd \leb}$ (whether or not $G$ is continuously differentiable); we say that $G$ is absolutely continuous.  Suppose furthermore that $G'>0$ a.e.-$\leb$ so that $G$ maps $\mathbb{R}$ continuously and strictly increasingly ($\therefore$ bijectively) onto  $(G(-\infty),G(\infty))$ (where $G(\infty):=\lim_\infty G$ and likewise for $G(-\infty)$). Changing $G'$ on an $\leb$-negligible set only we may and do assume that $G'$ is strictly positive.  Unraveling the definitions we then see, using \eqref{eq:image}, that $((G'\circ G^{-1})\cdot (G_\star \leb))((a,b])=\leb[G';(G^{-1}(a),G^{-1}(b)]]=\dd G((G^{-1}(a),G^{-1}(b)])=b-a=\leb_{(G(-\infty),G(\infty))}((a,b])$ for all real $a\leq b$ from $(G(-\infty),G(\infty))$, and we conclude via Proposition~\ref{proposition:equality of measures} (or directly from the definition of $\leb=\dd(\mathrm{id}_\mathbb{R})$, but the latter itself rests on Proposition~\ref{proposition:equality of measures}) that $(G'\circ G^{-1})\cdot (G_\star \leb)=\leb_{(G(-\infty),G(\infty))}$, i.e. $G_\star \leb\sim \leb_{(G(-\infty),G(\infty))}$ and 
	\begin{equation}
	\frac{\dd (G_\star \leb)}{\dd (\leb_{(G(-\infty),G(\infty))})}=\frac{1}{G'\circ G^{-1}}
	\end{equation} a.e.-$\leb_{(G(-\infty),G(\infty))}$. Here in taking the push-forward $G_\star \leb$ we have viewed $G$ as mapping into $(G(-\infty),G(\infty))$ endowed with its Borel $\sigma$-field. Extensions to the case when $G$ is defined only on some open interval are ``automatic''. 
\end{example}
\begin{remark}
Variants and extensions of the preceding in the context of higher-dimensional Lebesgue measures  involve the absolute value of the determinant of the Jacobian matrix of $G$ (in lieu of $G'$).

For instance, let $G:O\to V$ be a $C^1$ bijection between open subset $O$ and $V$ of $\mathbb{R}^n$, $n\in \mathbb{N}$, such that $\det J(G)\ne 0$ throughout $O$. Suppose $\mu$ is a $\sigma$-finite measure on $\mathcal{B}_O$ absolutely continuous w.r.t. $\leb^n_O$. Then $G_\star \mu$ is absolutely continuous w.r.t.  $\leb^n_V$ and $\frac{\dd (G_\star \mu)}{\dd \leb^n_V}=\frac{\frac{\dd\mu}{\dd \leb^n_O}\circ G^{-1}}{\vert \det J(G)\vert\circ G^{-1}}$ a.e.-$\leb^n_V$.  Indeed, for $A\in \mathcal{B}_V$, $(G_\star \mu)(A)=\mu(G^{-1}(A))=\int_{G^{-1}(A)} \frac{\dd\mu}{\dd \leb^n_O}\dd \leb^n=\int_A \frac{\dd\mu}{\dd \leb^n_O}\circ G^{-1}\vert \det J(G^{-1})\vert\dd\leb^n=\int_A \frac{\frac{\dd\mu}{\dd \leb^n_O}\circ G^{-1}}{\vert \det J(G)\vert\circ G^{-1}}\dd\leb^n$ by change of variables (maybe we know it from calculus only for $A=$ bounded rectangle  and for $1$ in lieu of $\frac{\dd\mu}{\dd \leb^n_O}$ but it extends at once by the usual arguments of monotone class).  We may remark that if $\mu$ is the law of an $O$-valued $X$ under a measure $\nu$, then $G_\star\mu=G_\star (X_\star \nu)=(G\circ X)_\star \nu$ is the law of $G( X)$ under $\nu$.
 
\end{remark}

\section{L-spaces and some integral inequalities}
\emph{the inequalities of Markov, Jensen,  H\"older, Cauchy-Schwartz and Minkowski; the spaces $\mathcal{L}^p$ for $p\in [1,\infty]$}

\begin{definition}
	Let $(\Omega,\FF,\mu)$ be a measure space, $p\in [1,\infty)$ and $f\in \FF/\mathcal{B}_{[-\infty,\infty]}$. We put $$\Vert f\Vert_{p_\mu}:=\left(\int \vert f\vert^p\dd \mu\right)^{1/p}$$ and set $\LL^p(\mu):=\{g\in \FF/\mathcal{B}_{\mathbb{R}}:\Vert g\Vert_{p_\mu}<\infty\}$ (this has already been defined for $p=1$); in addition we set $\Vert f\Vert_{\infty_\mu}:=\inf \{M\in [0,\infty]:\vert f\vert\leq M\text{ a.e.-$\mu$}\}$ and put $\LL^\infty(\mu):=\{g\in \FF/\mathcal{B}_{\mathbb{R}}:\Vert g\Vert_{\mu_\infty}<\infty\}$. For $q\in [1,\infty]$, a sequence  $(f_n)_{n\in \mathbb{N}}$ in $\LL^q(\mu)$ and $f_0\in \LL^q(\mu)$: $\lim_{n\to\infty}f_n=f_0$ in $\LL^q(\mu)$ \deff $\Vert f_n-f_0\Vert_{q_\mu}\to 0$ as $n\to\infty$, in which case $f_0$ is  called the $\LL^q(\mu)$-limit of $(f_n)_{n\in\mathbb{N}}$. One also defines  $\langle f,g\rangle_\mu:=\mu[fg]$ for $\{f,g\}\subset\LL^2(\mu)$. 
\end{definition}

\begin{remarks}
	One tends to omit $\mu$ in the notation if it can be gathered from context. We understand $\infty^p:=\infty$ and $\infty^{1/p}:=\infty$, of course ($p\in [1,\infty)$). It is equally possible to introduce the L-spaces for complex-valued functions (but we shall keep to the real case). Convergence in $\LL^\infty(\mu)$ corresponds to uniform convergence a.e.-$\mu$. The set of which the infimum is being taken in the definition of $\Vert f\Vert_{\infty_\mu}$ is closed and non-empty, i.e. $\vert f\vert\leq \Vert f\Vert_{\infty_\mu}$ a.e.-$\mu$.
	\end{remarks}
\begin{example}
Taking $\mu=c_\mathbb{N}$ in the preceding definition gives the sequence spaces $l^p(\mathbb{R}):=\LL^p(c_\mathbb{N})$, $p\in [1,\infty]$. 
\end{example}
\begin{proposition}
Let $\{p,q\}\subset [1,\infty]$, $p\leq q$, and  let $\mu$ be a finite measure. Then $\LL^q(\mu)\subset \LL^p(\mu)$.
\end{proposition}
\begin{proof}
We note that $([1,\infty)\ni r\mapsto x^r)$ is nondecreasing for $x\in [1,\infty)$. For $q<\infty$, take $f\in \LL^q(\mu)$ and estimate: 
\begin{align*}
\int \vert f\vert^p\dd\mu&=\int_{\{\vert f\vert\geq 1\}} \vert f\vert^p\dd\mu+\int_{\{\vert f\vert< 1\}} \vert f\vert^p\dd\mu\\
&\leq \int_{\{\vert f\vert\geq 1\}} \vert f\vert^q\dd\mu+\int_{\{\vert f\vert< 1\}} 1\dd\mu\\
&\leq \int \vert f\vert^q\dd\mu+\mu(\vert f\vert< 1)<\infty,
\end{align*} 
by monotonicity of the integral and since $\mu$ is finite; therefore $f\in \LL^p(\mu)$. The case $q=\infty$ we treat separately. We may assume $p<\infty$. Then taking $f\in \LL^\infty(\mu)$ we get $\int  \vert f\vert^p\dd\mu\leq \int  \Vert f\Vert_\infty^p\dd\mu<\infty$, since $\vert f\vert\leq  \Vert f\Vert_\infty$ a.e.-$\mu$ and (again) since $\mu$ is finite; hence deduce that $f\in \LL^p(\mu)$.
\end{proof}

\begin{theorem}
	Let $(\Omega,\FF,\mu)$ be a measure space, $\{f,g\}\subset \FF/\mathcal{B}_{[-\infty,\infty]}$. The following inequalities hold true.
	\begin{enumerate}[(i)]
		\item Markov: $\mu[f;f\geq a]\geq a\mu(f\geq a)$ for all $a\in [-\infty,\infty]$; hence, for all $a\in [0,\infty]$, if $f\geq 0$, then  $\mu[f]\geq a\mu(f\geq a)$. 
		\item Minkowski: $\Vert f+g\Vert_p\leq \Vert f\Vert_p+\Vert g\Vert_p$ for all $p\in [1,\infty]$.
		\item H\"older: $\Vert fg\Vert_1\leq \Vert f\Vert_p\Vert g\Vert_q$ whenever $\{p,q\}\subset [1,\infty]$ and $p^{-1}+q^{-1}=1$ (we interpret $\infty^{-1}:=0$). (Special case: Cauchy-Schwartz, when $p=q=2$.)
		\item Jensen: Suppose $\mu$ is a probability, $f\in \LL^1(\mu)$ and $\phi:I\to \mathbb{R}$ is convex, with $I$ an open interval of $\mathbb{R}$ into which $f$ maps. Then $\phi\in \mathcal{B}_I/\mathcal{B}_\mathbb{R}$, $\int (\phi\circ f)^-\dd\mu<\infty$, $\int f\dd\mu\in I$ and 
		\begin{equation}\label{eq:jensen}
\int \phi\circ f\dd\mu\geq \phi\left(\int f\dd\mu\right);
		\end{equation}
in case $\phi$ is strictly convex there is equality in \eqref{eq:jensen} iff $f=\mu[f]$ a.s.-$\mu$.
				\end{enumerate}
	Furthermore, for each $p\in [1,\infty]$, $\Vert\cdot \Vert_p$ is a seminorm on the real linear 
	 space $\LL^p(\mu)$ and the  $\Vert\cdot \Vert_p$-limit of a sequence in $\LL^p$, if it exists, is $\mu$-a.e. unique; it exists iff the sequence is Cauchy in the seminorm $\Vert\cdot\Vert_p$. Finally, $\langle\cdot,\cdot\rangle$ is an inner semiproduct on $\LL^2(\mu)$. 
\end{theorem}
\begin{remark}
	One gets only seminorms and  an inner semiproduct because $\int f\dd\mu=0$ \& $f\geq 0$ implies merely that $f=0$ a.e.-$\mu$, and not everywhere. To get norms and an inner product, one has to pass to the equivalence classes w.r.t. a.e. equality.  The quotient spaces of $\LL^p$ are then denoted $\LLL^p$ ($p\in [1,\infty]$), they become Banach spaces, and $\LLL^2$ a Hilbert space. One also introduces $\LLL^0(\mu)$, the quotient of $\LL^0(\mu):=\FF/\BB_\mathbb{R}$ by a.e.-$\mu$ equality. We do not pursue this further here.
	\end{remark}

\begin{example}\label{example:symmetric-measure-space}
Let $(S,\Sigma,\mu)$ be a $\sigma$-finite measure space, non-atomic in the sense that $\mu(\{x\})=0$ for all $x\in S$. Let $S_s:=\cup_{n\in \mathbb{N}_0}{S\choose n}$. For $n\in \mathbb{N}_0$ set $\Sigma^{\otimes n}:=\underbrace{\Sigma\otimes \cdots\otimes \Sigma}_{\text{$n$-times}}$ ($\Sigma^{\otimes 0}=2^{\{\emptyset\}}$), a $\sigma$-field on $S^n$ ($S^0=\{\emptyset\}$), also $\Delta_n:=\{x\in S^n:\text{the components of $x$ are pairwise distinct}\}$, which we assume belongs to $\Sigma^{\otimes n}$ (it suffices to assume the latter for $n=2$, it then follows automatically for all $n$). 

For $n\in \mathbb{N}_0$ and a  map $f\in \Sigma^{\otimes n}/\mathcal{B}_\mathbb{R}$ put next (the sum is over all bijections $\pi:[\vert\sigma\vert]\to\sigma$) $$f_s(\sigma):=\frac{1}{n!}\sum_{\pi}f(\pi_1,\ldots,\pi_n), \quad \sigma\in {S\choose n}$$ (you see, we symmetrize $f$, only the values of $f$ on $\Delta_n$ matter) and endow $S_s$ with the smallest $\sigma$-field $\Sigma_s$ that makes all such $f_s\mathbbm{1}_{{S\choose n}}$ Borel measurable.  Thus, for all $n\in \mathbb{N}_0$, ${S\choose n}\in \Sigma_s$ and $q_n:=(\Delta_n\ni (x_1,\ldots,x_n)\mapsto \{x_1,\ldots,x_n\}\in {S\choose n})\in \Sigma^{\otimes n}\vert_{\Delta_n}/\Sigma_s$. Also, $F\in \Sigma_s/\mathcal{B}_{[-\infty,\infty]}$ iff for each $n\in \mathbb{N}_0$ there is an $f_n\in\Sigma^n/\mathcal{B}_{[-\infty,\infty]}$ such that $F=\sum_{n\in \mathbb{N}_0}\mathbbm{1}_{{S\choose n}}(f_n)_s$. Another way to describe $\Sigma_s$ is by noting that $\Sigma_s\vert_{S\choose n}=\{q_n(E):E\in \Sigma^n\vert_{\Delta_n},\text{ $E$ symmetric}\}$ for all $n\in \mathbb{N}_0$, i.e. $\Sigma_s$ is the largest $\sigma$-field on $S_s$ w.r.t. which the $q_n$, $n\in \mathbb{N}_0$, are measurable. 

Next, remark that by Tonelli, for each $n\in \mathbb{N}_0$, $\mu^n(S^n\backslash \Delta_n)=0$; here $\mu^n=\underbrace{\mu\times\cdots\times \mu}_{\text{$n$-times}}$ and we interpret $\mu^0$ as $\delta_{\emptyset}$, of course. There is then clearly a unique measure $\mu_s$ on $\Sigma_s$ such that  $$\mu_s[F]=\sum_{n=0}^\infty \frac{1}{n!}\int_{\Delta_n} F(\{x_1,\ldots,x_n\})\mu^{n}(\dd x),\quad F\in \Sigma_s/\mathcal{B}_{[0,\infty]};$$ indeed, $\mu_s=\sum_{n\in \mathbb{N}_0}n!^{-1}(q_n)_\star [(\mu^n)_{\Delta_n}]$. $(S_s,\Sigma_s,\mu_s)$ is the so-called symmetric measure space over $(S,\Sigma,\mu)$.  Take $u\in \LL^2(\mu)$ and define $\mathsf{e}^u(\sigma):=\prod_{x\in \sigma}u(x)$ for $\sigma\in S_s$. Then $\mathsf{e}^u\in \LL^2(\mu_s)$ and $\Vert \mathsf{e}^u\Vert_{2_{\mu_s}}^2=e^{\Vert u\Vert_{2_\mu}^2}$. More generally, $\langle \mathsf{e}^u,\mathsf{e}^v\rangle_{\mu_s}=e^{\langle u,v\rangle_\mu}$ for $\{u,v\}\subset \LL^2(\mu)$. In a sense that can be made precise $L^2(\mu_s)$ is (naturally unitarily isomorphic to) the exponential (a.k.a. symmetric Fock) space of $L^2(\mu)$.
\end{example}

	\begin{proof}[Proof excluding completeness of $\LL^p(\mu)$] The instances corresponding to $\LL^\infty(\mu)$ are trivial and in the following we exclude these.
	
	Markov. $f\mathbbm{1}_{\{f\geq a\}}\geq a\mathbbm{1}_{\{f\geq a\}}$; apply monotonicity of the integral. 
	
	H\"older. If $\Vert f\Vert_p\lor \Vert g\Vert_p=\infty$ or if $\Vert f\Vert_p \Vert g\Vert_p=0$ it is trivial/immediate; assume the converse. Replacing $f$ with $f/\Vert f\Vert_p$ and $g$ with $g/\Vert g\Vert_p$ we may assume that $\Vert f\Vert_p=\Vert g\Vert_p=1$. Young's inequality tells us that for $\{a,b\}\subset (0,\infty)$, $ab\leq \frac{a^p}{p}+\frac{b^q}{q}$ (why is it true?: by concavity of the logarithm $a^\alpha b^\beta\leq \alpha a+\beta b$ if $\{\alpha,\beta\}\subset (0,1)$, $\alpha+\beta=1$; put $\alpha=p^{-1}$, $\beta=q^{-1}$ and substitute $a\to a^p$, $b\to b^p$); plainly it remains true if merely $\{a,b\}\subset [0,\infty]$. Apply monotonicity and additivity of the integral to $\vert f\vert \vert g\vert\leq \frac{\vert f\vert^p}{p}+\frac{\vert g\vert^q}{q}$.  
	
			Jensen. Being convex, $\phi$ is continuous, hence  $\phi\in\mathcal{B}_I/\mathcal{B}_\mathbb{R}$. Next, let $a_1:=\inf I\in [-\infty,\infty)$ and $a_2:=\sup I\in (-\infty,\infty]$ (note that, necessarily, $I\ne \emptyset$). If $a_1=-\infty$ (resp. $a_2=\infty$), clearly $\mu[f]>a_1$ (resp. $\mu[f]<a_2$), since, by assumption, $f\in \LL^1(\mu)$. Otherwise, observe that $\mu[ f]=a_1$ (resp. $\mu [ f]=a_2$) would imply $f=a_1$ (resp. $f=a_2$) $\mu$-a.s., which cannot be. It follows that $\mu[ f]\in I$. Further to this, note that for some $\{a,b\}\subset \mathbb{R}$, $\phi \circ f\geq af+b$, hence $(\phi\circ f)^-\leq (af+b)^-\leq \vert af+b\vert\leq \vert a\vert \vert f\vert+\vert b\vert$, so that $\mu[(\phi\circ f)^-]<\infty$.  Finally observe that $\phi$ is the pointwise supremum of the affine minorants of $\phi$.\footnote{This is a consequence of the fact that $\phi$ has ``nondecreasing difference quotients'', in the precise sense that $\frac{\phi(t)-\phi(s)}{t-s} \leq \frac{\phi(u)-\phi(s)}{u-s}\leq \frac{\phi(u)-\phi(t)}{u-t}$, whenever $\{s,t,u\}\subset I$ and $s<t<u$. Moreover, $\phi$ admits a finite left and right derivative function (because $I$ is open), the latter pointwise no smaller than the former.\label{foot:convex}} Let the set of the latter be denoted $\mathcal{A}$. Then $\phi\circ f\geq a\circ f$, and hence by linearity (and because $\mu$ is a probability) $\mu[\phi \circ f]\geq a(\mu[f])$ for all $a\in \mathcal{A}$; taking the supremum over $a\in \mathcal{A}$ yields $\mu[\phi\circ f]\geq \phi(\mu[f])$. Assume now $\phi$ is strictly convex. There is $l\in \mathcal{A}$ such that $\phi(\PP[f])=l(\PP[f])$ and (automatically by strict convexity) $l<\phi$ on $I\backslash \{\PP[f]\}$. Consequently $\phi(\PP[f])=\PP[\phi(f)]$ implies $\phi(\PP[f])=\PP[\phi(f)]\geq \PP[l(f)]=l(\PP[f])=\phi(\PP[f])$ rendering $\PP[\phi(f)]=\PP[l(f)]$ and so $\PP(f=\PP[f])=1$. The converse implication is essentially trivial.
	
	Minkowski. We may and do assume $f\land g\geq 0$, $\Vert f\Vert_p\lor \Vert g\Vert_p<\infty$ and $\Vert f+g\Vert_p>0$. For $p=1$ it is the triangle inequality; let $p>1$. Then an elementary inequality is that $(x+y)^p\leq 2^{p-1}(x^p+y^p)$ for $\{x,y\}\subset [0,\infty)$ (why is it true?: the $p$-th power is convex; therefore $((x+y)/2)^p\leq (x^p+y^p)/2$); by monotonicity and additivity of the integral it follows that $\Vert f+g\Vert_p<\infty$. We then compute and  estimate (via H\"older): $\Vert f+g\Vert_p^p=\mu[(f+g)^p]=\mu[(f+g)^{p-1}(f+g)]=\mu[f(f+g)^{p-1}]+\mu[g(f+g)^{p-1}]\leq (\Vert f\Vert_p+\Vert g\Vert_p)(\mu[(f+g)^{(p-1)(1-\frac{1}{p})^{-1}}])^{1-\frac{1}{p}}=(\Vert f\Vert_p+\Vert g\Vert_p)\frac{\Vert f+g\Vert_p^p}{\Vert f+g\Vert_p}$.
	
	The claims concerning the seminorm and inner semiproduct follow at once from the above. Uniqueness of limit in $\LL^p(\mu)$: use Markov's inequality. 
	\end{proof}
\begin{remark}
One can generalize H\"older to the case $p^{-1}+q^{-1}=r^{-1}$, $\{p,q,r\}\subset [1,\infty]$, in the obvious way.
\end{remark}

	\begin{proof}[Proof of completeness of $\LL^p(\mu)$] Again the case $p=\infty$ is elementary, so assume $p<\infty$.

 Let $(f_{n_k})_{k\in \mathbb{N}}$ be a subsequence such that $\Vert f_{n_k}-f_{n_{k+1}}\Vert_p\leq 2^{-k}$ for all $k\in\mathbb{N}$. By Minkowski and monotone convergence $\Vert \sum_{k\in \mathbb{N}}\vert f_{n_k}-f_{n_{k+1}}\vert\Vert_p\leq 1$. It follows in particular that the series $f_{n_1}+\sum_{k\in \mathbb{N}} (f_{n_{k+1}}-f_{n_k})$ converges absolutely in $\mathbb{R}$ a.e.-$\mu$; let $f_0$ be equal to its sum on the set on which the convergence is absolute in $\mathbb{R}$, and set $f_0$ equal to zero off this set. Then $f_0=\lim_{k\to\infty}f_{n_k}$ a.e.-$\mu$. We show that this convergence is also in $\LL^p(\mu)$. Let $\epsilon>0$. Then there is $N\in \mathbb{N}$ such that $\Vert f_n-f_m\Vert_p\leq \epsilon$ for all  $m$ and $n$ from $\mathbb{N}_{\geq N}$. By Fatou, for all $n\in \mathbb{N}_{\geq N}$, $\mu[\vert f-f_n\vert^p]\leq \liminf_{k\to\infty}\mu[\vert f_{n_k}-f_n\vert^p]\leq \epsilon^p$. In particular $f\in \LL^p(\mu)$ and $\lim_{n\to\infty}f_n= f$  in $\LL^p(\mu)$. 
	\end{proof}
\begin{remarks}
	The preceding proof shows in fact that if a sequence is Cauchy in $\LL^p(\mu)$, then it has a subsequence that converges a.e.-$\mu$ (to a limit, which is the $\LL^p(\mu)$ limit of the sequence). The following approximation result is sometimes useful.
	\end{remarks}
	\begin{proposition}\label{approx:lp}
		Let $(\Omega,\FF,\mu)$ be a measure space and $p\in [1,\infty]$. Then $\{f\in \LL^p(\mu):\mathrm{range}(f)\text{ is finite}\}$ is dense in $\LL^p(\mu)$. If furthermore $\mathcal{A}$ is an algebra generating $ \FF$ on $\Omega$ with  $\Omega$ being a union of a sequence of sets from $\mathcal{A}$ each of which has finite $\mu$-measure, then:  for any $F\in \FF$ of finite $\mu$-measure there is a sequence $(A_n)_{n\in \mathbb{N}}$ in $\mathcal{A}$ with $\lim_{n\to\infty}\mu(F\triangle A_n)=0$; in particular, if $p<\infty$, then finite linear (over $\mathbb{R}$) combinations of indicators of elements of $\mathcal{A}\cap \mu^{-1}([0,\infty))$ are  dense in $\LL^p(\mu)$.
	\end{proposition}
	\begin{proof}
Let $f\in\LL^p(\mu)$. For the first part we want to show that $f$ can be approximated arbitrarily well in the $\LL^p(\mu)$-seminorm by an element of  $\LL^p(\mu)$ with finite range. By Minkowski and $f=f^+-f^-$ we reduce to the case when $f\geq 0$. For $p=\infty$ the matter is then immediate by Proposition~\ref{proposition:approx}. Assume $p<\infty$. Again approximate via Proposition~\ref{proposition:approx}, applying (monotonicity of the integral and) dominated convergence.  For the second part, let $(A_n)_{n\in \mathbb{N}}$ be a sequence in $\mathcal{A}$ such that  $\Omega=\cup_{n\in\mathbb{N}}A_n$ is a nondecreasing union of sets of finite $\mu$-measure. Because $\mu(F)<\infty$ it will be enough to argue that $\inf_{A\in \mathcal{A}\vert_{A_n}}\mu((F\cap A_n)\triangle A)=0$ for each $n\in \mathbb{N}$, so 
we may just as well assume that $\mu$ is finite and show that  $\inf_{A\in \mathcal{A}}\mu(F\triangle A)=0$ for each $F\in  \FF$. The latter is immediate because one checks  easily that the class of sets $F\in\FF$ for which the claim holds true is a $\sigma$-algebra on $\Omega$, which contains $\AA$. The final part follows as a corollary to the first two; the fact that $p<\infty$ makes sure that when approximating $f\in \LL^p(\mu)$ with an element $g\in \LL^p(\mu)$ with finite range, the preimages $g^{-1}(x)$, $x\in \mathrm{range}(g)$, have finite $\mu$-measure.
	\end{proof}
\begin{example}\label{example:density-counter}
The case $p= \infty$ is conspicuously missing in the last part of the preeding proposition. This is beause in general it fails. Consider for instance $\AA$, which consists of finite disjoint unions of intervals of the form $(a,b]\cap \mathbb{R}$ where $a\leq b$ are from the extended real line, $\mu=\leb$ Lebesgue measure. Then no finite linear combination of indicators of elements of $\AA$ is closer than $1$ to $\mathbbm{1}_{\cup_{n\in\mathbb{N}}(2n,2n+1]}$ in  the $\LL^\infty(\leb)$-seminorm. The same example shows why the assumption that $F$ have finite measure cannot be dropped entirely.
\end{example}

\begin{example}[Riemann-Lebesgue lemma]
Retain the notation for $\AA$ of Example~\ref{example:density-counter}. If $A\in \AA$ has finite Lebesgue measure (i.e. is bounded), then $\lim_{\vert t \vert\to\infty}\int e^{\mathrm{i} t x}\mathbbm{1}_A\dd \leb=0$ follows by a straightforward computation. By approximation (so applying Proposition~\ref{approx:lp}) and an ``epsilon-limsup'' argument  it follows that $\lim_{\vert t \vert\to\infty}\int e^{\mathrm{i} t x}f\dd \leb=0$ prevails for all $f\in \LL^1(\leb)$.
\end{example}

\begin{example}\label{example:approximate-indicators}
Let $(\Omega,\FF,\mu)$ be a measure space and $\AA$ an algebra generating $\FF$ on $\Omega$ with  $\Omega$ being a union of a sequence $(A_n)_{n\in \mathbb{N}}$ of sets from $\mathcal{A}$ each of which has finite $\mu$-measure. Let $F\in \FF$. We assert existence of  a sequence $(C_l)_{l\in \mathbb{N}}$ in $\AA$ converging to $F$ a.e.-$\mu$; or in other words such that $F=\liminf_{l\to\infty}C_l=\limsup_{l\to\infty}C_l$ a.e.-$\mu$. Since $(A_n)_{n\in \mathbb{N}}$  is assumed without loss of generality to consist of pairwise disjoint sets we reduce indeed at once to the case when $\mu$ is finite to begin with. But the latter  is contained in Proposition~\ref{approx:lp}. 
\end{example}

\begin{corollary}
	Let $(\Omega,\FF,\mu)$ be a measure space, $\mu$ finite, $p\in [1,\infty)$. Let $\Pi$ be a $\pi$-system generating $ \FF$ on $\Omega$ with $\Omega\in\Pi$. Then finite linear (over $\mathbb{R}$) combinations of indicators of elements of $\Pi$ are  dense in $\LL^p(\mu)$.
\end{corollary}
\begin{proof}
Because $\Pi$ is a $\pi$-system, the product of two  finite linear combinations of indicators of elements of $\Pi$ is again a finite linear combination of indicators of elements of $\Pi\cup\{\Omega\}$. It follows easily ($\because$ $\Omega\in \Pi$) that such linear combinations exhaust the finite linear combinations of indicators of elements of the algebra generated by $\Pi$. Apply Proposition~\ref{approx:lp}.
\end{proof}
\part{Probability}

\chapter{Probability as a normalized measure}

\begin{remark}
In some sense the content of this chapter is nothing but a  ceaseless application of the first part.
\end{remark}
\section{Basic notions}
\emph{probability spaces, random elements, distribution functions and laws, absolutely continuous and discrete random variables; expectations of functions of random elements; quantile coupling of random variables of all possible laws on the real line; convergence in probability}

\begin{definition}
	A probability space is a measure space $(\Omega,\FF,\PP)$ where $\PP$ is a probability measure; let $(\Omega,\FF,\PP)$ be so. $\Omega$ is called the sample space and to the elements of $\FF$ one refers to as events.  $A$ is $\PP$-almost sure (abbreviated $\PP$-a.s.) \deff $A\in \FF$ and $\PP(A)=1$. For a measurable space $(E,\EE)$, elements of $\FF/\EE$ are called $(E,\EE)$-valued random elements; in the particular case when $(E,\EE)=(\mathbb{R},\mathcal{B}_{\mathbb{R}})$ they are called random variables. 
	For a random element $X$: $X\sim_\PP \QQ$ \deff $X$ has law $\QQ$ under $\PP$. Two random elements  valued in the same measurable space (and possibly defined on different probability spaces) are identically distributed \deff they have the same law. The distribution function of a random variable $X$ is the map $F_X:\mathbb{R}\to [0,1]$ given by $F_X(u):=\PP(X\leq u)$ for $u\in \mathbb{R}$. A random variable $X$ is said to be discrete (resp. absolutely continuous, continuous (also, diffuse)) \deff there is a countable $C\subset \mathbb{R}$ such that $\PP(X\in C)=1$ (resp. \deff $\PP_X\ll \leb$, \deff $F_X$ is continuous). A bivariate random vector is an element $(X,Y)$ of $\FF/\mathcal{B}_{\mathbb{R}^2}$ (a general random vector is defined analogously), absolutely continuous \deff  $\PP_{(X,Y)}\ll \leb^2$, and so on. 
\end{definition}
\begin{remarks}
	For the expectation of an $X\in \FF/\mathcal{B}_{[-\infty,\infty]}$ one tends to write $\mathsf{E}_\PP[X]:=\PP[X]$ (and just $\mathsf{E}[X]$ if $\PP$ can be deduced from context), but we shall prefer to keep $\PP[X]=\int X\dd \PP$. The notation $F_X$ is in principle unsatisfactory; it fails to reference $\PP$. In deference to standard practice we will allow this (for all occurrences of $F_X$ below, only one probability measure on the space on which $X$ is defined will ever be ``in sight'', so in principle no confusion should arise). Not every continuous random variable is absolutely continuous (non-trivial examples attest to this); the converse is true as we shall soon see.   To say that a random variable $X$ is discrete is (not without irony) to say that its law is absolutely continuous w.r.t. the  measure $\mathbbm{1}_C\cdot c_\mathbb{R}\vert_{\mathcal{B}_\mathbb{R}}$ for some countable $C\subset \mathbb{R}$. (Thus) much of what follows could be unified to the case when the laws of the random elements are absolutely continuous w.r.t. some ($\sigma$-finite) reference measures (indeed, like $\leb$, so too $\mathbbm{1}_C\cdot c_\mathbb{R}\vert_{\mathcal{B}_\mathbb{R}}$ is $\sigma$-finite, for $\{\{c\}:c\in C\}\cup \{\mathbb{R}\backslash C\}$ constitutes a countable Borel measurable cover  of the real line consisting of sets of finite $\mathbbm{1}_C\cdot c_\mathbb{R}$-measure). But enough is enough. 
\end{remarks}

\begin{example}
	Let $p:D\to [0,1]$ satisfy $\sum_{k\in D}p_k=1$ (a probability mass function). Necessarily $$\{p>0\}=\cup_{n\in \mathbb{N}}\underbrace{\left\{p\geq \frac{1}{n}\right\}}_{\mathrm{finite}}$$ is countable. Then $(D,2^D,p\cdot c_D)$ is a probability space. If further $D\subset \mathbb{R}$, then $\id_D$ is a discrete random variable  thereon: just because  $(p\cdot c_D)(\id_D\in \{p>0\})=1$. 
\end{example}

\begin{example}\label{example:fund}
	$([0,1],\mathcal{B}_{[0,1]},\leb_{[0,1]})$ is a probability space and $\id_{[0,1]}$ is an absolutely continuous random variable thereon, indeed its law (assuming we view it is an $(\mathbb{R},\mathcal{B}_\mathbb{R})$-valued random element, i.e. as a random variable) is $\mathbbm{1}_{[0,1]}\cdot \leb$: just because $\leb_{[0,1]}(\id_{[0,1]}\in A)=\leb([0,1]\cap A)=(\mathbbm{1}_{[0,1]}\cdot \leb)(A)$ for $A\in \mathcal{B}_\mathbb{R}$.
	 \end{example}
\begin{example}
Let $F:\mathbb{R}\to [0,1]$ be nondecreasing, right-continuous, $\lim_{-\infty}F=0$, $\lim_\infty F=1$ (a distribution function). Then $(\mathbb{R},\mathcal{B}_\mathbb{R},\dd F)$ is a probability space and $\id_\mathbb{R}$ is a random variable thereon whose distribution function is  $F$: just because $\dd F(\id_\mathbb{R}\leq u)=\dd F((-\infty,u])=F(u)$ for $u\in \mathbb{R}$.
\end{example}

\begin{example}
Let $f\in \mathcal{B}_\mathbb{R}/\mathcal{B}_{[0,\infty)}$ with $\leb[f]=1$ (a density). Then $(\mathbb{R},\mathcal{B}_\mathbb{R},f\cdot \leb)$ is a probability space and $\id_\mathbb{R}$ is an absolutely continuous random variable thereon, indeed its law is just the underlying probability, $f\cdot \leb$. 
\end{example}
\begin{proposition}
Let  $(\Omega,\FF,\PP)$ be a probability space. For an $(E,\EE)$-valued random element $X$ and for $f\in \EE/\mathcal{B}_{[-\infty,\infty]}$,
 \begin{equation}
\PP[f(X)]=\PP_X[f],
\end{equation}
the expectation on the left-hand side being well-defined iff the expectation on the right-hand side is well-defined.

For a random variable $X$, $F_X$ is nondecreasing, right-continuous, $\lim_{-\infty}F_X=0$, $\lim_\infty F_X=1$ and $\PP_X=\dd F_X$. 

If $X$ is  a discrete random variable, then there is a smallest countable $C\subset \mathbb{R}$, denoted $\supp(X)$ and called the support of $X$, for which $\PP(X\in C)=1$, namely $\supp(X)=\{x\in \mathbb{R}:\PP(X=x)>0\}$; further, for any $f\in \mathcal{B}_\mathbb{R}/\mathcal{B}_{[-\infty,\infty]}$,
\begin{equation}
\PP[f(X)]=\sum_{a\in \supp(X)}f(a)\PP(X=a)
\end{equation}
provided at least one of the series $\sum_{a\in \supp(X)}f^+(a)\PP(X=a)$ and $ \sum_{a\in \supp(X)}f^-(a)\PP(X=a)$ is $<\infty$ (and then, and only then, is $\PP[f(X)]$  well-defined).

If $X$ is an absolutely continuous random variable, then $X$ is continuous and there is an a.e.-$\leb$ unique $f\in \mathcal{B}_\mathbb{R}/\mathcal{B}_{[0,\infty)}$ for which $\PP_X=f\cdot \leb$, denoted $f_X$, and called the density of $X$, namely $f_X=\frac{\dd\PP_X}{\dd\leb}$ a.e.-$\leb$; further, for any $g\in \mathcal{B}_\mathbb{R}/\mathcal{B}_{[-\infty,\infty]}$,
\begin{equation}
\PP[g(X)]=\int gf_X\dd\leb,
\end{equation}
the expectation being well-defined iff $\int g^+f_X\dd\leb\land  \int g^-f_X\dd\leb<\infty$.

Finally, in order for a random variable $X$ to be absolutely continuous, it is equivalent that there exists an $f\in \mathcal{B}_\mathbb{R}/\mathcal{B}_{[0,\infty)}$ such that 
\begin{equation}\label{eq:enough}
\PP(X\leq x)=\int_{(-\infty,x]}f\dd\leb \text{ for all $x\in \mathbb{R}$}, 
\end{equation}
in which case $f$ is a density for $X$ (more generally it is equivalent to have $\PP(X\in A)=\int_{A}f\dd\leb$ for all $A\in\Pi\cup\{\mathbb{R}\}$, where $\Pi$ is a $\pi$-system generating $\mathcal{B}_\mathbb{R}$ on $\mathbb{R}$).
\end{proposition}
\begin{remark}
There are ``obvious'' analogues of the preceding in case $X$ is not a random variable but rather a bivariate (or even higher-dimensional) random vector.
\end{remark}
\begin{definition}
	We keep the notation for densities developed in the preceding proposition. Statements involving them are  asserted for any of the versions of the density. As a result they tend to be inflicted with an a.e.-$\leb$ qualifier. For a discrete random variable $X$ we keep also the notation $\supp(X)$ and call the map $p_X:=(\supp(X)\ni a\mapsto \PP(X=a))$ the probability mass function of $X$.
\end{definition}
\begin{remark}
The notations $\supp(X)$, $p_X$ and $f_X$ are incomplete for they fail to reference $\PP$ on which they depend; we transgress. They will be used  sparingly and only when there is only one probability ``in sight''. 
\end{remark}
\begin{proof}
	The results are easy consequences and particular cases of those from part one. It is perhaps only non-trivial to note (1) why $\PP_X=\dd F_X$ for a random variable $X$: there is equality on the $\pi$-system $\{(-\infty,a]:a\in \mathbb{R}\}$ so one can apply Proposition~\ref{proposition:equality of measures}; 
	(2) why the last claim holds true: by monotone convergence in \eqref{eq:enough} $\leb[f]=1$, then \eqref{eq:enough} and Proposition~\ref{proposition:equality of measures} again yield  $\PP_X=f\cdot \leb$; and (3)  why absolute continuity implies continuity of $X$: left-continuity of $F_X$ follows by mononotone convergence in \eqref{eq:enough} (while right-continuity of $F_X$ is automatic). 
	\end{proof}
	
	\begin{example}
Let $p\in (0,1]$ and define $\geom_\mathbb{N}(p):=(\mathbb{N}\ni k\mapsto p(1-p)^{k-1})\cdot c_\mathbb{N}$ [we interpret $0^0:=1$]. Then  $\geom_\mathbb{N}(p)$ is a probability measure on $2^\mathbb{N}$ called the geometric law on $\mathbb{N}$ with success probability $p$. If $X\sim_\PP\geom_\mathbb{N}(p)$, then $X$ is a discrete random variable with probability mass function $(\mathbb{N}\ni k\mapsto p(1-p)^{k-1})$ or $(\{1\}\ni k\mapsto 1)$ according as to whether $p<1$ or $p=1$, and  $$\PP[X]=\geom_\mathbb{N}[\id_\mathbb{N}]=\sum_{k\in \mathbb{N}}k p(1-p)^{k-1}=p^{-1};$$
incidentally, for $p<1$ the preceding sum may be evaluated via differentiation under the summation (= integral against $c_\mathbb{N}$) sign, while for $p=1$ the sum is trivial.
	\end{example}

\begin{example}
	Let $\lambda\in (0,\infty)$ and define $\Exp(\lambda):=((0,\infty)\ni x\mapsto \lambda e^{-\lambda x})\cdot \leb_{(0,\infty)}$. Then  $\Exp(\lambda)$ is a probability measure on $\mathcal{B}_{(0,\infty)}$ called the exponential law of rate $\lambda$. If $X\sim_\PP\Exp(\lambda)$, then $X$ is an absolutely continuous random variable with density $(\mathbb{R}\ni x\mapsto \mathbbm{1}_{(0,\infty)}(x)\lambda e^{-\lambda x})$ and for $\mu\in [0,\infty)$, $$\PP[e^{-\mu X}]=\Exp(\lambda)[e^{-\mu\cdot}]=\int_{(0,\infty)}e^{-\mu x}\lambda e^{-\lambda x}\leb(\dd x)=\frac{\lambda}{\lambda+\mu}.$$
\end{example}

\begin{proposition}\label{example:quantile}	
	Let $(\Omega,\FF,\PP)$ be a probability space.
	If $X$ is a  random variable, then $F_X\in \mathcal{B}_\mathbb{R}/\mathcal{B}_{[0,1]}$, and $F_X(X)\sim_\PP\leb_{[0,1]}$ iff $F_X$ is continuous. 
	Conversely, let $U\sim_\PP \leb_{(0,1)}$. If $F:\mathbb{R}\to [0,1]$ is a distribution function (right-continuous, nondecreasing, $\lim_{-\infty}F=0$, $\lim_\infty F=1$),  putting $$F^\leftarrow(u):=\inf\{v\in \mathbb{R}:F(v)>u\},\quad u\in (0,1),$$ for its quantile function (a.k.a. right-continuous inverse), then $F^\leftarrow\in \mathcal{B}_{(0,1)}/\mathcal{B}_\mathbb{R}$ and $F^\leftarrow(U)\sim_\PP\dd F$. 
\end{proposition}
\begin{proof} 
Borel measurabilities follow from the functions being nondecreasing. Then a main observation is that  $u< F(x)\Rightarrow F^\leftarrow(u)\leq x\Rightarrow u\leq  F(x)$, which is valid for all $u\in (0,1)$ and $x\in \mathbb{R}$; this gives the second claim. If $F$ is continuous, then $F(x) \leq u$ iff $x\leq F^\leftarrow( u)$  for all $u\in (0,1)$ and $x\in \mathbb{R}$, while $F\circ F^\leftarrow =\mathrm{id}_{(0,1)}$, which yields the first claim. 
\end{proof}

\begin{definition}\label{definition:convergence-in-probability}
Let  $(\Omega,\FF,\PP)$  be a probability space, $(X)_{n\in \mathbb{N}}$ a sequence in $\FF/\mathcal{B}_{\mathbb{R}}$ and $X\in \FF/\mathcal{B}_{\mathbb{R}}$. This sequence is said to converge to $X$ in $\PP$-probability \deff $\PP(\vert X_n-X\vert\geq \epsilon)\to 0$ as $n\to\infty$ for all $\epsilon\in (0,\infty)$.
\end{definition}
\begin{remarks}
The qualifying ``$\PP$'' in ``$\PP$-probability'' is dropped if it can be gathered from context.

Convergence of a sequence of indicators to another indicator in probability corresponds to convergence in the \underline{\href{https://drive.google.com/file/d/0B0Mmq8RrwCHDeTRreWE2MWxsRE0/view}{measure of symmetric difference pseudometric}}. 

Suppose $X_n\to X$ as $n\to\infty$ in $\PP$-probability. Then inductively we find $\mathsf{n}:\mathbb{N}\to \mathbb{N}$ that is $\uparrow\uparrow$ (a subsequence),  such that $\PP(\vert X_{\mathsf{n}_k}-X\vert \geq 2^{-k})\leq 2^{-k}$ for all $k\in \mathbb{N}$ and therefore (by Borel-Cantelli I,  Example~\ref{Borel-Cantelli-I}) $\PP(\lim_{k\to\infty}X_{\mathsf{n}_k}=X)=1$. In words, convergence in probability implies convergence a.s. along a subsequence.    
 
\end{remarks}

\begin{proposition}\label{proposition:types-of-convergence}
Let  $(\Omega,\FF,\PP)$  be a probability space. If $(X)_{n\in \mathbb{N}}$ is a sequence in $\FF/\mathcal{B}_{\mathbb{R}}$ that converges to an $X\in \FF/\mathcal{B}_{\mathbb{R}}$ $\PP$-a.s. or in $\LL^p(\PP)$ for some $p\in [1,\infty]$, then it does so also in $\PP$-probability.
\end{proposition}
\begin{proof}
For $p<\infty$: Markov's inequality. For a.s. convergence: bounded convergence. For $p=\infty$: use the fact that $\vert X_n-X\vert \leq \Vert X_n-X\Vert_\infty$ a.s-$\PP$ for all  $n\in \mathbb{N}$, to reduce to a.s. convergence.
\end{proof}
\begin{remark}
Convergence in probability is perhaps not really a ``basic'' notion; we have included it in the preceding ever so briefly due to its fundamental importance, and because we have nowhere else to put it. On the other hand, many concepts that indeed are ``basic'', like $\var_\PP(X):=\PP[(X-\PP[X])^2]$ for $X\in \LL^2(\PP)$ and $\cov_\PP(X,Y):=\PP[(X-\PP[X])(Y-\PP[Y])]$ for $\{X,Y\}\subset \LL^2(\PP)$ have been omitted, but one is anyway typically able to ``guess'' the measure-theoretic counterparts from knowing the corresponding notions from elementary probability.
\end{remark}
\section{Independence}
\emph{definition of stochastic independence and first properties; independent fair coin tosses as a fundamental probabilistic model; product of arbitrary family of probability measures}

\begin{remark}
	Independence is a (and is perhaps the only) truly fundamental notion specific to probability (as opposed to any, or even just finite) measures. The role that this notion has to play in all matters stochastic cannot be overemphasized. Remarkably, independence (of pairs) of events, to be introduced presently, \underline{\href{https://www.ams.org/journals/proc/1997-125-12/S0002-9939-97-03994-4/S0002-9939-97-03994-4.pdf}{determines a non-atomic probability}}.
\end{remark}

\begin{definition}
	Let  $(\Omega,\FF,\PP)$  be a probability space. For a collection $\CC=(\CC_\lambda)_{\lambda\in \Lambda}$ of subsets of $\FF$, we say $\CC$ is an independency (under $\PP$)  \deff for any (non-empty) finite $I\subset \Lambda$, and then for any choices of $C_i\in \CC_i$, $i\in I$, we have $$\PP(\cap_{i\in I}C_i)=\prod_{i\in I}\PP(C_i).$$ Subsets $\BB$ and $\CC$ of $\FF$ are independent (under $\PP$) \deff the family $(\BB,\CC)$ consisting of them alone, is an independency (under $\PP$); $B$ and $C$ from $\FF$ are independent (under $\PP$) \deff $\{B\}$ and $\{C\}$ are independent. Further, given a measurable space $(E,\Sigma)$, $Z\in \FF/\Sigma$, and $\BB \subset\FF$,  $Z$ is independent of $\BB$ (relative to $\Sigma$ under $\PP$) \deff $\sigma^\Sigma(Z)$ is independent of $\BB$ (independence for random elements means independence of their initial structures). And so on, and so forth.
	\end{definition}

	\begin{remarks}
A sub-collection of an independency is an independency. A collection is an independency if and only if every finite sub-collection thereof is so. 
\end{remarks}
\begin{example}
If $Y$ and $Z$ are random elements on a probability space $(\Omega,\FF,\PP)$ with values in measurable spaces $(G,\GG)$ and $(E,\EE)$, respectively, then $Y$ and $Z$ are independent, by definition, iff $\sigma^\GG(Y)$ is independent of $\sigma^\EE(Z)$, i.e. iff $\PP(\{Y\in A\}\cap \{Z\in B\})=\PP(Y\in A)\PP(Z\in B)$ for all $A\in\GG$ and $B\in\EE$.\footnote{  You may complain that it would be more natural to insist that $\PP(\{Y\in A\}\cap \{Z\in B\})=\PP(Y\in A)\PP(Z\in B)$, whenever $A\subset G$ and $B\subset E$ are such that it happens to be the case that $\{\{Y\in A\},\{Z\in B\}\}\subset\FF$ i.e. all the probabilities involved are defined (making the notion of independence of random elements independent (so to speak) of the measurable structures on the codomains). It would be ``independence relative to the final structures''. For random variables (and the Borel $\sigma$-fields) the two concepts of independence   \underline{\href{https://www.impan.pl/en/publishing-house/journals-and-series/colloquium-mathematicum/all/1/3/93612/on-two-notions-of-independent-functions}{may indeed disagree}}, but \underline{\href{https://projecteuclid.org/euclid.bsmsp/1200502002}{only if the underlying probability space is not ``nice enough''}}. }
\end{example}

\begin{example}\label{example:binary}
	Define  a sequence $X=(X_n)_{n\in \mathbb{N}}$ in $\{0,1\}^{[0,1]}$ as follows: for $\omega\in [0,1]$ and $n\in \mathbb{N}$,  $X_n(\omega)$ is the $n$-th digit (after the decimal (binal?) point) in the binary expansion of $\omega$ (with some convention when there is ambiguity; it is so only on a countable set (the dyadic numbers of $[0,1]$), which is of Lebesgue measure zero: for this reason the particulars of the convention will not be relevant). Thus $\omega=\sum_{n\in \mathbb{N}}\frac{X_n(\omega)}{2^n}$ for $\omega\in [0,1]$ (using the correct convention for $\omega=1$, namely writing $1$ as $(0.1111\ldots)_2$ rather than $(1.0)_2$). Then $(X_n)_{n\in \mathbb{N}}$ is an independency (relative to $2^{\{0,1\}}$) under $\leb_{[0,1]}$; furthermore, $\leb_{[0,1]}(X_n=0)=\frac{1}{2}$ for all $n\in \mathbb{N}$ (so we have ``independent fair coin tosses''). Conversely, if on a given probability space $(\Omega,\FF,\PP)$, $X=(X_n)_{n\in \mathbb{N}}$ is a sequence of independent fair coin tosses, i.e. if $X$ an independency of $(\{0,1\},2^{\{0,1\}})$-valued random elements with $\PP(X_n=0)=\frac{1}{2}$ for all $n\in \mathbb{N}$, then $\sum_{n\in \mathbb{N}}\frac{X_n}{2^n}\sim_\PP\leb_{[0,1]}$.
\end{example}

\begin{remarks}
	Combining Example~\ref{example:binary} and Proposition~\ref{example:quantile} (and using a bijection $\mathbb{N}\times \mathbb{N}\to\mathbb{N}$) we see (modulo details involving the ``(dis)aggregation'' of independence) that already the probability measure $\leb_{[0,1]}$ is rich enough to support a countable independency of random variables, each of which has any probability law that we may like (the law may vary across the random variables).  This is sufficient for almost all situations one encounters in probabilistic practice. We see also that starting with an independency of equiprobable coin tosses one can construct any Lebesgue-Stieltjes measure. The latter is not quite already tantamount to an outright construction of these measures though, because one still needs to show existence of such a sequence of coin tosses (Proposition~\ref{proposition:arbitrary-products}\eqref{product--arbit:iii}); nevertheless, this ``probabilistic'' route to Lebesgue-Stieltjes measures is quite appealing. 
\end{remarks}
	\begin{proposition}
		Let $(\Omega,\FF,\PP)$  be a probability space and let $X$ and $Y$ be random elements valued in $(E,\EE)$ and $(A,\AA)$, respectively. Then $(X,Y)\in \FF/(\EE\otimes \AA)$. Furthermore, $X$ and $Y$ are independent under $\PP$ iff $\PP_{(X,Y)}=\PP_X\times \PP_Y$, in which case:  for any $f\in (\EE\otimes \AA)/\mathcal{B}_{[-\infty,\infty]}$,
		\begin{equation}\label{eq:independent-separration}
		\PP[f(X,Y)]=\PP_{(X,Y)}[f]=\int \PP[f(x,Y)]\PP_X(\dd x),
		\end{equation}
		provided $\PP[f^+(X,Y)]\land \PP[f^-(X,Y)]<\infty$; in particular $\PP[g(X)h(Y)]=\PP[g(X)]\PP[h(Y)]$ for $g\in \EE/\mathcal{B}_{[-\infty,\infty]}$ and $h\in \AA/\mathcal{B}_{[-\infty,\infty]}$ satisfying $\PP[(g(X)h(Y))^+]\land \PP[(g(X)h(Y))^-]<\infty$. \qed
	\end{proposition}
\begin{remark}
	``Independence can be raised from $\pi$-systems to the $\sigma$-fields they generate" in the precise sense of the proposition that follows.
\end{remark}
	\begin{proposition}\label{corollary:raising-independence}
		Let $(\Omega,\FF,\PP)$  be a probability space, $(\CC_\lambda)_{\lambda\in \Lambda}$ an independency under $\PP$ consisting of $\pi$-systems alone. Then $(\sigma_\Omega(\CC_\lambda))_{\lambda\in \Lambda}$ also is an independency under $\PP$. 
	\end{proposition}
	\begin{proof}
		It will suffice to show that if the finite collection of $\pi$-systems $(\AA_1,\ldots,\AA_n)$ ($n\geq 2$ an integer) on $\Omega$ is an independency, each containing $\Omega$, then $(\sigma_\Omega(\AA_1),\AA_2,\ldots,\AA_n)$ is one also. (For thereafter one can apply mathematical induction to conclude that $(\sigma_\Omega(\AA_1),\sigma_\Omega(\AA_2),\ldots,\sigma_\Omega(\AA_n))$ is an independency. It is enough because independence of $(\CC_\lambda)_{\lambda\in \Lambda}$ implies trivially independence of $(\CC_\lambda\cup \{\Omega\})_{\lambda\in \Lambda}$ and because checking independence reduces to checking it on finite subfamilies.) Consider then  the collection of all elements $L\in \FF$, such that for all $A_2\in \AA_2$, \ldots, $A_n\in \AA_n$,  one has $\PP(L\cap A_2\cap\cdots \cap A_n)=\PP(L)\PP(A_2)\cdots \PP(A_n)$. It is, from the hypothesis and by properties of probability measures, a $\lambda$-system, containing the $\pi$-system $\AA_1$. Apply Dynkin's lemma to deduce that $(\sigma_\Omega(\AA_1),\AA_2,\ldots,\AA_n)$ is an independency (since the systems $\AA_k$, $k\in [n]$, contain $\Omega$ one does not have to check the condition on finite subcollections, it is automatic).
	\end{proof}
	
			\begin{example}
	Let $(\Omega,\FF,\PP)$ be a probability space and let  $\CC=(\CC_\lambda)_{\lambda\in \Lambda}$ be an independency of sub-$\sigma$-fields of $\FF$. Then also $(\CC_\lambda\lor \PP^{-1}(\{0,1\}))_{\lambda\in \Lambda}$ is an independency. Indeed one has only to note that  $(\CC_\lambda \cup \PP^{-1}(\{0\}))_{\lambda\in \Lambda}$ is an independency of $\pi$-systems with $\sigma_\Omega(\CC_\lambda \cup \PP^{-1}(\{0\}))=\CC_\lambda\lor \PP^{-1}(\{0,1\})$ for all $\lambda\in \Lambda$.
	\end{example}
	\begin{remark}
	We might say that independece does not see trivial sets.
	\end{remark}
	\begin{example}
	Let $(\Omega,\FF,\PP)$ be a probability space and let $X$ and $Y$ two random elements with values in the countable sets $E$ and $F$, respectively (endowed with their discrete measurable structures). Then $X$ is independent of $Y$ iff $\PP(\{X=x\}\cap \{Y=y\})=\PP(X=x)\PP(Y=y)$ for all $x\in E$ and $y\in F$. To see it, apply Proposition~\ref{corollary:raising-independence} to the generating $\pi$-systems $\{\{X=x\}:x\in E\}\cup\{\emptyset\}$ and $\{\{Y=y\}:y\in F\}\cup\{\emptyset\}$.  One could generalize this to arbitrary families of ``discrete'' random elements in a clear way. Another extension would be to drop the countability assumptions on $E$ and $F$, replacing the power sets with the countable--co-countable $\sigma$-fields.
	\end{example}
\begin{proposition}
Let $(\Omega,\FF,\PP)$  be a probability space and let $(X,Y)$ be a bivariate absolutely continuous random vector. Denote $f_{(X,Y)}:=\frac{\dd\PP_{(X,Y)}}{\dd \leb^2}$. Then $X$ and $Y$ are both absolutely continuous, 
\begin{equation}\label{eq:abs:density} 
f_X(x)=\int f_{(X,Y)}(x,y)\leb(dy)\text{ for $\leb$-a.e. }x\in \mathbb{R}\text{ and } f_Y(y)=\int f_{(X,Y)}(x,y)\leb(dx)\text{ for $\leb$-a.e. }y\in \mathbb{R},
\end{equation}
 and $X$ is independent of $Y$ iff 
\begin{equation*}
f_{(X,Y)}(x,y)=f_X(x)f_Y(y) \text{ for $\leb^2$-a.e. $(x,y)\in \mathbb{R}^2$}.
\end{equation*}
When $X$ and $Y$ are independent, then, for all $g\in \mathcal{B}_{\mathbb{R}^2}/\mathcal{B}_{[-\infty,\infty]}$,
\begin{equation}
\PP[g(X,Y)]=\int\int g(x,y)f_X(x)f_Y(y)\leb(\dd x)\leb(\dd y),
\end{equation}
provided $\PP[g^-(X,Y)]\land \PP[g^+(X,Y)]<\infty$. 
\end{proposition}
\begin{definition}
We retain in what follows the notation $f_{(X,Y)}$ for the density of a bivariate absolutely continuous random vector $(X,Y)$. 
\end{definition}
\begin{remark}
One could generalize to higher-dimensional random vectors (whose laws are absolutely continuous w.r.t. some product of $\sigma$-finite reference measures) in  a clear way. 
\end{remark}
\begin{proof} 
 For $A\in \mathcal{B}_\mathbb{R}$, by Tonelli, $$\PP_X(A)=\PP(X\in A)=\PP((X,Y)\in A\times\mathbb{R})=\int_{A\times \mathbb{R}}f_{(X,Y)}(x,y)\leb^2(\dd(x,y))=\int_A\int f_{(X,Y)}(x,y)\leb(\dd y)\leb(\dd x).$$ Formula~\eqref{eq:abs:density} follows. Now, $X$ is independent of $Y$ iff for all $(A,B)\in \mathcal{B}_\mathbb{R}\times \mathcal{B}_\mathbb{R}$, $\PP(X\in A,Y\in B)=\PP(X\in A)\PP(Y\in B)$, i.e. (via Tonelli, again) $$\int_{A\times B} f_{(X,Y)}(x,y)\leb^2(\dd (x,y))=\int_{A\times B} f_X(x)f_Y(y)\leb^2(\dd (x,y)),$$
	which in turn is the same as $f_{(X,Y)}(x,y)=f_X(x)f_Y(y)$ for $\leb^2$-a.e. $(x,y)\in \mathbb{R}^2$ (recall Corollary~\ref{corollary}). 	\end{proof}

\begin{definition}
	Let $((\Omega_\lambda,\FF_\lambda))_{\lambda\in \Lambda}$ be a family of measurable spaces; we set $$\otimes_{\lambda\in \Lambda}\FF_\lambda:=\lor_{\lambda\in\Lambda}\sigma^{\FF_\lambda}(\pr_\lambda),$$ where $\pr_\lambda:\prod_{\mu\in \Lambda}\Omega_\mu\to \Omega_\lambda$, $\lambda\in \Lambda$, are the canonical projections, and call it the product (tensor) $\sigma$-algebra. For a $\sigma$-field  $\FF$ and any set $\Lambda$ we put $\FF^{\otimes \Lambda}:=\otimes_{\lambda\in\Lambda}\FF$.
\end{definition}
\begin{remark}
For a $\sigma$-field $\FF$, $\FF^{\otimes \emptyset}$ is to be understood as the trivial (indeed it is the only) $\sigma$-algebra on $\{\emptyset\}$.
\end{remark}
\begin{proposition}\label{proposition:arbitrary-products}
		Let $((\Omega_\lambda,\FF_\lambda))_{\lambda\in \Lambda}$ be a family of measurable spaces. 
		\begin{enumerate}[(i)]
		\item\label{product--arbit:i} $\otimes_{\lambda\in \Lambda}\FF_\lambda$ is the smallest $\sigma$-field $\GG$ on $\prod_{\lambda\in \Lambda}\Omega_\lambda$ for which  $\pr_\lambda\in \GG/\FF_\lambda$ for all $\lambda\in \Lambda$ (i.e. such as makes all the canonical projections   measurable). 
		\item\label{product--arbit:ii} For a measurable space $(\Omega,\FF)$ and a family of functions $f_\lambda:\Omega\to \Omega_\lambda$, as $\lambda$ ranges over $\Lambda$, we have that $f_\lambda\in \FF/\FF_\lambda$ for all $\lambda\in\Lambda$ iff $(f_\lambda)_{\lambda\in\Lambda}\in \FF/\otimes_{\lambda\in \Lambda}\FF_\lambda$. 
		\item\label{product--arbit:iii} Let there be given a probability $\mu_\lambda$  on $\FF_\lambda$ for each $\lambda\in \Lambda$. There exists a unique probability measure $\mu$ on $\otimes_{\lambda\in \Lambda}\FF_\lambda$ such that $\mu\circ\pr_\lambda^{-1}=\mu_\lambda$ for all $\lambda\in \Lambda$ and rendering $(\pr_\lambda)_{\lambda\in \Lambda}$ an independency under $\mu$. 
		\end{enumerate}
\end{proposition}
\begin{definition}
	The $\mu$ of Proposition~\ref{proposition:arbitrary-products}\eqref{product--arbit:iii} is called the product of $(\mu_\lambda)_{\lambda\in\Lambda}$ and denoted $\times_{\lambda\in \Lambda}\mu_\lambda$. For a probability $\nu$ and any set $\Lambda$ we put $\nu^{\times\Lambda}:=\times_{\lambda\in \Lambda}\nu$ (of course $\nu^{\times \emptyset}=\delta_\emptyset$, a probability on $(\{\emptyset\},2^{\{\emptyset\}})$, no matter what the $\nu$).
	\end{definition}
	\begin{remark}
For $n\in \mathbb{N}_0$, $\nu^n=\nu^{\times [n]}$  up to the natural identification of $\Omega^n$ with $\Omega^{[n]}$, $\Omega$ being the sample space of the probability $\nu$.
	\end{remark} 
\begin{proof}[Proof of \eqref{product--arbit:i}, \eqref{product--arbit:ii} and uniqueness in \eqref{product--arbit:iii}]  \eqref{product--arbit:i} is clear. For  \eqref{product--arbit:ii} recall Propositions~\ref{proposition:compositions} and~\ref{proposition:maps}\eqref{maps:iv}. The uniqueness in \eqref{product--arbit:iii} follows from Proposition~\ref{proposition:equality of measures} taking for the $\pi$-system the collection of sets of the form $\cap_{\lambda\in G}\pr_\lambda^{-1}(F_\lambda)$, where $F_\lambda\in\FF_\lambda$ for $\lambda\in G$, and where $G$ runs over all the (non-empty) finite subsets of $\Lambda$. 
	\end{proof}	
	\begin{proof}[Proof of existence in \eqref{product--arbit:iii}] 
	Because $\otimes_{\lambda\in \Lambda}\FF_\lambda$ is the union of the $\sigma$-fields generated by countable subfamilies of $(\pr_\lambda)_{\lambda\in \Lambda}$ one reduces at once to the case when $\Lambda$ is countable. The finite case being clear by Theorem~\ref{thm:tonelli-fubini} or in any case, we may just as well assume $\Lambda=\mathbb{N}$. For each $n\in \mathbb{N}$, Theorem~\ref{thm:tonelli-fubini} gives us a unique measure $\gamma_n$ on the $\sigma$-field $\GG_n$ generated by the projections $\pr_i$, $i\in [n]$, such that $\gamma_n\circ (\pr_{[n]})^{-1}=\mu_1\times \cdots \times \mu_n$. We see that $(\gamma_n)_{n\in \mathbb{N}}$ is an increasing sequence of functions; the map $\gamma_\infty:=\cup_{n\in \mathbb{N}}\gamma_n$ is defined on the algebra $\GG_\infty:=\cup_{n\in \mathbb{N}}\GG_n$, which generates $\otimes_{n\in \mathbb{N}}\FF_n$. By the theorem of Carath\'eodory it suffices to check that $\gamma_\infty$ is countably additive; it being clearly additive, it will further be enough to take an arbitrary $\downarrow$ sequence $(A_n)_{n\in \mathbb{N}}$ in $\GG_\infty$ with $\epsilon:=\inf_{n\in \mathbb{N}}\gamma_\infty(A_n)>0$ and check that $A:=\cap_{n\in \mathbb{N}}A_n\ne \emptyset$. Without loss of generality we may assume that for each $n\in \mathbb{N}$ we have $\mathbbm{1}_{A_n}=g_n(\pr_{[n]})$ for some $g_n\in  (\otimes_{i\in [n]}\FF_i)/2^{\{0,1\}}$. By Tonelli $\mu_1^{\omega_1}[\cdots \mu^{\omega_n}_n[g_n(\omega_1,\ldots,\omega_n)]\cdots]=\gamma_n(A_n)=\gamma_\infty(A_n)\geq \epsilon$ and $g_{n+1}(\pr_{[n+1]})=\mathbbm{1}_{A_{n+1}}\leq \mathbbm{1}_{A_n}=g_{n}(\pr_{[n]})$ for each $n\in \mathbb{N}$. (We have written everything with functions rather than sets for ease of notation, but only the sets matter, of course.) 
We construct an $\tilde\omega\in A \subset \prod_{n\in \mathbb{N}}\Omega_n$ by inductively defining its coordinates as follows. For  $n\in \mathbb{N}$ define $f_n:\Omega_1\to [0,1]$ by setting $f_n(\omega_1):=\mu_2^{\omega_2}[\cdots \mu_n^{\omega_n}[g_n(\omega_1,\omega_2,\ldots,\omega_n)]\cdots]$ for $\omega_1\in \Omega_1$; then $f_n(\omega_1)$ is $\downarrow$ in $n\in \mathbb{N}$ for each $\omega_1\in\Omega_1$. By dominated convergence $\mu_1^{\omega_1}[\inf_{n\in \mathbb{N}}f_n(\omega_1)]\geq \epsilon$. Therefore there exists at least one $\tilde\omega_1\in \Omega_1$ such that $f_n(\tilde\omega_1)\geq \epsilon$ for each $n\in \mathbb{N}$. In particular $g_1(\tilde\omega_1)=f_1(\tilde\omega_1)\geq \epsilon$. Now repeat the preceding for $(\mu_{i+1})_{i\in \mathbb{N}}$ in lieu of $(\mu_{i})_{i\in \mathbb{N}}$ and for $(g_{i+1}(\tilde\omega_1,\cdot))_{i\in \mathbb{N}}$ in lieu of  $(g_{i})_{i\in \mathbb{N}}$. You get $\tilde \omega_2\in \Omega_2$ such that $g_2(\tilde{\omega}_1,\tilde{\omega}_2)\geq \epsilon$. ``Und so weiter und so fort.'' Then $g_n(\tilde{\omega})\geq \epsilon$ for all $n\in \mathbb{N}$. Therefore $\mathbbm{1}_A(\tilde{\omega})=\lim_{n\to\infty}g_n(\pr_{[n]}(\tilde\omega))\geq \epsilon$, which means that $A\ne \emptyset$.
	\end{proof}
	 
\begin{proposition}
Let $(\Omega,\FF,\PP)$ be a probability space and let $f_\lambda$ be an $(E_\lambda,\EE_\lambda)$-valued random element, as $\lambda$ ranges over some index set $\Lambda$. Then $(f_\lambda)_{\lambda\in\Lambda}$ is an independency under $\PP$ iff $((f_\lambda)_{\lambda\in\Lambda})_\star\PP=\times_{\lambda\in\Lambda}{f_\lambda}_\star \PP$. 
\end{proposition}
\begin{proof}
Let $\pr_\lambda$, $\lambda\in \Lambda$, be the canonical projections on $\prod_{\lambda\in \Lambda}E_\lambda$. For finite $G\subset \Lambda$ and $A\in \prod_{\lambda\in  G}\EE_\lambda$ one has, from the definitions,
\begin{align*}
(((f_\lambda)_{\lambda\in\Lambda})_\star\PP)(\cap_{\lambda\in G}\pr_\lambda^{-1}(A_\lambda))&=\PP(\cap_{\lambda\in G}\{f_\lambda\in A_\lambda\})\text{ on the one hand},\\
(\times_{\lambda\in\Lambda}{f_\lambda}_\star \PP)(\cap_{\lambda\in G}\pr_\lambda^{-1}(A_\lambda))&=\prod_{\lambda\in G}({f_\lambda}_\star \PP)(A_\lambda)\quad (\because\, (\pr_\lambda)_{\lambda\in \Lambda} \text{ is an independency under }\times_{\lambda\in\Lambda}{f_\lambda}_\star \PP)\\
&=\prod_{\lambda\in G}\PP(f_\lambda\in A_\lambda)\text{ on the other ($\because$ $(\pr_\lambda)_\star (\times_{\mu\in\Lambda}{f_\mu}_\star \PP)={f_\lambda}_\star \PP$ for all $\lambda\in \Lambda$};
\end{align*}
it remains to note that sets of the form $\cap_{\lambda\in G}\pr_\lambda^{-1}(A_\lambda)$ form a $\pi$-system generating $\otimes_{\lambda\in \Lambda}\EE_\lambda$ on $\prod_{\lambda \in\Lambda}E_\lambda$ and to apply Proposition~\ref{proposition:equality of measures}. 
\end{proof}
	
\section{Conditioning}
\emph{conditioning on an event; conditional expectations w.r.t. a sub-$\sigma$-field: definition, properties, and means of computation}
\begin{proposition}
If $(\Omega,\AA,\PP)$ is a probability space and $A\in\AA$ with $\PP(A)>0$, then $(A,\AA\vert_A,\frac{1}{\PP(A)}\PP_A)$ is also a probability space, furthermore $\left(\frac{1}{\PP(A)}\PP_A\right)[f]=\PP[f;A]/\PP(A)$ for all $f\in(\AA\vert_A)/\mathcal{B}_{[-\infty,\infty]}$. \qed
\end{proposition}
\begin{definition}
Let $(\Omega,\AA,\PP)$ be a probability space and $A\in \AA$, $\PP(A)>0$. Then for $B\in \AA$, we put $\PP(B\vert A):=\PP(A\cap B)/\PP(A)$, the conditional  probability of $B$ given $A$ under $\PP$. Write $\PP(\cdot\vert A):=\frac{1}{\PP(A)}\PP_A$. For further $f\in \AA/\mathcal{B}_{[-\infty,\infty]}$ or for $f:\Omega\to \mathbb{C}$ with $\{\Re f,\Im f\}\subset \LL^1(\PP)$, $\PP[f\vert A]:=\PP[f;A]/\PP(A)$, the conditional expectation of $f$ given  $A$ under $\PP$.
\end{definition}
\begin{proposition}\label{proposition:conditional-expectation}
Let $(\Omega,\AA,\PP)$ be a probability space, $\BB$ a sub-$\sigma$-field of $\AA$,  $f\in \AA/\mathcal{B}_{[-\infty,\infty]}$, $\PP[f^+]\land \PP[ f^-]<\infty$.  Then there exists a up to $\PP$-a.s. equality unique $g\in \BB/\mathcal{B}_{[-\infty,\infty]}$ with $\PP [g^+]\land \PP [g^-]<\infty$ such that 
\begin{equation}\label{definingg}
\PP[f;B]=\PP[g;B]\text{ for all }B\in \BB,\text{ i.e. }(f\cdot \PP)\vert_\BB=g\cdot (\PP\vert_\BB);
\end{equation}
this $g$ is $\PP$-a.s. $=\frac{\dd ((f\cdot \PP)\vert_\BB)}{\dd (\PP\vert_\BB)}$  provided $f\geq 0$ and $\PP[f]<\infty$.
\end{proposition}
\begin{remark}
Condition \eqref{definingg} may be summarized informally as follows: testing (averaging) $g$ on a member of $\BB$ I get the same thing as if I test (average) $f$. Indeed for $B\in \BB$ for which $\PP(B)=0$ the condition in \eqref{definingg} holds trivially, while for $B\in \BB$ for which $\PP(B)>0$ it may be rewritten as the equality of the conditional expectations given the event $B$: $\PP[f\vert B]=\PP[g\vert B]$.
\end{remark}
\begin{proof}
Uniqueness.  Corollary~\ref{corollary} for $\PP\vert_\BB$. Existence. Suppose first $f \geq 0$ and $\PP[f]<\infty$. Then $(f\cdot \PP)\vert_\BB\ll \PP\vert_\BB$ are finite measures and we can (indeed) take  $g:=\frac{\dd ((f\cdot \PP)\vert_\BB)}{\dd (\PP\vert_\BB)}$. If (merely) $f\geq 0$, for each $n\in \mathbb{N}$, take a $g_n$ corresponding to $f_n:=f\land n$: by \eqref{definingg} (for the pair $(f_n,g_n)$) and Proposition~\ref{lemma}\eqref{lemma:b} $0\leq g_n\leq g_{n+1}$ for all $n\in \mathbb{N}$ a.s.-$\PP$; put $g:=\limsup_{n\to\infty}g_n\in \BB/\mathcal{B}_{[-\infty,\infty]}$, by monotone convergence it satisfies \eqref{definingg}. Finally, for a general $f$ take $g_+$ corresponding to $f^+$ and $g_-$ corresponding to $f^-$, set $g:=g_+-g_-\in \BB/\mathcal{B}_{[-\infty,\infty]}$ and check that it verifies \eqref{definingg}.
\end{proof}
\begin{definition}
	We denote the $g$ from the previous proposition variously with\footnote{It cannot be confused with $\PP[f\vert \BB]$ for $\BB\in \FF$ with $\PP(\BB)>0$; for, a sub-$\sigma$-field $\BB$ of $\AA$ cannot also be an element of $\AA$ (it would entail that $\Omega\in \BB\subset \Omega$, which is disallowed).} $\PP[f\vert\BB]:=\PP_\BB(f):=\PP_\BB f$ and call it the conditional expectation of $f$ w.r.t. $\BB$ under $\PP$, insisting in addition (as we may) that $\PP[f\vert \BB]$ is  $\geq 0$ when $f$ is so, also finite-valued when $\PP[\vert f\vert]<\infty$ (for in such case \eqref{definingg} [with $B=\Omega$] entails $\PP[\vert g\vert]<\infty$, a fortiori $\PP(\vert g\vert=\infty)=0$, and we may replace $g$ with $g\mathbbm{1}_{\{g\in \mathbb{R}\}}$). For $A\in \AA$ we set $\PP_\BB A:=\PP_\BB(A):=\PP(A\vert \BB):=\PP[\mathbbm{1}_A\vert\BB]$ and call it the conditional probability of $A$ w.r.t $\BB$ under $\PP$. Such conditional expectation (probability) is only defined uniquely up to a.s. equality; when a statement involving it appears, it means that it is being asserted for any of its versions. Parallel to the notation $\PP_\BB f$ we also write $\PP f:=\PP[f]$ for short. 
\end{definition} 
\begin{remarks}
Because conditional expectations are only defined uniquely  up to a.s. equality, most statements involving conditional expectations are subject to the a.s. qualifier.  As long as at most denumerably many conditional expectations are under inspection this carries little significance; conversely, if one considers more than denumerably many conditional expectations at once, then usually some choice of the versions is imposed to achieve sufficient regularity of such a family.  Another way around the issue would be (is) to use equivalence classes w.r.t. a.s. equality, but it seems less ``forgiving''. In literature one tends to find the notation $\mathsf{E}_\PP[f\vert\BB]:=\PP[f\vert\BB]$ (or $\mathsf{E}[f\vert\BB]$ if no ambiguity can arise), but we shall stick with (the simpler) $\PP[f\vert\BB]=\PP_\BB(f)=\PP_\BB f$. Intuitively $\PP[f\vert \BB]$ is supposed to be the formalization of the notion of an ``expectation of $f$ given the information contained in $\BB$''; here the ``information'' of $\BB$ may be interpreted as ``knowing of each event in $\BB$ whether or not it has occured''. This will gradually become clear as we explore its properties. Conditioning is of paramount importance in probability because it formalizes the idea of ``taking into account what is known''. The proper treatment of such objects as Markov processes and martingales, both exceptionally important classes of stochastic processes, hinges crucially on the notion of conditional expectations. If the need arises, one extends, by linearity, $\PP[f\vert \BB]$ to include complex-valued $f$ for which $\{\Re f,\Im f\}\subset \LL^1(\PP)$. We will keep to the real case, but the complex extensions are usually straightforward (because of linearity).
 
Sometimes the following notion of a ``generalized'' conditional expectation is also introduced (and is useful): dropping in Proposition~\ref{proposition:conditional-expectation} the requirement that $\PP[f^+]\land\PP[f^-]<\infty$, but asking instead that $\PP[f^+\vert \BB]<\infty$ a.s.-$\PP$ and $\PP[f^-\vert \BB]<\infty$ a.s.-$\PP$ (the latter can only happen when $\vert f\vert<\infty$ a.s.-$\PP$, but it may happen even if $\PP[f^+]=\PP[f^-]=\infty$, and it is equivalent to $f^+\cdot \PP\vert_{\BB}$ and $f^-\cdot \PP\vert_\BB$ being $\sigma$-finite), then the difference $\PP[f^+\vert\BB]-\PP[f^-\vert\BB]$ is meaningfully defined to be $\PP[f\vert \BB]$. In fact, it should presumably be enough to ask even just for $\PP[f^+\vert \BB]\land \PP[f^-\vert \BB]<\infty$ a.s.-$\PP$.
 
\end{remarks}

\begin{proposition}\label{proposition}
	Let $(\Omega,\AA,\PP)$ be a probability space, $\BB$ be a sub-$\sigma$-field of $\AA$ and $f\in \AA/\mathcal{B}_{[-\infty,\infty]}$ with $\PP[ f^+]\land \PP [f^-]<\infty$. The conditional expectation enjoys the following properties.
		\begin{enumerate}[(1)]
			\item\label{stability} (Stability.) If $f\in \BB/\mathcal{B}_{[-\infty,\infty]}$, then $\PP[ f\vert \BB]=f$ a.s.-$\PP$.
			\item\label{tower} (Law of total expectation/tower property.) $\PP[\PP[f\vert \BB]]=\PP[ f]$.
			\item Let $\CC$ be another sub-$\sigma$-field of $\AA$.
			
			\begin{enumerate}[(a)]
				\item \label{repeated} (Repeated conditioning/tower property (bis).) If $\BB$ and $\CC$ are comparable with respect to inclusion, then $ \PP_\BB\PP_\CC f=\PP_{\BB\cap\CC}f$ a.s.-$\PP$.
				\item \label{negligible} (Irrelevance of trivial events.) If $\BB\lor \PP^{-1}(\{0,1\})=\CC\lor \PP^{-1}(\{0,1\})$, then $\PP_\BB f=\PP_\CC f$ a.s.-$\PP$.
			\end{enumerate}
			\item \label{ae-equal} (Irrelevance of negligible events.)  If  also $g\in  \AA/\mathcal{B}_{[-\infty,\infty]}$  with $f=g$ a.s.-$\PP$, then $\PP g^-\land \PP g^+<\infty$ and $\PP_\BB f=\PP_\BB g$ a.s.-$\PP$. 
			\item\label{trivial} (Conditioning w.r.t. a trivial $\sigma$-field.) If $\BB\subset \PP^{-1}(\{0,1\})$, then  $\PP$-a.s. $\PP_\BB f=\PP f$. In particular $\PP_{\{\emptyset,\Omega\}} f=\PP f$.
			\item\label{linearity} (Additivity.) If also $g\in  \AA/\mathcal{B}_{[-\infty,\infty]}$ and $\PP f^-\lor \PP g^-<\infty$, then $\PP_\BB(f+g)=\PP_\BB f+\PP_\BB g$ a.s.-$\PP$. 
			\item\label{limited} (Sufficient condition for defining property.) Assume $\PP\vert f\vert<\infty$. Let also $g\in \mathcal{B}/\mathcal{B}_{[-\infty,\infty]}$, $\PP\vert g\vert< \infty$. In order that $\PP_\BB f=g$ a.s.-$\PP$ it is sufficient that $\PP[f;B]=\PP[g;B]$ for all $B\in \Pi\cup \{\Omega\}$, where $\Pi\subset 2^\Omega$ is some $\pi$-system  with $\sigma_\Omega(\Pi)=\BB$.
			\item\label{defining} (Defining property extended.) For $h\in \BB/\mathcal{B}_{[-\infty,\infty]}$ with $\PP(hf)^+\land \PP(hf)^-<\infty$ one has $\PP(h\PP_\BB f)^+\land \PP(h\PP_\BB f)^-<\infty$ and $\PP[hf]=\PP[h\PP_\BB f]$. 
			\item\label{conditional-determinism} (Conditional determinism/``taking out what is known''. Homogeneity.) If $g\in \BB/\mathcal{B}_{[-\infty,\infty]}$ with $\PP (fg)^+\land \PP (fg)^-<\infty$, then $\PP_\BB(gf)=g\PP_\BB f$ a.s.-$\PP$. In particular if $c\in \mathbb{R}$, then $\PP_\BB(cf)=c\PP_\BB f$ a.s.-$\PP$.
			\item \label{monotonicity} (Monotonicity.) If also $g\in \AA/\mathcal{B}_{[-\infty,\infty]}$ with $\PP g^-\land \PP g^+<\infty$ and if $f\leq g$ a.s.-$\PP$ (remember: the latter is equivalent to $\PP[f;B]\leq \PP[g;B]$ for every $B\in \AA$), then $\PP_\BB f\leq \PP_\BB g$ a.s.-$\PP$.
			\item\label{triangle} (Triangle inequality.) $\vert \PP_\BB f\vert\leq \PP_\BB\vert f\vert$ a.s.-$\PP$. 
			
			\item In this item only we drop the a priori assumption that $\PP f^-\land \PP f^+<\infty$.  Let $(f_n)_{n\in \mathbb{N}}$ be a sequence in $\AA/\mathcal{B}_{[-\infty,\infty]}$ and $g$ an element of $\AA/\mathcal{B}_{[0,\infty]}$, such that $\PP g<\infty$ and $f_n^-\leq g$ a.s.-$\PP$ for all $n\in \mathbb{N}$. 
			
			\begin{enumerate}[(a)]
				\item \label{monotone-convergence} (Monotone convergence.) If $f_n\uparrow f$ as $n\to \infty$ a.s.-$\PP$, then $\PP f^-<\infty$ and $\PP_\BB f_n\uparrow \PP_\BB f$ a.s.-$\PP$.
				\item \label{fatou} (Fatou's lemma.) $\PP(\liminf_{n\to\infty}f_n)^-<\infty$ and $\PP_\BB\liminf_{n\to\infty} f_n\leq \liminf_{n\to\infty}\PP_\BB f_n$ a.s.-$\PP$.
				\item \label{dominated} (Dominated convergence.) If $\vert f_n\vert\leq g$ a.s.-$\PP$ for all $n\in \mathbb{N}$ and $f_n\to f$ as $n\to\infty$ a.s.-$\PP$, then $\PP \vert f\vert<\infty$ and $\PP_\BB f=\lim_{n\to\infty}\PP_\BB f_n$ a.s.-$\PP$ and also in $\LL^1(\PP)$.
			\end{enumerate}
			\item
			\begin{enumerate}[(a)]
				\item\label{independent-conditioning} (Independent conditioning.) If $\BB'$, another sub-$\sigma$-field of $\AA$, and $f'\in \AA/\mathcal{B}_{[-\infty,\infty]}$ with $\PP (f')^+\land \PP (f')^-<\infty$, are such that $\BB\lor \sigma(f)$ is independent of $\BB'\lor \sigma(f')$ under $\PP$, and if furthermore $\PP(ff')^+\land \PP(ff')^-<\infty$, then $\PP_{\BB\lor \BB'}(ff')=(\PP_{\BB}f)(\PP_{\BB'}f')$ a.s.-$\PP$.
				\item\label{independent-enlargement} (Irrelevance of independent events.) If $\CC$ is a sub-$\sigma$-field of $\AA$ that is independent of $\sigma(f)\lor \BB$ under $\PP$, then $\PP_{\BB\lor \CC}f=\PP_\BB f$ a.s.-$\PP$. 
				\item\label{independent} (Conditioning w.r.t. an independent $\sigma$-field.) If  $\BB$ is independent of $\sigma(f)$ under $\PP$, then $\PP_\BB f=\PP f$ a.s.-$\PP$. 
			\end{enumerate}
						 \item\label{jensen} (Jensen's inequality.) If $\phi:I\to\mathbb{R}$ is a convex function, $I$ an interval of $\mathbb{R}$, and if $f$ takes values in $I$ with $\PP [\vert f\vert]<\infty$, then $\PP(\PP[f\vert\BB ]\in I)=1$,  $\phi\circ f\in \AA/\mathcal{B}_\mathbb{R}$, $\PP [\phi^-\circ f]<\infty$, $\PP[ \phi\circ f\vert \BB]\geq \phi\circ \PP[ f\vert \BB]$ a.s.-$\PP$, and $f=\PP[f\vert \BB ]$ a.s.-$\PP$ on $\{\PP[f\vert \BB ]\in \partial I\}$ (where, on the $\PP$-negligible set on which this is relevant, we set e.g. $ \phi(x):=0$ for $x\in\mathbb{R}\backslash I$; $\partial I=\overline{I}\backslash \overset{\circ}{I}$ is the boundary of $I$ in $\mathbb{R}$). 
		\end{enumerate}
\end{proposition}
 
\begin{proof} \eqref{stability} and \eqref{repeated} are immediate from the definition. \eqref{tower}. Take $B=\Omega$ in \eqref{definingg}. \eqref{negligible}.  Note that $\GG:=\BB\lor \PP^{-1}(\{0,1\})= \{A\in \AA:\exists B\in \BB\text{ such that }\PP(A\triangle B)=0\}$. Then check that $\PP$-a.s. $\PP_\GG  f=\PP_\BB f$  and (hence) $\PP_\GG  f=\PP_\CC f$. \eqref{ae-equal} and \eqref{trivial}. Immediate from the definition, because (unconditional) integrals do not see sets of measure zero. \eqref{linearity}. One checks the defining property for $\PP_\BB(f+g)$, taking into account additivity of the (unconditional) integral (apart from the more trivial observations). \eqref{limited}. $\Pi\cup \{\Omega\}$ is still a $\pi$-system; apply Dynkin's lemma. \eqref{defining}. For $h$ indicators it is the defining property of $\PP_\BB f$. The general case follows by the usual approximation, first for $f\geq 0$, $h\geq 0$. Then  for $f\geq 0$, $h$ general, by taking differences. Finally for a general $f$, again by taking differences (note that $(hf)^+=h^+f^++h^-f^-$, $(hf)^-=h^-f^++h^+f^-$, $(hf^+) ^+=h^+f^+$, $(hf^+)^-=h^-f^+$ etc.). \eqref{conditional-determinism}. Follows from the definition of the conditional expectation, and from the previous item. \eqref{monotonicity}. 
Follows directly from the definition and Lemma~\ref{lemma}. \eqref{triangle}. Follows from the previous item and $-\vert f\vert\leq f\leq \vert f\vert$. \eqref{monotone-convergence}. Follows from the penultimate item, the parallel property for the (unconditional) integral and the definition of the conditional expectation. \eqref{fatou}. Follows from the previous item and monotonicity precisely as in the unconditional case. \eqref{dominated}. Let $n\in \mathbb{N}$; because $\PP \vert f_n\vert<\infty$ we may and do take  $f_n$ real-valued (while finite values are automatic for $\PP_\BB f_n$ by definition). Put $Z_n:=\sup_{k\in \mathbb{N}_{\geq n}}\vert f_k-f\vert$ for $n\in \mathbb{N}$. Then  $0\leq Z_n\leq 2 g$ a.s.-$\PP$ for all $n\in \mathbb{N}$ and $Z_n\downarrow 0$ as $n\to \infty$ a.s.-$\PP$. By (unconditional) dominated convergence, $\PP Z_n\to 0$ as $n\to\infty$. Then  by linearity, the triangle inequality and the tower property $\PP\vert \PP_\BB f_n-\PP_\BB f\vert=\PP\vert \PP_\BB(f_n-f)\vert\leq \PP \PP_\BB \vert f_n-f\vert=\PP \vert f_n-f\vert\leq \PP Z_n\to 0$ as $n\to \infty$ proving convergence in $\LL^1(\PP)$. Set $Z:=\liminf_{n\to\infty}\PP_\BB Z_n$ (note that by monotonicity the $\liminf$ is equal to the $\limsup$ a.s.-$\PP$). By (unconditional) Fatou and the tower property $\PP Z\leq \liminf_{n\to\infty}\PP \PP_\BB Z_n=\lim_{n\to\infty}\PP Z_n=0$, so that $Z=0$ a.s.-$\PP$. Finally, $\vert \PP_\BB f_n-\PP_\BB f\vert\leq \PP_\BB \vert f_n-f\vert\leq \PP_\BB Z_n$ a.s.-$\PP$ for all $n\in \mathbb{N}$, which yields the a.s.-$\PP$ convergence. \eqref{independent-conditioning}. By the usual reduction techniques it will suffice to handle the case when $f$ and $f'$ are indicator functions, say of $F$ and $F'$, respectively. Then it will suffice to check that $\PP(F\cap F'\cap G)=\PP[\PP_\BB(F)\PP_{\BB'}(F');G]$ for $G$ belonging to the $\pi$-system $\{B\cap B':(B,B')\in \BB\times \BB'\}$ that generates $\BB\lor \BB'$ on $\Omega$. But this follows at once from the assumed independence.  [It is easy to check that if $\{h,m\}\subset \AA/\mathcal{B}_{[-\infty,\infty]}$ with $\sigma(h)$ independent of $\sigma(m)$, then $\PP[hm]=\PP[h]\PP[m]$, whenever all the integrals are well-defined.]
\eqref{independent-enlargement}. Use the previous item with $\BB'=\CC$ and $f'=1$. 
\eqref{independent}. Follows from the previous item upon taking $\BB=\{\emptyset,\Omega\}$ therein. 
\eqref{jensen}. Thanks to $\PP\vert f\vert<\infty$ the quantity $\PP_\BB f$ is real-valued (by definition).  By monotonicity $\PP(\PP_\BB f\in \overline{I})=1$. 
The conclusions $\phi\circ f\in \AA/\mathcal{B}_\mathbb{R}$ and $\PP [\phi^-\circ f]<\infty$ are part of the unconditional version of Jensen's inequality. Indeed $\phi\vert_{\overset{\circ}{I}}$ is continuous and in particular $\phi\in \mathcal{B}_I/\mathcal{B}_\mathbb{R}$. 
Next, if say $0$ is a left-endpoint of $I$, then $\PP[f;\{\PP_\BB f=0\}]=\PP[\PP_\BB f;\{\PP_\BB f=0\}]=0$ and so $f=0$ a.s.-$\PP$ on $\{\PP_\BB f=0\}$. The case when the left endpoint is some other real number or there is a finite right-endpoint is treated analogously (by a suitable translation and/or reflection). Hence on the event $\{\PP_\BB f\in \partial I\}$ one has $f=\PP_\BB f$ a.s.-$\PP$, we see that $\PP(\PP_\BB f\in I)=1$ and in particular on $\{\PP_\BB f\in \partial I\}$ the target inequality is trivial. To prove it on $A:=\{\PP_\BB f\in \overset{\circ}{I}\}\in \BB$, let $\mathfrak{A}$ be the set of the affine minorants (with real slope, intercept) of $\phi$ and note that, \emph{on} $\overset{\circ}{I}$, $\phi$ is the pointwise supremum  of the members of $\mathfrak{A}$.\footnoteref{foot:convex} 
Then $\phi\circ f\geq a\circ f$, and hence by linearity and since $\PP$ is a probability measure, $\PP[\PP_\BB(\phi \circ f);B]=\PP[\phi \circ f;B]\geq \PP[a\circ f;B]=\PP[\PP_\BB(a\circ f);B]=\PP[a\circ \PP_\BB f;B]$ for all $a\in \mathfrak{A}$, $B\in \mathcal{B}$. By Lemma~\ref{lemma}, it follows that $\PP_\BB (\phi\circ f)\geq a\circ (\PP_\BB f)$ a.s.-$\PP$. Now take the supremum over those $a\in \mathfrak{A}$ with rational slope and intercept (of which there are denumerably many) to obtain the desired conclusion. 
\end{proof}
 
\begin{remarks}
	In essence all of the properties of the expectation carry over to the corresponding properties of conditional expectations, mutatis mutandis. We have listed in the preceding proposition some, but not nearly all, and not even all  the interesting/useful instances of this phenomenon\footnote{A more complete account can be found \underline{\href{https://drive.google.com/file/d/12d0uQj_IJk882UGiAzByl5DZtO9S7tBG/view?usp=sharing}{here}}.}. Because of \eqref{trivial} statements involving conditional expectations w.r.t. arbitrary sub-$\sigma$-fields always have as special cases statements involving the ordinary (unconditional) expectation, for instance \eqref{jensen} tells us that $f=\PP[f]$ a.s. if $\PP[f]\in \partial I$. As always, additivity \eqref{linearity} and homogeneity \eqref{conditional-determinism} combine to give linearity.
	   Just like expectations are only a special case of integrals, so too conditional expectations may be generalized to ``conditional integrals'' (modulo some technical caveats). 
\end{remarks}

\begin{proposition}\label{proposition:methods}
	Let $(\Omega,\AA,\PP)$ be a probability space. Some methods for directly computing the conditional expectation are as follows. 
	\begin{enumerate}[(I)]
		\item\label{means:0} (Discrete conditioning $\sigma$-field.) Let $\II\in 2^\AA$ be a countable partition of $\Omega$. Then for $X\in \AA/\mathcal{B}_{[-\infty,\infty]}$ with $\PP [X^+]\land \PP [X^-]<\infty$, $\PP$-a.s. 
		\begin{equation}
\PP[X\vert \sigma_\Omega(\II)]=\sum_{I\in \II\cap \PP^{-1}((0,1])}\PP[X\vert I]\mathbbm{1}_I
		\end{equation}
(so $\PP[X\vert \sigma_\Omega(\II)]=\PP[X\vert I]$ on\footnote{No a.s. qualifier is needed here. True, conditional expectations are in general defined only up to a.s. equality. However, in this case the measurability requirement forces the constancy of $\PP_{\sigma_\Omega(\II)}(X)$ on every member of $I\in\II$ and then the ``testing condition'' determines this value provided $\PP(I)>0$.} $I$, for all $I\in \II$ with $\PP(I)>0$).
		\item Let $(F,\FF)$ and  $(E,\EE)$  be measurable spaces, $X\in \AA/\FF$, $Y\in \AA/\EE$ two random elements. Let also $h\in (\FF\otimes \EE)/\mathcal{B}_{[-\infty,\infty]}$ be such that $\PP[(h\circ (X,Y))^+]\land \PP[(h\circ (X,Y))^-]<\infty$. 
		\begin{enumerate}[(1)]
			\item\label{cond.b} (Absolutely continuous random elements.) Suppose $\ee$ is a $\sigma$-finite measure on $(E,\EE)$ and $\ff$ a $\sigma$-finite measure on $(F,\FF)$ and suppose that $\PP_{(X,Y)}\ll\ff\times \ee$ (which implies $\PP_X\ll \ff$ and $\PP_Y\ll\ee$); denote $f_{12}:=\frac{\dd \PP_{(X,Y)}}{\dd(\ff\times \ee)}$ and $f_2:=\frac{\dd \PP_Y}{\dd \ee}$. Then $\PP$-a.s. 
			\begin{equation}
\PP [h (X,Y)\vert \sigma(Y)]=c( Y),
			\end{equation}  where $$c(y):=\int h(x,y)\frac{f_{12}(x,y)}{f_2(y)}\ff(\dd x)\mathbbm{1}_{\{f_2>0\}}(y)\text{ for }y\in E.$$ Furthermore, it is the case that $c\in \EE/\mathcal{B}_{[-\infty,\infty]}$.
			\item\label{cond.a} (Independent maps.) Let $\GG$ be a sub-$\sigma$-field of $\AA$ such that $Y\in \GG/\mathcal{E}$ and $X$ is independent of $\GG$ under $\PP$. Then $\PP$-a.s. 			
			\begin{equation}
			\PP [h (X,Y)\vert \GG]=d(Y),
			\end{equation} where $$d(y):=\PP[h (X,y)]\text{ for } y\in E.$$ Furthermore, it is the case that $d\in \EE/\mathcal{B}_{[-\infty,\infty]}$. 
		\end{enumerate}
	\end{enumerate}
\end{proposition}
\begin{remarks}
(i) An important special case of \eqref{cond.b} is when  $(E,\EE,\ee)=(F,\FF,\ff)=(\mathbb{R},\mathcal{B}_\mathbb{R},\leb)$, in which case $f_2=f_Y$ and $f_{12}=f_{(X,Y)}$ are the densities that we have already introduced before.   (ii) Another case of \eqref{cond.b} of some interest is when $F$ is countable, $\FF=2^F$ and $\ff$ is (equivalent to) counting measure, in which case $\PP_{(X,Y)}\ll \ff\times \ee$ is equivalent to $\PP_Y\ll \ee$. When so, then $f_{12}(x,y)=\frac{\dd \PP(Y\in \cdot\vert X=x)}{\dd \ee}(y)\frac{\PP(X=x)}{\ff(\{x\})}$ a.e.-$\ee$ in $y\in E$ for all $x\in \tilde{F}:=\{f\in F:\PP(X=f)>0\}$ and $f_{12}(x,\cdot)=0$ a.e.-$\ee$ for $x\in F\backslash \tilde{F}$, accordingly $f_2(y)=\sum_{x\in\tilde{F}}f_{12}(x,y)\ff(\{x\})=\sum_{x\in\tilde{F}}\frac{\dd \PP(Y\in \cdot\vert X=x)}{\dd \ee}(y)\PP(X=x)$ a.e.-$\ee$ in $y\in E$, whence $$\frac{f_{12}(x,y)}{f_2(y)}=\frac{\frac{\dd \PP(Y\in \cdot\vert X=x)}{\dd \ee}(y)\frac{\PP(X=x)}{\ff(\{x\})}}{\sum_{f\in\tilde{F}}\frac{\dd \PP(Y\in \cdot\vert X=f)}{\dd \ee}(y)\PP(X=f)}\text{ for $\ee$-a.e. $y\in \{f_2>0\}$},\quad x\in \tilde{F},$$ which reminds us of Bayes' rule (it is indeed one of its many incarnations). Plugging it into the expression for $c$ we get $$c(y)=\frac{\sum_{x\in \tilde{F}}h(x,y)\frac{\dd \PP(Y\in \cdot\vert X=x)}{\dd \ee}(y)\PP(X=x)}{\sum_{x\in\tilde{F}}\frac{\dd \PP(Y\in \cdot\vert X=x)}{\dd \ee}(y)\PP(X=x)}\text{ for $\ee$-a.e. }y\in \{f_2>0\}$$($\therefore$ for $\PP_Y$-a.e. $y\in E$, of course); the intervention of $\ff$ has disappeared.  $\ee=\PP_Y$ is allowed, for $\PP_Y\ll\PP_Y$ trivially. (iii) Still with regard to \eqref{cond.b}: it may happen (but it is by no means automatic) that $\PP_{(X,Y)}\ll \PP_X\times \PP_Y$, in which case one can take  $\ee=\PP_Y$ and $\ff=\PP_X$, so that $f_2=1$ a.e.-$\ee$; though, it is not particularly useful unless one can actually compute $f_{12}=\frac{\dd\PP_{(X,Y)}}{\dd(\PP_X\times \PP_Y)}$. If $X$ and $Y$ are independent, then $\PP_{(X,Y)}=\PP_X\times\PP_Y$, and hence, by the preceding, a special case of \eqref{cond.a} (namely, when $\GG=\sigma(Y)$) is recovered from \eqref{cond.b}. 
\end{remarks}
\begin{proof} 
Considering positive and negative parts separately one reduces at once to the case when the conditional expectations are those of nonnegative functions. The claims \eqref{means:0} and \eqref{cond.b} are then a simple matter of checking the defining property of a conditional expectation once one recognizes: in \eqref{means:0}, that $\sigma_\Omega(\mathcal{I})$ consists precisely of all the unions of the members of $\mathcal{I}$; in \eqref{cond.b}, that $f_2(y)>0$ for $\PP_{(X,Y)}$-a.e. $(x,y)\in F\times E$, because $\PP_{(X,Y)}(f_2\circ \pr_2>0)=\PP(f_2(Y)>0)=\PP(Y\in \{f_2>0\})=\int_{\{f_2>0\}}f_2\dd\ee=\int f_2\dd \ee=1$. For \eqref{cond.a} by a $\pi$-$\lambda$--monotone class argument one reduces to the case when $h=\mathbbm{1}_{A\times B}$ for an $A\in \FF$ and a $B\in \EE$. Then one takes into account the basic properties of conditional expectations (taking out what is known, conditioning w.r.t. an independent $\sigma$-field). 
\end{proof}

\begin{definition}
If $(\Omega,\AA,\PP)$ is a probability space and $Z$ is a random element valued in a measurable space $(E,\EE)$, then we write $\PP[f\vert Z]:=\PP[f\vert \sigma(Z)]=\PP[f\vert Z^{-1}(\EE)]$ for $f\in\AA/\mathcal{B}_{[-\infty,\infty]}$ verifying $\PP[f^+]\land \PP[ f^-]<\infty$. 
\end{definition}
\begin{remarks}
(i) The notation should reference $\EE$ (it is hidden in $\sigma(Z)=\sigma^\EE(Z)$), but it does not. (ii)	Proposition~\ref{proposition:methods}\eqref{means:0} provides a means of computing the  conditional expectation  $\PP[f\vert Z]$ given the  $(E,\EE)$-valued random element $Z$ whenever $\EE$ is generated by a countable partition (use Proposition~\ref{proposition:maps}\eqref{maps:iv}), so in particular for a discrete random variable $Z$. (iii)   Condtioning w.r.t. a sub-$\sigma$-field $\BB$ is conditioning w.r.t. $\mathrm{id}_\Omega$ viewed as a random element of $(\Omega,\BB)$. Thus conditioning w.r.t. random elements is actually not less general than conditioning w.r.t. sub-$\sigma$-fields. 
\end{remarks}

\begin{proposition}\label{proposition:regular}
Let $(\Omega,\AA,\PP)$ be a probability space and let $Z$ be a random element valued in a measurable space $(E,\EE)$. For all $f\in \AA/\mathcal{B}_{[-\infty,\infty]}$ for which $\PP[f^+]\land \PP[f^-]<\infty$ there is a $\PP_Z$-a.s. unique   $g\in \EE/\mathcal{B}_{[-\infty,\infty]}$ such that $\PP[f\vert Z]=g(Z)$ a.s.-$\PP$.

Furthermore, let $X$ be any random variable under $\PP$.  Then there exists a map $\mu:E\times \mathcal{B}_\mathbb{R}\to [0,1]$ such that: (i) $\mu(\cdot,A)\in \EE/\mathcal{B}_{[0,1]}$ for each $A\in \mathcal{B}_\mathbb{R}$; (ii) $\mu(e,\cdot)$ is a probability on $\mathcal{B}_\mathbb{R}$ for each $e\in E$; and (iii) $\PP[f(X,Z)\vert Z]=\int f(x,Z)\mu(Z,\dd x)$ a.s.-$\PP$ for all $f\in (\mathcal{B}_\mathbb{R}\otimes \EE)/\mathcal{B}_{[-\infty,\infty]}$ for which $\PP[f(X,Z)^+]\land \PP[f(X,Z)^-]<\infty$. This $\mu$ is $Z_\star\PP$ unique in the sense that if $\mu'$ also satisfies the preceding (in fact  (iii) we need only for $f=\mathbbm{1}_{Q\times E}$, $Q$ running through some countable $\pi$-system that generates $\mathcal{B}_\mathbb{R}$; e.g. $\{(-\infty,r]:r\in \mathbb{Q}\}$ is such a system), then $\mu(\cdot,A)=\mu'(\cdot,A)$ for all $A\in \mathcal{B}_\mathbb{R}$ a.s.-$Z_\star\PP$. Finally, if $X$ takes values in a Borel set $W\in \mathcal{B}_\mathbb{R}$, then $\mathcal{B}_W$ may replace $\mathcal{B}_\mathbb{R}$ in the preceding.
 
\end{proposition}
\begin{proof}
Existence of $g$ is by the Doob-Dynkin lemma (Proposition~\ref{prop:doob-dynkin}); uniqueness is from  Proposition~\ref{proposition:uniquness-image-meausure}.

Let us prove existence of $\mu$ in the second part. For each $r\in \mathbb{Q}$ there is a $g_r\in \EE/\mathcal{B}_{[0,1]}$ such that $\PP(X\leq r\vert Z)=g_r(Z)$ a.s.-$\PP$. There is an exceptional set $E'\in \EE$ of $\PP_Z$ probability zero, such that for $e\in E\backslash E'$, $g_r(e)$ is $\uparrow$ in $r\in \mathbb{Q}$, $\lim_{r\to\infty} g_r(e)=1$ and $\lim_{r\to-\infty}g_r(e)=0$ (due to monotonicity, monotone and dominated convergence, in this order). Define $$m_e(x):=\mathbbm{1}_{E\backslash E'}(e)\left(\inf_{r\in \mathbb{Q}\cap (x,\infty)}g_r(e)\right)+\mathbbm{1}_{E'}(e)\mathbbm{1}_{[0,\infty)}(x),\quad (e,x)\in E\times \mathbb{R}.$$ For each $e\in E$, $m_e$ is a distribution function on $\mathbb{R}$; for each $x\in \mathbb{R}$, $m_\cdot(x)\in \EE/\mathcal{B}_{[0,1]}$. Then put $\mu(e,\cdot):=\dd m_e$, a probability measure on $\mathcal{B}_\mathbb{R}$  for each  $e\in E$, i.e. we have (i). For each $x\in \mathbb{R}$, $\mu(\cdot,(-\infty,x])=m_\cdot(x)\in \EE/\mathcal{B}_{[0,1]}$; by Dynkin's lemma $\mu(\cdot,A)\in \EE/\mathcal{B}_{[0,1]}$ for each $A\in \mathcal{B}_\mathbb{R}$, which is (ii). Finally, for each $x\in \mathbb{R}$, by dominated convergence, $\mu(Z,(-\infty,x])=\PP(X\leq x\vert Z)$ a.s.-$\PP$; by Dynkin and monotone class we get  $\PP[f(X,Z)\vert Z]=\int f(x,Z)\mu(Z,\dd x)$ a.s.-$\PP$ for all $f\in \mathcal{B}_\mathbb{R}\otimes \EE/\mathcal{B}_{[0,\infty]}$, the final extension to get (iii) being immediate. In particular, if $X$ takes values in $W\in \mathcal{B}_\mathbb{R}$, then we get $\mu(Z,W)=\PP[X\in W\vert Z]=1$ a.s.-$\PP$ and one can replace (for some fixed $w\in W$) $\mu(e,\cdot)$ with $\delta_w\vert_{\mathcal{B}_W}$ for $e$ from the $\PP_Z$-negligible set from $\EE$ on which this fails, restricting $\mu(e,\cdot)$ to $\mathcal{B}_W$ elsewhere. Properties (i)-(ii)-(iii) then remain unaffacted with this new version of $\mu$ on transposing  $\mathcal{B}_\mathbb{R}$  to  $\mathcal{B}_W$.

Let us prove $Z_\star\PP$-uniqueness of $\mu$. We have $\mu(Z,Q)=\PP(X\in Q\vert Z)=\mu'(Z,Q)$ a.s.-$\PP$ for all $Q$ belonging to some countable generating $\pi$-system for $\mathcal{B}_\mathbb{R}$ ($\therefore$ for all such $Q$ a.s.-$\PP$); by Dynkin it extends to all Borel subsets of $\mathbb{R}$ a.s.-$\PP$. Then use Proposition~\ref{proposition:uniquness-image-meausure} (again).
\end{proof}
\begin{remarks}
One calls the $\mu$ from the preceding proposition a regular conditional probability for $X$ given $Z$ (relative to $\EE$). 

\noindent If $\BB$ is a sub-$\sigma$-field of $\AA$ then $\BB=\sigma^\BB(\id_\Omega)$,\footnote{In this vein it is worth mentioning in passing that if $\BB$ is countably generated in the sense that $\BB=\sigma_\Omega(\CC)$ for a countable $\CC\subset 2^\Omega$, then there is even a random variable $X$ such that $\BB=\sigma^{\mathcal{B}_\mathbb{R}}(X)$, i.e. $\BB$ is generated by a ``nice'' random element (clearly it can in fact happen only if $\BB$ is countably generated, for $\mathcal{B}_\mathbb{R}$ is so).}  i.e. $\BB$ can be viewed as the generated $\sigma$-field for the identity taking values in $(\Omega,\BB)$ and we get (with $Z=\id_\Omega$) $\mu:\Omega\times \BB\to [0,1]$ such that $\PP[f(X)\vert \BB](\omega)=\int f(x)\mu(\omega,\dd x)$ for $\PP$-a.e. $\omega$ for all $f$ etc. 

\noindent  Another special case  is when $\Omega=\mathbb{R}\times E$, $\AA= \mathcal{B}_{\mathbb{R}}\otimes \EE$, $X=\pr_1$ and $Z=\pr_2$. Then $\PP[f\vert Z]=\int f(x,Z)\mu(Z,\dd x)$ a.s.-$\PP$, in particular $\PP[f]=\int \int f(x,z)\mu(z,\dd x)(Z_\star\PP)(\dd z)$ for all $f\in \AA/\mathcal{B}_{[0,\infty]}$, say; we have a disintegration of the joint law $\PP$ against the second marginal.

\noindent The final particular case of the second part of Proposition~\ref{proposition:regular} that we find worth highlighting is as follows. We ask that $(\Omega,\AA)=(\mathbb{R},\mathcal{B}_\mathbb{R})$, $X=\mathrm{id}_\Omega$ and that $(E\times E,\EE\otimes \EE)$ has a measurable diagonal. Taking in that case $f=\mathbbm{1}_{\{Z(\pr_1)=\pr_2\}}$ we see that for $Z_\star \PP$-a.e. $e$ the probability $\mu(e,\cdot)$ is carried by the ``fiber'' $\{Z=e\}$. Together with $\PP[f\vert Z]=\int f(\omega)\mu(Z,\dd \omega)$ holding true a.s.-$\PP$, in particular $\PP[f]=\int \int f(\omega)\mu(e,\dd\omega)(Z_\star\PP)(\dd e)$ for all $f\in \mathcal{B}_\mathbb{R}/\mathcal{B}_{[0,\infty]}$, we may say that in a sense $\mu(e,\cdot)=\PP(\cdot\vert Z=e)$ (for $\PP_Z$-a.e. $e$, though the conditional probability $\PP(\cdot\vert Z=e)$ is ill-defined unless $\PP(Z=e)>0$).

\noindent One can produce the regular conditional probability for random elements $X$ taking values in more general spaces other than  a Borel subset of $\mathbb{R}$ with its Borel $\sigma$-field, but not in complete generality. How so? In particular, why can't we just get the $\mu(\cdot,A)$, $A\in \mathcal{B}_\mathbb{R}$, by  taking arbitrary versions $\tilde{\mu}(\cdot,A)$ of the $g_A$ for which $\PP(X\in A\vert Z)=g_A(Z)$ a.s.-$\PP$? Well, you see, true, for given pairwise disjoint $A_k$, $k\in \mathbb{N}$, we would get $\tilde{\mu}(\cdot,\cup_{k\in \mathbb{N}}A_k)=\sum_{k\in\mathbb{N}}\tilde{\mu}(\cdot,A_k)$ a.s.-$\PP$ (by additivity, dominated convergence), however the exceptional set on which this fails may depend on the choice of the sequence $(A_k)_{k\in \mathbb{N}}$, of which there are uncountably many choices  to be made, while we want $\mu(\cdot,\cup_{k\in \mathbb{N}}A_k)=\sum_{k\in \mathbb{N}}\mu(\cdot,A_k)$ for all such sequences everywhere on $E$. For this reason it becomes a delicate matter how to choose the representatives of the $g_A$, $A\in \mathcal{B}_\mathbb{R}$, so that a ``regular'' $\mu$ results. In the case of the real line we are saved by the a.s. monotonicity of $\PP(X\leq r\vert Z)$ in $r\in \mathbb{Q}$.
\end{remarks}

\begin{example}
Let $U$ and $V$ be independent random variables under a probability $\PP$, $U$ having the law $\mathrm{N}(0,1):=\left(\mathbb{R}\ni x\mapsto \frac{1}{\sqrt{2\pi}}e^{-x^2/2}\right)\cdot \leb$. Then $\PP[\cos(UV)\vert V]=e^{-V^2/2}$ a.s.-$\PP$.
\end{example}

\begin{example}
Recall the notation of Example~\ref{example:cc}. $\leb_{[0,1]}[f\vert \sigma^{\mathrm{ccc}}_{[0,1]}]=\leb_{[0,1]}[f]$ a.s.-$\leb_{[0,1]}$ for all $f\in \mathcal{B}_{[0,1]}/\mathcal{B}_{[-\infty,\infty]}$ having $\leb_{[0,1]}[f^+]\land \leb_{[0,1]}[f^-]<\infty$.
\end{example}

\begin{proposition}[Conditional image-measure theorem]
Let $(\Omega,\FF,\PP)$ be a probability space, $X$ an $(E,\EE)$-valued random element, $f\in \EE/\mathcal{B}_{[-\infty,\infty]}$ with $\PP[f^+(X)]\land \PP[f^-(X)]<\infty$, finally let $\AA$ be a sub-$\sigma$-field of $\EE$. Then 
$$\PP[f(X)\vert X^{-1}(\AA)]=\left((X_\star\PP)[f\vert \AA]\right)( X)\text{ a.s.-}\PP.$$
\end{proposition}
\begin{proof}
Remark that  $\PP[f^+(X)]\land \PP[f^-(X)]<\infty$ implies (is even equivalent to) $(X_\star\PP)[f^+]\land (X_\star \PP)[f^-]<\infty$ (by the ``unconditional'' image-measure theorem); besides, $X^{-1}(\EE)\subset \FF$ and $\AA\subset \EE$ render $X^{-1}(\AA)\subset \FF$, i.e. $X^{-1}(\AA)$ is a sub-$\sigma$-field of $\FF$. Next, $\left((X_\star\PP)[f\vert \AA]\right)( X)\in \sigma^\AA(X)/\mathcal{B}_{[-\infty,\infty]}$ as a composition of the measurable maps $X\in \sigma^\AA(X)/\AA$ and $(X_\star\PP)[f\vert \AA]\in \AA/\mathcal{B}_{[-\infty,\infty]}$. Besides, $\PP[\left(\left((X_\star\PP)[f\vert \AA]\right)( X)\right)^\pm]=(X_\star\PP)[\left((X_\star\PP)[f\vert \AA]\right)^\pm]$ and one of these ($+$ or $-$) is finite. Finally, for all $A\in \AA$,
$\PP[f(X);X^{-1}(A)]=(X_\star\PP)[f;A]=(X_\star\PP)[(X_\star\PP)[f\mathbbm{1}_A\vert\AA]]=(X_\star\PP)[(X_\star\PP)[f\vert\AA];A]=\PP[\left((X_\star\PP)[f\vert \AA]\right)( X);X^{-1}(A)]$. By definition of $\PP[f(X)\vert X^{-1}(\AA)]$ we get the stipulated equality.
\end{proof}

\section{Kolmogorov's extension theorem}
\emph{projective/inverse limits of probabilities; extending a consistent family of laws on products}

\begin{theorem}[Projective/inverse limits]\label{thm:projective-limit}
Let $F$ be a set on which is defined a partial (reflexive, transitive, antisymmetric) order $\leq$, directed upwards in the sense that for arbitrary $\{\alpha,\beta\}\subset F$ there is $\gamma\in F$ such that $\gamma\geq\alpha$ and $\gamma\geq\beta$. For each $\alpha\in F$ let there be given a probability space $(\Omega_\alpha,\FF_\alpha,\mu_\alpha)$, $\Omega_\alpha$ being a Hausdorff\footnote{or, more generally, one in which every compact subset is closed} topological space, $\FF_\alpha$ including all the compact sets of $\Omega_\alpha$, $\mu_\alpha$ being inner regular w.r.t. the compact sets in the sense that $$\mu_\alpha(A)=\sup\{\mu(K):K\text{ a compact subset of }A\},\quad A\in \FF_\alpha.$$

Let maps $\pr_{\beta\downarrow \alpha}$ as $(\alpha,\beta)$ runs over $\leq$ verify the following properties:
\begin{enumerate}[(a)]
 \item  $\pr_{\alpha\downarrow \alpha}=\mathrm{id}_{\Omega_\alpha}$ for all $\alpha\in F$;
 \item $\pr_{\beta\downarrow \alpha}\circ \pr_{\gamma\downarrow \beta}=\pr_{\gamma\downarrow \alpha}$ for all $\alpha\leq\beta\leq \gamma$ from $F$;
 \item $\pr_{\beta\downarrow \alpha}:\Omega_\beta\to \Omega_\alpha$ is continuous, also measurable in the sense that $\pr_{\beta\downarrow \alpha}\in \FF_\beta/\FF_\alpha$ for all $\alpha\leq \beta$ from $F$;
 \item $\pr_{\beta\downarrow \alpha}$ maps onto $\Omega_\alpha$ for all $\alpha\leq \beta$ from $F$.
 \end{enumerate}
 Suppose furthermore the consistency condition
\begin{enumerate}[(e)]
\item  $(\pr_{\beta\downarrow \alpha})_\star \mu_\beta=\mu_\alpha\text{ for all $\alpha\leq \beta$ from $F$}$
\end{enumerate} holds true. 

Next, set $$\Omega:=\left\{\omega\in \prod_{\gamma\in F}\Omega_\gamma: \pr_{\beta\downarrow \alpha}\omega_\beta=\omega_\alpha\text{ for all }\alpha\leq \beta\text{ from }F\right\},$$ introduce $\pr_\alpha:=(\Omega\ni\omega\mapsto \omega_\alpha)$ for $ \alpha\in F$, and assume 
\begin{enumerate}[(f)]
\item $\pr_\alpha$ maps $\Omega$ onto $\Omega_\alpha$ for each $\alpha\in F $. 
\end{enumerate}
Suppose finally that 
 \begin{enumerate}[(g)]
 \item for each non-empty countable $F'\subset F$ that is not bounded above (no $\gamma\in F$ such that $\alpha\leq\gamma$ for all $\alpha\in F'$) admitting an enumeration in increasing order (there is a $\uparrow$ map from $\mathbb{N}$ onto $F'$): if $\omega'\in \prod_{\alpha\in F'}\Omega_\alpha$ is such that $\pr_{\beta\downarrow \alpha}\omega_\beta'=\omega'_\alpha\text{ for all }\alpha\leq \beta\text{ from }F'$, then there is $\omega\in \Omega$ such that $\omega'=\omega\vert_{F'}$.
 \end{enumerate}
Then there exists a unique probability $\mu$ on $(\Omega,\lor_{\alpha\in F}(\pr_\alpha)^{-1}(\FF_\alpha))$ satisfying ${\pr_\alpha}_\star\mu=\mu_\alpha$ for all $\alpha\in F$.
\end{theorem}
\begin{proof}
For $\alpha\in F$ put $\FF'_\alpha:=(\pr_\alpha)^{-1}(\FF_\alpha)$ and define the probability $\mu'_\alpha:\FF_\alpha'\to [0,1]$ by insisting that $\mu_\alpha=(\pr_\alpha)_\star \mu'_\alpha$ (the definition is without ambiguity because for each $E'\in \FF'_\alpha$ there is a unique $E\in \FF_\alpha$ such that $E'=(\pr_\alpha)^{-1}(E)$ /$\because$ $\pr_\alpha$ maps onto $\Omega_\alpha$/). As $$\pr_\alpha=\pr_{\beta\downarrow \alpha}\circ \pr_\beta$$ we deduce that $$\FF'_\alpha\subset \FF'_\beta,$$ moreover $$\mu'_\beta\vert_{\FF'_\alpha}=\mu'_\alpha$$ for all $\alpha\leq \beta$ from $F$.  Thus $\mu':=\cup_{\alpha\in F}\mu'_\alpha$ is a well-defined vanishing-at-$\emptyset$ and finitely additive map on the algebra $\FF':=\cup_{\alpha\in F}\FF_\alpha'$, which generates $\lor_{\alpha\in F}(\pr_\alpha)^{-1}(\FF_\alpha)$ on $\Omega$.  In particular it follows at once from Proposition~\ref{proposition:equality of measures} that $\mu$, which must extend $\mu'$, is unique, if it exist.

In order to conclude the proof, by Carath\'eodory's extension theorem~\ref{theorem:caratheodory}  it remains to show  that $\mu'$ is countably additive. To that end let now $(E_n')_{n\in \mathbb{N}}$ be any $\downarrow$ sequence in $\FF'$ and $\epsilon\in (0,\infty)$ be such that $\mu'(E_n') \geq \epsilon$ for all $n\in \mathbb{N}$; it will suffice to check that then $\cap_{n\in \mathbb{N}}E_n'$ is not empty. 

Now, for each $n\in \mathbb{N}$ there is $\alpha_n\in F$ such that $E_n'\in \FF'_{\alpha_n}$; we may and do ask that $(\alpha_m)_{m\in \mathbb{N}}$ is $\uparrow$ relative to the partial order $\leq$ ($\because$ $\leq$ is upwards directed). If $F':=\{\alpha_n':n\in \mathbb{N}\}$ is bounded above by a $\gamma\in F$, then trivially $\mu_\gamma'(\cap_{n\in \mathbb{N}}E_n')\geq \epsilon$, a fortiori $\cap_{n\in \mathbb{N}}E_n'\ne \emptyset$. Henceforth $F'$ is not bounded above.

Proceeding onwards, $E_n'=\pr_\alpha^{-1}(E_n)$ for a unique $E_n\in \FF_{\alpha_n}$, namely $E_n=\pr_{\alpha_n}(E_n')$ and $$E_n\supset \pr_{\alpha_m\downarrow \alpha_n}(E_m)$$ for all $n\leq m$ from $\mathbb{N}$. Further, using the inner regularity w.r.t. the compact sets, for each $n\in \mathbb{N}$ there is a compact $D_n\subset E_n$ for which $$\mu_{\alpha_n}(E_n\backslash D_n)\leq\frac{\epsilon}{2^{n+1}};$$  then set $$C_n:=\cap_{k=1}^n\pr_{\alpha_n\downarrow \alpha_k}^{-1}(D_k),$$ which is a compact subset of $D_n$. We observe that $C_m\subset (\pr_{\alpha_m\downarrow \alpha_n})^{-1}(C_n)$ for all $n\leq m$ from $\mathbb{N}$. Also, for each $n\in \mathbb{N}$, $$E_n\backslash C_n\subset \cup_{k=1}^n(\pr_{\alpha_n\downarrow \alpha_k})^{-1}(E_k\backslash D_k),$$ hence $$\mu_{\alpha_n}(E_n\backslash C_n)\leq \frac{\epsilon}{2}$$ and therefore $$\mu_{\alpha_n}(C_n)\geq \frac{\epsilon}{2},$$ in particular $$C_n\ne \emptyset.$$ 

Now, $$(\pr_{\alpha_n\downarrow \alpha_1}(C_n))_{n\in \mathbb{N}}$$ is a $\downarrow$ sequence of non-empty compact subsets of $\Omega_{\alpha_1}$, its intersection is therefore also non-empty; pick arbitrary $\omega_1'$ therefrom. Then $$\left((\pr_{\alpha_2\downarrow \alpha_1})^{-1}(\{\omega_1'\})\cap \pr_{\alpha_n\downarrow \alpha_2}(C_n)\right)_{n\in \mathbb{N}_{\geq 2}}$$ is a $\downarrow$ sequence of non-empty compact subsets of $\Omega_{\alpha_2}$, its intersection being therefore again non-empty; pick arbitrary $\omega_2'$ therefrom. Proceeding inductively we deduce existence of $\omega_n'\in C_n$ for each $n\in \mathbb{N}$, such that $\pr_{\alpha_{n+1}\downarrow \alpha_n}(\omega_{n+1}')=\omega_n'$ for all $n\in \mathbb{N}$. From the last hypothesis of the theorem we avail ourselves finally of an $\omega\in \Omega$ having $\pr_{\alpha_n}(\omega)=\omega_n'$ for all $n\in \mathbb{N}$, which entails $\omega\in \cap_{n\in \mathbb{N}}E_n'$. The proof is now complete.
 \end{proof}

\begin{example}
Set $t_0:=\{\emptyset\}$ and then inductively $t_{n+1}:=t_n\times \{0,1\}$ for $n\in \mathbb{N}_0$. For $n\in \mathbb{N}_0$ let $T_n:=\cup_{k\in [n]\cup \{0\}}t_k$ be the rooted binary tree up to level $n$ and let $T_\infty:=\cup_{n\in \mathbb{N}_0}t_n$ be the infinite rooted binary tree\footnote{We trust here  the student of these notes can guess at the obvious rooted tree structures.}. Then, for each $n\in \mathbb{N}_0$ we have a unique probability $\PP_n$ on $(\Omega_n,\FF_n):=(\{-1,1\}^{T_n},2^{\{-1,1\}^{T_n}})$ satisfying the following with $\xi_v^n$, $v\in T_n$, the canonical projections on $\Omega_n$: (1) the $\xi_v^n$, $v\in t_n$, are independent and $\PP(\xi_v^n=1)=\frac{1}{2}$ for all $v\in t_n$;  also (2) $\xi_w^n=\xi_a^n\xi_b^n$ a.s.-$\PP$ with $a$ and $b$ the two immediate descendants of $w$, for all $w\in T_{n-1}$. Denote finally  by $\xi_v$, $v\in T_\infty$, the canonical projections on $\Omega_\infty:=\{-1,1\}^{T_\infty}$. Then, by Theorem~\ref{thm:projective-limit}, we deduce at once the existence of a unique probability $\PP$ on $(\Omega_\infty,\lor_{v\in T_\infty}\sigma(\xi_v))$ such that : (1) for all $n\in \mathbb{N}_0$ the $\xi_v$, $v\in t_n$, are independent; (2) $\PP(\xi_v=1)=\frac{1}{2}$ for all $v\in T_\infty$; (3) $\xi_w=\xi_a\xi_b$ a.s.-$\PP$ with $a$ and $b$ the two immediate descendants of $w$, for all $w\in T_\infty$.
\end{example}

\begin{corollary}
Let $\Lambda$ be an arbitrary set, let $X$ be a locally compact second countable Hausdorff space (e.g. the real line, or more generally any Euclidean space) and let there be given for each $F\in (2^\Lambda)_{\mathrm{fin}}$ a probability $\mu_F$ on $(X^F,{\mathcal{B}_X}^{\otimes F})$, the family $(\mu_F)_{F\in (2^\Lambda)_{\mathrm{fin}}}$ being consistent in the sense that ${\pr_{G\downarrow F}}_\star\mu_G=\mu_F$ for all $G\supset F$ from $(2^\Lambda)_{\mathrm{fin}}$, $\pr_{G\downarrow F}:=(X^G\ni x\mapsto x\vert_F\in X^F)$ being the canonical projection from $X^G$ to $X^F$. Then there exists a unique probability $\mu$ on $(X^\Lambda,{\mathcal{B}_X}^{\otimes \Lambda})$ such that ${\pr_F}_\star\mu=\mu_F$ for all $F\in (2^\Lambda)_{\mathrm{fin}}$, $\pr_F:=(X^\Lambda\ni x\mapsto x\vert_F\in X^F)$ being the canonical projection from $X^\Lambda$ to $X^F$.
\end{corollary}
\begin{proof}
We have only to apply Theorem~\ref{thm:projective-limit} and the fact that, for each $n\in \mathbb{N}_0$,  $\BB_{X^n}={\BB_X}^{\otimes [n]}$ (by second countability), while  any probability $\nu$ on $(X^n,\mathcal{B}_{X^n})$ is automatically inner regular, in fact regular: the class of sets $$\{A\in \mathcal{B}_{X^n}:\forall\epsilon>0\, \exists \text{ compact $K\subset X^n$ and open $U\subset X^n$ such that $K\subset A\subset U$ and $\mu(U\backslash K)<\epsilon$}\}$$ is a $\sigma$-algebra on $X^n$ containing all the open sets of $X^n$ (which is not so difficult to check, because the properties of Hausdorfness, local compactness and second countability are preserved under finite products, they are also inherited by open sets, and together imply $\sigma$-compactness); therefore it is just the class of all Borel sets in $X^n$.
\end{proof}

\vfill
\begin{center}
He puzzled and puzzled till his puzzler was sore. (How the Grinch Stole Christmas.)
\end{center}

\end{document}